\documentclass[12pt,reqno]{amsart}
\usepackage{amsmath,amsfonts,amssymb,amscd,amsthm,amsbsy, color}
\usepackage{graphicx}
\usepackage{comment}
\usepackage[normalem]{ulem}

\DeclareMathOperator{\arcsinh}{arcsinh}
\DeclareMathOperator{\arccot}{arccot}

\theoremstyle{plain} 
\newtheorem{thm}{Theorem}[section]
\newtheorem{cor}[thm]{Corollary} 
\newtheorem{lemma}[thm]{Lemma} 
\newtheorem{prop}[thm]{Proposition}
\newtheorem{defn}[thm]{Definition}
\newtheorem{thmv}[thm]{Theorem}

\newtheorem{rem}[thm]{Remark}
\theoremstyle{definition}

\newtheorem{esempio}[thm]{Example}
\newtheorem{notaz}[thm]{Notation}
\numberwithin{equation}{section}

\newcommand\E{\mathcal{E}}
\newcommand\om{\omega^2}
\newcommand\R{\mathbb{R}}

\newcommand\Eh{\E+h}

\newcommand{\x}[1]{\xi_{#1}}

\newcommand\bx{\bar{\xi}}

\newcommand{\stkout}[1]{\ifmmode\text{\sout{\ensuremath{#1}}}\else\sout{#1}\fi}
\newcommand\Prob{HS}

\title{Refraction Periodic Trajectories in Central Mass Galaxies}
\author{Irene De Blasi and Susanna Terracini}

\address{Dipartimento di Matematica ``G. Peano''
	\newline\indent
	Universit\`a degli Studi di Torino
	\newline\indent
	 Via Carlo Alberto 10, 10123 Torino, Italy\\}
\email{irene.deblasi@unito.it}
\email{susanna.terracini@unito.it}

\date{\today} %%  this cancels date in article format


\begin{document}
	\maketitle
	
	%%%%%%%%%%%%%%%%%%%%%%%%%
%	\baselineskip=18pt              %% DRAFT MODE -- double spaced.
%	\tableofcontents
	
\begin{abstract}
We consider a new type of dynamical systems of physical interest, where  two different forces act in two complementary regions of the space, namely a Keplerian attractive center sits in the inner region, while an harmonic oscillator is acting in the outer one. In addition, the two regions are separated by an interface $\Sigma$, where a Snell's law of ray refraction holds. Trajectories concatenate arcs of Keplerian hyperbol\ae~ with harmonic ellipses, with a refraction at the boundary. When the interface also has a radial symmetry, then the system is integrable, and we are interested in the effect of the geometry of the interface  on the stability and bifurcation of periodic orbits from the homotetic collision-ejection ones. We give local condition on the geometry of the interface for the stability and obtain a complete picture of stability and bifurcations in the elliptic case for period one and period two orbits.
\end{abstract}

\section{Introduction and description of the model}\label{sec: model}

According with Bertrand's theorem, among all  central forces with bounded trajectories, there are only two cases with the property that all orbits are also periodic: the attractive inverse-square gravitational force and the linear elastic restoring one governed by Hooke's law.

In this paper we consider a new type of dynamical systems of physical interest, where such two forces act in two complementary regions of the space; a Keplerian attractive center sits in the inner region, while an harmonic oscillator is acting in the outer one. In addition, the two regions are separated by an interface $\Sigma$, where a Snell's law of ray refraction holds. Hence trajectories concatenate arcs of Keplerian hyperbol\ae~ with harmonic ellipses, with a refraction at the boundary. When the interface also has a radial symmetry, then the system is integrable, and we are interested in the effect of symmetry breaking  on the stability and bifurcation of periodic orbits. A subsequent paper \cite{IreneSusNew} will be devoted to the analysis, in terms of KAM and Mather theories, of systems with close to circular interfaces.

Our first motivation comes from an elliptical galaxy model with a central core, of interest in Celestial Mechanics \cite{Delis20152448}, which deals with the dynamics of a point-mass particle $P$ moving in a galaxy with an harmonic biaxial core, in whose center there is a Black Hole. As known, Black Holes appear when, caused by gravitation collapse,  the mass densities of celestial bodies exceeds some critical value, and  act as  attractors of both matter and light. Following the relativistic equivalence between energy and matter, the critical  behaviour in the presence of Black Holes has been the recent object of investigation related with optical properties of metamaterials \cite{Genov2009687}. In this framework, light behaves in space as in an optical medium having an effective refraction index which incorporates the gravitational field and may have a discontinuity accounting for the inhomogeneity of the  material itself. Therefore, our type of systems are of interests in view of possible applications in engineering artificial optical devices that control, slow and trap light in a controlled manner \cite{Krishnamoorthy2012205}. The systems we consider are somehow reminiscent of billiards, but with two fundamental differences: first of all the rays are curved by the gravitational force;  moreover, reflection is replaced by refraction combined with excursion in the outer region. There is a vast bibliography on Birkhoff billiards, with recent relevant advances (see the book \cite{MR2168892} and papers \cite{MR3868417,MR3815464,MR3788206,MR4085879}).  Recently, some cases of composite billiard  with reflections and refractions have been studied \cite{Baryakhtar2006292}, also in the case of a periodic inhomogenous lattice \cite{Glendinning2016}. Finally, a problem with some stronger analogy with ours is that of inverse magnetic billiards, where the trajectories of a charged particle in this
setting are straight lines concatenated with circular arcs of a given Larmor radius \cite{GasiorekThesis,gasiorek2019dynamics}. Let us add that, compared with the cases mentioned, an additional difficulty is that the corresponding return map is not globally well defined.

Going back to our model in Celestial Mechanics,  supposing the axes of the galaxy's mass distribution being orthogonal, we can then use a planar reference frame whose $x$ and $y$ axes are the galaxy's ones, while the BH is at its origin. In the described reference frame, we denote with $z\in\mathbb{R}^2$ the particle's coordinates. In the model studied in the present work, the plane $\mathbb{R}^2$ is divided into two regions, according to whether the gravitational effects of the galaxy's mass distribution or of the BH dominate. The BH's domain of influence is set to be a generic regular domain $\boldsymbol{0}\in D\subset\mathbb{R}^2$, and the particle moves on the plane under the influence of inner and external potentials 
\begin{equation}\label{potential}
	V(z)=
	\begin{cases}
		V_I(z)=\mathcal{E}+h+\frac{\mu}{\vert z\vert} \quad &\text{if }z\in D\\
		V_E(z)=\mathcal{E}-\frac{\omega^2}{2}\vert z\vert^2 &\text{if } z\notin D, 
	\end{cases}
\end{equation}   
whith \textcolor{black}{$\mathcal{E},\mu, \omega>0$ and $\Eh>0$},  while the behaviour of the particle's trajectory while it reaches the boundary $\partial D=\Sigma$ is ruled by a generalization of Snell's law (see Section \ref{subs:snell}). 
The motion of $P$ will take place inside the \emph{Hill's region} 
\begin{equation*}
	\mathcal{H}=\{z\in\mathbb{R}^2\quad \vert\quad V_E(z)\geq0\};
\end{equation*}
for computational reasons, we impose $2\mathcal{E}>\omega^2$ to ensure that the circle of radius $1$ $S^1$ is contained in $\mathcal{H}$. 
\begin{figure}[!h]
	\centering
	\includegraphics[height=.2\textheight]{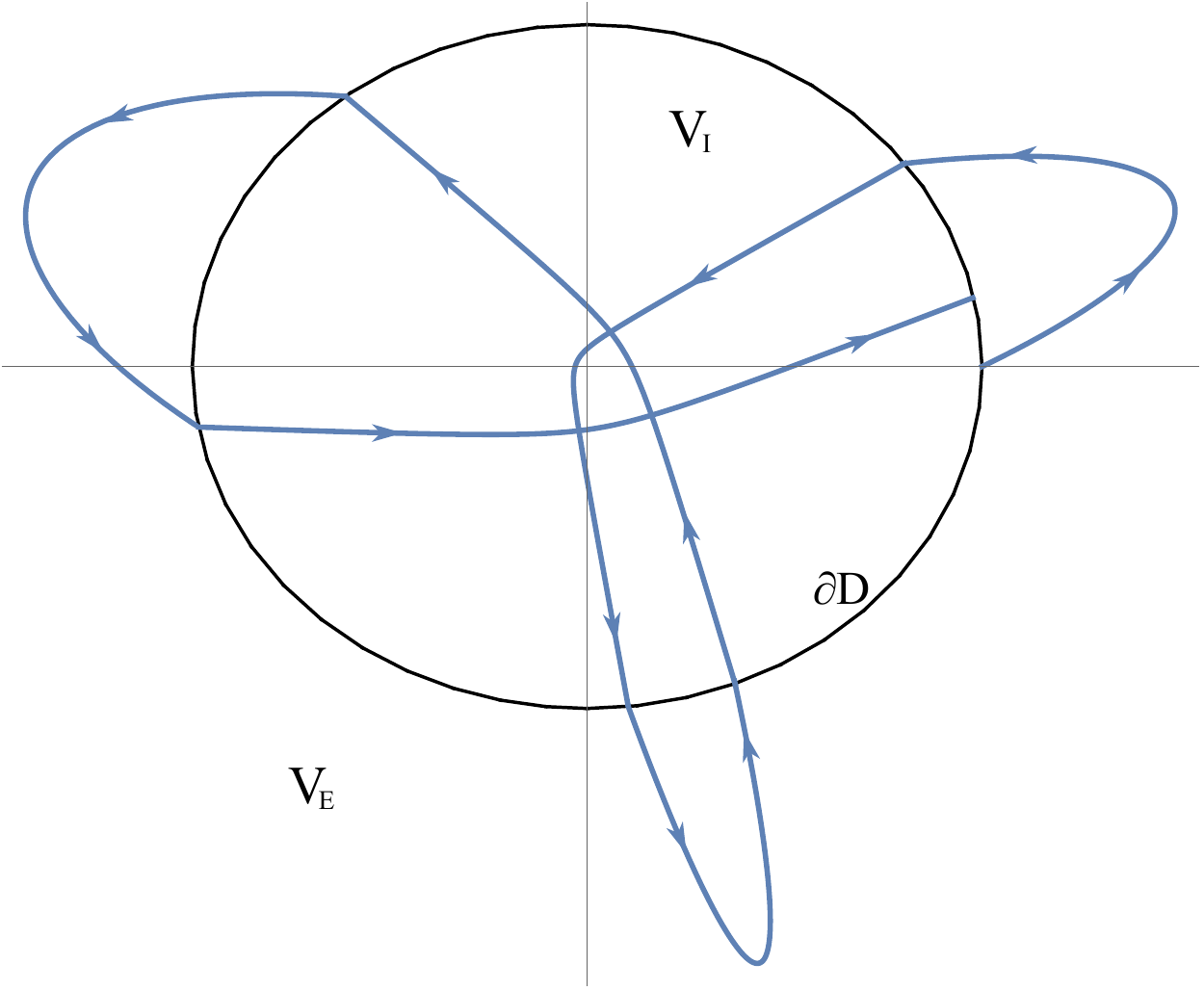}
	\hglue2cm
	\includegraphics[height=.2\textheight]{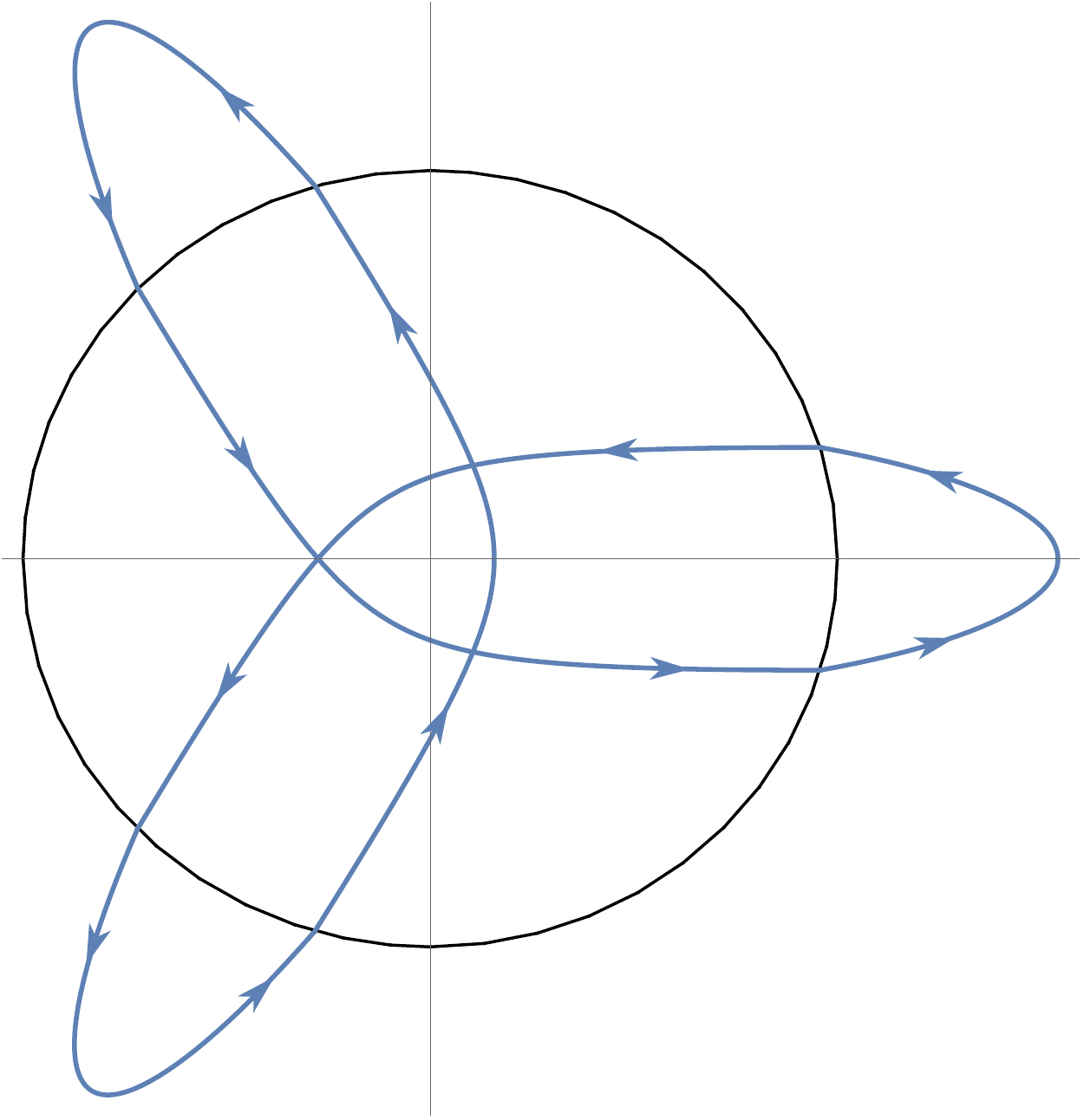}
	\caption{Left: trajectory for the general case. The inner and outer arcs are connected by a refraction Snell's law. Right: a period three orbit for an elliptic domain with eccentricity $e=0.3$ and physical parameters $\E=2.5,$ $\omega=\sqrt{2},$ $h=0.1$ and $\mu=1$. }
	\label{fig:intro-generale}
\end{figure}

The aim of this work is to study the trajectories of zero energy of the system whose potential is defined as in (\ref{potential}), in relation with the geometry of the boundary $\partial D$, taking $\mathcal E,h,\mu$ and $\omega$ as parameters. Though usually $\partial D$ will be elliptical shaped, most of our results just involve its geometrical features, namely its tangent and curvature. More precisely we \emph{will \textcolor{black}{usually} assume that $\partial D$ intersects orthogonally both coordinate axes}. In this way, there will be two collision homotetic periodic solutions in the horizontal and vertical directions. Taking advantage of Levi-Civita regularisation (\cite{Levi-Civita}), we may indeed assume motions to be extended after a collision with the gravity center by complete reflection.
We are concerned with the stability (dynamical and structural) of such periodic trajectories and their bifurcations in dependence of the system parameters. We shall focus in particular on bifurcations of free-fall period-two trajectories. In the case of the ellipse we will be able to describe the full picture, in dependence of the physical parameters.

Although the presence of periodic orbits depends in general on the global geometry of $D$, as well as on the physical parameters $\E, h, \mu, \omega$, there is a class of them whose existence is guaranteed by particular local conditions on $\partial D$: this is the case with the \emph{homotetic orbits}, which keep a fixed configuration during motion. Let us suppose that $\partial D$ is a curve of class $C^2$, and take $\mathbf{p}\in\partial D$, satisfying the two conditions 
\begin{equation}\label{homcond}
	\begin{aligned}
		&(i)\quad&&\mathbf{p}\perp\partial D, \\
		&(ii)\quad&&\text{defined }s=t\mathbf{p},\text{ } t\in[0,\infty),\text{ } supp(s)\cap\partial D=\{\mathbf{p}\}. 
	\end{aligned}
\end{equation}
In this case, the system admits a collision homotetic orbit in the direction of $\mathbf{p}$, which we denote by $\bar{z}_p(t)$. While condition $(i)$ is necessary to assure that the orbit is not deflected when crossing the interface $\partial D$ by Snell's law, condition $(ii)$, along with the regularity of the boundary, guarantees that the rays starting from the origin intersect it only once. As a consequence, conditions (\ref{homcond}) not only imply the existence of the homotetic orbits, but also the existence and uniqueness of inner and outer arcs in some neighbourhoods, as well as the good definition of the refraction law in its vicinity. Under the hypotheses (\ref{homcond}), it makes then sense to study the linear stability of $\bar{z}_p$ under the regularised flow:  indeed,  the following Theorem provides a full characterisation of stability in terms of the physical parameters and the local properties of $\partial D$ in $\mathbf{p}$. 
\begin{thmv}\label{intro thm 1}
	Let us suppose $\partial D=supp(\gamma)$, with $\gamma\in C^2([0,2\pi])$. Let $\mathbf{p}=\gamma(\bx)\in\partial D$ satisfying (\ref{homcond}), denote $k(\bar{\xi})$ the curvature at $\gamma(\bar\xi)$, and denote with $\bar{z}_p$ the homotetic orbit in the direction of $\mathbf{p}$. Let the inner and outer potentials be defined as in \eqref{potential}. Then, we have

\begin{itemize}
	\item if $\Delta(\bx)>0$, then $\bar{z}_p$ is linearly unstable; 
	\item if $\Delta(\bx)<0$, then $\bar{z}_p$ is linearly stable,
\end{itemize}
where we have denoted:

	\begin{equation*}
		\begin{aligned}
			&\Delta(\bx)(\E,h,\mu,\omega;\gamma)=A B C D\\
			&A=\frac{16}{\mathcal{E}^2\mu^2}\left(\sqrt{V_I(\gamma(\bar{\xi})}-\sqrt{V_E(\gamma(\bar{\xi}))}\right)\left(|\gamma(\bx)|k(\bar{\xi})-1\right), \\
			&B=\mathcal{E}-\left(|\gamma(\bx)|k(\bar{\xi})-1\right)\left(\sqrt{V_I(\gamma(\bar{\xi}))}-\sqrt{V_E(\gamma(\bar{\xi}))}\right)\sqrt{V_E(\gamma(\bar{\xi}))},\\ 
			&C=-\mu\sqrt{V_E(\gamma({\bar{\xi}}))}+2|\gamma(	\bx)|B\sqrt{V_I(\gamma(\bx))}, \\
			&D=\mu+2|\gamma(\bx)|\left(|\gamma(\bx)|k(\bar{\xi})-1\right)\sqrt{V_I(\gamma(\bar{\xi}))}\left(\sqrt{V_I(\gamma(\bar{\xi}))}-\sqrt{V_E(\gamma(\bar{\xi}))}\right).
		\end{aligned}
	\end{equation*}
\end{thmv}
When $\partial D$ is an ellipse, the stability of the two homotetic orbits, which are parallel to the coordinate axes, can be studied explicitely in terms of the physical parameters of the problem and the eccentricity $0\leq e<1$ of the ellipse. By symmetry, only the homotetic orbits intersecting the positive directions of the axes, which we denote with $\bar{z}_0$ and $\bar{z}_{\pi/2}$, are considered. 
\begin{cor}
	If $\partial D$ is an ellipse, explicit expressions for $\Delta(0)$ and $\Delta(\pi/2)$ are provided in \eqref{delta01}, leading to the complete description, in terms of the physical parameters and of the eccentricity, of all stability regimes. \end{cor}

	As the expressions of the $\Delta(0)$ and $\Delta(\pi/2)$, though explicit, include the many different parameters in a rather intricated formula,  the  general study of their sign can be ardous and we shall perform it numerically in general and analytically in some specific regimes; indeed, an asymptotic analysis for $e\to0$ can be done, leading to a rather simple stability criterion for small eccentricities. In particular, we have the following result for small eccentricities:
	
	\begin{cor}
If $\frac{\sqrt{\E+h+\mu}}{\mu}<\frac{\sqrt{2\E-\om}}{2\sqrt{2}\E}$, then, for small eccentricities, 
 $\bar{z}_0$ is stable and $\bar{z}_{\pi/2}$ is unstable.  Symmetrically, if $\frac{\sqrt{\E+h+\mu}}{\mu}>\frac{\sqrt{2\E-\om}}{2\sqrt{2}\E}$,  then, for small eccentricities. 
$\bar{z}_0$ is unstable and $\bar{z}_{\pi/2}$ is stable.  
\end{cor}
	
A similar asymptotic analysis, which holds for arbitrary eccentricities, can be performed for high values of $h$ or $\mu$ and $\E$ (see Proposition \ref{eccpiccola}): fixing all the parameters but $h$ (resp. $\mu$), if $h$ (resp. $\mu$) is large enough, both $\bar{z}_0$ and $\bar{z}_{\pi/2}$ are unstable homotetic orbits. In such cases, with the additional hypothesis of a good definition of the dynamics on the whole ellipse, we can infer the existence of an intermediate non-homotetic  stable periodic orbit with exactly two distinct crossings of $\partial D$. Furthermore, if $\E$ is large enough, one has that $\bar{z}_{\pi/2}$ is unstable, while the stability of $\bar{z_{0}}$ is determined by the value of $\mu$, in the sense that there is a theshold value $\bar{\mu}(\omega, h, e)$ such that if $\mu<\bar{\mu}$ $\bar{z}_{0}$ is unstable, while if $\mu>\bar{\mu}$ its stability is reversed. 

The differentiable dependence of the stabilty with respect to the physical parameters leads naturally to bifurcation phenomena whenever a variation of any of $\E, h, \mu$ or $\omega$ determines a change in the sign of $\Delta$: Section \ref{ssec: numerics} provides concrete examples of such transitions. \\
 The second class of periodic orbits on which this work is focused is represented by the free-fall two-periodic orbits, where two homotetic outer arcs are connected by an inner Keplerian hyperbola: Theorem \ref{intro thm 2} provides a sufficient condition for their existence in the elliptic case. 
 \begin{thmv}\label{intro thm 2}
 	Suppose that $\partial D$ is an ellipse with eccentricity $e\in(0,1/\sqrt{2})$, and consider $\E>0,\omega>0$ such that $2\E>\om$. Then there are $\bar{\mu}=\bar{\mu}(\E,\omega, e)$ and $\bar{h}=\bar{h}(\E,\omega,e,\mu)$ such that, if $\mu>\bar{\mu}$ and $h>\bar{h}$, then the dynamics admits at least two nontrivial free-fall brake orbits of period two. 
 \end{thmv}
 
 Nontrivial here stands for non homotetic. The existence of this type of periodic orbits on the ellipse is a significant fact, which distinguishes the stricly elliptic case, namely, with $e\neq0$ from the circular case: we have indeed that, while in the latter there are infinitely many homotetic orbits, there is no possibility to have a nontrivial free-fall two periodic trajectory. \\
 As in the case of Theorem \ref{intro thm 1}, also Theorem \ref{intro thm 2} admits an extension for general curves which share with the ellipse a common behaviour near to the homotetic orbits up to the second order and a particular type of global convexity property with respect to the hyperbol\ae. In particular, we shall define a class of boundaries $\gamma$ for which the inner arcs are globally well defined. 
 \begin{defn}
 	We say that the domain $D$ is \textbf{convex for hyperbol\ae for fixed } $\boldsymbol{h}, \boldsymbol{\E}$\textbf{ and} $\boldsymbol{\mu}$ if every Keplerian hyperbola with energy $\Eh$ and central mass $\mu$ intersects $\partial D$ at most in two points. \\
 	The domain $D$ is \textbf{convex for hyperbol\ae} if the previous condition holds for every positive $\E, h$ and $\mu$. 
 \end{defn}
%\begin{oss}

Let us observe that Theorem \ref{intro thm 2} remains true when $D$ is convex for hperbol\ae, \textcolor{black}{is everywhere transverse with the radial direction}  and $\partial D=supp(\gamma)$, $\gamma\in C^2([a,b])$ and $\gamma(\xi)=\gamma_0(\xi)+\gamma_1(\xi)$, where $\gamma_0$ parametrises the ellipse and $\gamma_1$ is such that has the same symmetry of the ellipse and 
	\begin{equation*}
		\gamma_1(k\pi/2)=\dot{\gamma}_1(k\pi/2)=\ddot{\gamma}_1(k\pi/2)=(0,0)
	\end{equation*} 
for $k=0,1,2,3$. \\
%\end{oss}
In the case of the ellipse, our analytical study is enriched by a numerical investigation, presented in Section \ref{ssec: numerics}, where the behaviour of the dynamics in different cases of interest is described. Of special interest is the evidence, in particular conditions, of diffusive \textcolor{black}{orbits}, even for very small eccentricities (i. e., near to the circular case, which is integrable), which is a strong sign of chaotic behaviour. \\
This work is organized as follows: \S \ref{sec:variational} recalls the basics for the variational approach to the problem and states Snell's law. In \S \ref{sec:existence} we analyse existence of outer an inner arcs, while \S \ref{sec:first_return} is devoted to the construction of a local first return map, close to homotetic ejection-collision trajectories. The stability of such orbits is the object of  \S 
\ref{sec:general_stab} for general domains. In \S \ref{sec:alliptic_domains} we deepen and complete the elliptic case, with an special emphasis on the existence of period one and period two (brake) orbits. \S \ref{ssec: numerics} presents many numerical results for the elliptic case.

\section{Preliminaries and notations}\label{sec:variational}
 
Most of the analytical techniques used to investigate the dynamics described in \S \ref{sec: model} rely on the variational structure. Consider a fixed-ends problem of the type 
\begin{equation}\label{general prob}
	\begin{cases}
		z''(s)=\nabla V(z(s))\quad&s\in[0,T]\\
		\frac{1}{2}|z'(s)|-V(z(s))=0&s\in[0,T]\\
		z(0)=z_0, z(T)=z_1
	\end{cases}
\end{equation}
with $z_0,z_1$ in a suitable subset of $\R^2$, and $V(z)$ a generic potential such that $V(z)>0$ a.e. As we will take advantage of regularisation techniques to deal with the inner potential, we can suppose that $V(z)$ is regular, as well as the solution of (\ref{general prob}). \\
In particular, the variational approach will be crucial for the analysis of the linearized problem around the homotetic solutions and for the determination of a suitable refraction law which could connect inner and outer arcs. \\
This Section is devoted to recollect the main definitions we shall use, along with the principal results. An extensive discussion of the topic and the following results can be found in Appendix \ref{sec:appA} and in \cite{ST2012}. 

\subsection{Maupertuis functional and Jacobi length}

\begin{defn}
	Given $z_0, z_1\in \R^2$, we denote with $M([0,1],z(t))$ the Maupertuis functional
	
	\begin{equation*}
		M(z)=M([0,1],z(t))=\int_0^1|\dot{z}(t)|^2V(z(t))dt
	\end{equation*}
which is defined on the set
	\begin{equation*}
		H_{z_0,z_1}=\{z(t)\in H([0,1],\R^2)\text{ }|\text{ }z(0)=z_0, z(1)=z_1\}.  
	\end{equation*}
Furthermore, the Jacobi length is given by 
\begin{equation*}
	L(z)=L([0,1],z(t))=  \int_0^1|\dot{z}(t)|\sqrt{V(z(t))}dt, 
\end{equation*}
and is defined on the closure of $H_0=\{z(t)\in H_{z_0,z_1}\text{ }|\text{ }\forall t\in[0,1] \text{ }|\dot{z}(t)|>0, V((z(t)))>0 \}$ in the weak topology of $H^1([0,1],\R^2)$. 

\end{defn}
If $\alpha(t)\in H_0$, $L(\alpha)$ represents the length of $\gamma$ in the Jacobi metric defined by $g_{ij}=V(x)\delta_{ij}$, while the functional $M$ is differentiable in $H^1([0,1],\R^2)$, and, as Proposition \ref{prop Mau} shows, its critical points at positive levels are reparametrisations of solutions of (\ref{general prob}). \\
The relation between $L$ and $M$ is given as follows: if $z(t)$ is a solution of a suitable reparametrised problem of (\ref{general prob}), given explicitely in Proposition \ref{prop Mau}, then $L^2(z)=2M(z)$, and then finding critical points of $M(z)$ at a positive level is equivalent to finding critical points of $L(z)$ (see Remark \ref{L^2=2M}).\\
By comparing problem (\ref{general prob}) with the definition of $H_{z_0,z_1}$, it is clear that in order to apply the properties of the Jacobi length to find the solutions $z(s)$ of (\ref{general prob}), one needs a suitable reparametrisation $t=t(s)$ such that $t(0)=0$ and $t(T)=1$. 
\begin{defn}
	We define: 
	\begin{itemize}
		\item the \textbf{geodesic time} $t\in[0,1]$ as the time parameter to be used in $M(z(t))$ and $L(z(t))$; 
		\item the \textbf{cinetic time} $s\in [0,T]$ (later on, without loss of generality, we will consider $s\in[-T,T]$) as the physical time parameter through which $z(s)=z(t(s))$ solves (\ref{general prob}).  
	\end{itemize}
We denote with $\dot{ }=d/dt$ and $'=d/ds$ respectively the derivatives with respect to the geodesic and the cinetic time.

\end{defn}  
The relation between the geodesic and the cinetic time is proved in Remark \ref{oss tempi} and is given by 
\begin{equation*}
	\frac{d}{dt}=\frac{L}{\sqrt{2}V(z(t(s)))}\frac{d}{ds}, 
\end{equation*}
where $L=|\dot{z}(t)|\sqrt{V(z(t))}$ is constant along the solutions of (\ref{general prob}). 

\subsection{Generalized Snell's law}\label{subs:snell}
One of the crucial steps of the broken geodesics method goes by the determination of a junction law between the different arcs. In our case, a local minimization argument in a neighborhood of the interface between the inner and outer region, represented by the domain's boundary $\partial D$, demonstrates the validity of a refraction law which turns out to be a generalization for curved interfaces and geodesics of the classical Snell's law. \\
In particular, let us consider $z_E(t)$ the solution of problem (\ref{general prob}) with $V=V_E$ and fixed ends $z_0^{(E)}$ and $z_1^{(E)}$, and $z_I(t)$ the analogous for $V=V_I$ and fixed ends $z_0^{(I)}$ and $z_1^{(I)}$, with $z_1^{(E)}=z_0^{(I)}=\bar{z}$, where both the arcs are parametrised by the geodesic time.
Then, denoting with $e$ the unit vector tangent to $\partial D$ in $\bar{z}$, one has (see Section \ref{ssec: appA2} for details) 
\begin{equation}\label{legge snell}
	\sqrt{V_E(\bar{z})}\frac{\dot{z_E}(1)}{|\dot{z_E}(1)|}\cdot e=\sqrt{V_I(\bar{z})}\frac{\dot{z_I}(0)}{|\dot{z_I}(0)|}\cdot e. 
\end{equation}
The refraction rule governing the transition from the inner to the outer region of $D$ is completely analogous. \\
The geometric interpretation of (\ref{legge snell}) can be found as follows: denoting with $\alpha_E$ and $\alpha_I$ respectively the angles of $\dot{z}_E(1)$ and $\dot{z}_I(0)$ (or, equivalently, $\dot{z}_I(1)$ and $\dot{z}_E(0)$) with the external normal unit vector to $\partial D$ in $\bar{z}$, we have 
\begin{equation}\label{snell law def}
	\sqrt{V_I({\bar{z}})}\sin{\alpha_I}=\sqrt{V_E({\bar{z}})}\sin{\alpha_E}.
\end{equation}
Relation (\ref{snell law def}) can be rephrased as the conservation of the tangent component of the velocity vector through the interface. 
\begin{rem}\label{rifl totale}
	As $\forall\bar{z}\in\mathbb{R}^2$ $V_E(\bar{z})<V_I(\bar{z})$,  the equation 
	\begin{equation*}
		\alpha_I=\arcsin{\left(\sqrt{\frac{V_E(\bar{z})}{V_I(\bar{z})}}\sin{\alpha_E}\right)}
	\end{equation*}
is always solvable in the domain $[-\frac{\pi}{2},\frac{\pi}{2}]$. Viceversa, the crossing from the interior to the exterior of the domain may encounter an obstruction: indeed, the equation 
	\begin{equation*}
		\alpha_E=\arcsin{\left(\sqrt{\frac{V_I(\bar{z})}{V_E(\bar{z})}}\sin{\alpha_I}\right)}
	\end{equation*}
	admits a solution if and only if $\Big\vert\sqrt{\frac{V_I(\bar{z})}{V_E(\bar{z})}}\sin{\alpha_I}\Big\vert\leq1$: in order to guarantee 
the solvability, we  define a \textit{critical angle}, depending on $\bar{z}$ and, as parameters, on the proper quantities of the problem $\mathcal{E}, h, \mu, \omega$, that is
	\begin{equation*}
		\alpha_{I,crit}=\arcsin{\left(\sqrt{\frac{V_E(\bar{z})}{V_I(\bar{z})}}\right)}.
	\end{equation*}
In this way, the passage from the inside to the outside of the domain $D$ takes place as long as  $\alpha_I\in[-\alpha_{I,crit},\alpha_{I,crit}]$.   
\textcolor{black}{We stress that, for $|\alpha_I|=\alpha_{I,crit}$, the refracted outer arc turns out to be tangent to $\partial D$. }
\end{rem}

%In other words, recalling the definition of $\alpha_I$, if the inner orbit is \textit{transverse enough} to the boundary $\partial D$, we can guarantee its exit from $D$. Section \ref{ssec: localInner} provides some conditions which assures, under suitable assumptions on $\partial D$ and the initial conditions, the transversality of the inner orbit.  

\section{Local existence of inner and outer arcs: a transversality approach}\label{sec:existence}

%\subsection{Local existence of inner and outer orbits: a transversality approach}

This section is devoted to state existence of outer and inner arcs, close to the homotetic ones, which will be next used to apply a broken geodesics technique. We shall use a classical transversality  approach, reminiscent to the one in \cite{ST2012}, to which we refer for a more detailed exposition. It is worthwhile stressing that, when dealing with the inner dynamics, we shall take advantage of Levi-Civita regularising transformation. \\
For the sake of brevity, in the present Section only the main results, whose application is crucial for the construction of a suitable first return map, see Section \ref{sec:first_return}, are presented; a more extensive discussion is postponed to Appendix \ref{appB}. \\
From now on, we will always suppose that $D$ is contained in the Hill's region 
\begin{equation*}
	\mathcal{H}=\left\{p\in\R^2\text{ }\Bigg|\text{ }\sqrt{\E-\frac{\om}{2}|p|^2}>0\right\}, 
\end{equation*}
and assume that its boundary $\partial D$ is parametrised as the trace of a curve $\gamma: I= [0,2\pi]\mapsto\mathbb{R}^2 $, with $\gamma\in C^2$. We focus on the points of $\gamma$ which satisfy a local  transversality property, as well as a local star-convexity, namely:
\begin{equation}\label{cond Hom}
	\begin{aligned}
	\bx\in I \text{ such that}:\text{ } &(i)\text{ }\gamma{(\bx)}\nparallel\dot{\gamma}{(\bx)}\\
	 &(ii)\text{ }\text{defined }s=t\gamma{(\bx)},\text{ } t\in[0,\infty),\text{ } supp(s)\cap\partial D=\{\gamma{(\bx)}\}. 
	\end{aligned}
\end{equation}
As our main interest lies in the local study of the trajectories around the homotetic solutions, \textcolor{black}{whose directions are in a subset of the ones identified by (\ref{cond Hom})}, we restrict our analysis to a neighborhood of $\bx$: condition (\ref{cond Hom}), along with the regularity of $\gamma$, assures indeed the existence of an open interval $I'\subset I$ such that $\bx\in I'$ and
\begin{equation}\label{local trans}
\forall \xi\in I'\text{ }	\gamma(\xi)\nparallel\dot{\gamma}{(\xi)}.
\end{equation} 
 Furthermore, possibly taking a smaller $I'$, we can suppose that condition (\ref{cond Hom}(ii)) holds for every $\xi\in I'$. The local transversality property of $supp(\gamma(I'))$ with respect to the radial directions and its star-convexity with reference to the origin will be the main ingredients to guarantee the existence of the inner and outer arcs in a neighborhood of an homotetic solution.

\begin{thmv}\label{thm esistenza ext}
	Suppose that the domain's boundary $\partial D$ is a regular  curve parametrised by $\gamma:I\rightarrow\mathbb{R}^2$ and suppose that $\bx\in I$ satisfies (\ref{cond Hom}). Then  there are $\epsilon_\alpha(\bar {\xi})>0$ and $\epsilon_{\x0}(\bar{\xi})>0$ such that for every $\x0\in I$, $\alpha\in[-\pi/2, \pi/2]$ with $\vert\bx-\x0\vert<\epsilon_{\x0}(\bx)$ and  $\vert\alpha\vert<\epsilon_\alpha(\bar{\xi})$, there exist $T>0, \xi_1\in I$ such that the problem (in complex notation $\gamma(\bar{\xi})=\vert\gamma(\bar{\xi})\vert e^{i\bar{\theta}}$)
	\begin{equation*}
		\begin{cases} 
			y''(s)=-\omega^2 y(s)\\
			\frac{1}{2}\vert y'(s)\vert^2-\mathcal{E}+\frac{\omega^2}{2}\vert y(s)\vert^2=0 \\ 
			y(0)=\gamma(\x0), y'(0)=v_0 e^{i(\bar{\theta}+\alpha)}, 
		\end{cases}
	\end{equation*}
	with $v_0=\sqrt{2\mathcal{E}-\omega^2\vert\gamma(\x0)\vert^2}$, admits the unique solution $y(s; \x0, \alpha)$ and $y(T; \bar{\xi}, \alpha)=\gamma(\xi_1)\in\partial D$.\\
	Moreover, for every $s\in(0,T)$ $y(s;\x0,\alpha)\notin\bar{D}$. 
\end{thmv}

In the case of the inner arcs, we turn to the problem 
\begin{equation}\label{inprob}
	\begin{cases}
		z''(s)=-\frac{\mu}{|z(s)|^3}z(s), \quad &s\in[0,S], \\
		\frac{1}{2}|z'(s)|^2-\mathcal{E}-h-\frac{\mu}{|z(s)|}=0, &s\in[0,S],\\
		z(0)=z_0, z'(0)=\mathbf{v}_0; 
	\end{cases}
\end{equation}
for some $S>0$ and some initial conditions $z_0\in\partial D$ and $\mathbf{v}_0$ pointing inward the domain $D$, and denote with $z(s;z_0,\mathbf{v}_0)$ its solution (with an abuse of notation, in the following the initial velocity will be defined either by its angle with the radial direction or its orthogonal component to the latter). As the Keplerian \textcolor{black}{orbits} with positive energy are unbounded and $D$ is bounded, for every initial condition for which the arc enters in $D$ there is $\tilde{S}>0$ such that the it encounters $\partial D$ again in a point which we call $z_1$. We search for constraints for $z_0,\mathbf{v}_0$ such that $\dot{z}(\tilde{S}; z_0, \mathbf{v}_0)$ is transverse to $\partial D$. \\
The singularity at the origin of the inner potential can be treated by means of the Levi-Civita regularisation technique (see \cite{Levi-Civita}), which consists in a change both in the temporal parameter and the spatial coordinates, in order to remove the singularity of Kepler-type potentials. In particular, the following Proposition, whose proof is discussed in Appendix \ref{appB}, holds.  
\begin{prop}\label{lem levi civita}
	Problem (\ref{inprob}) is conjugated, via a suitable set of transformations called the \textit{ Levi-Civita transformations}, to the problem 
	\begin{equation}\label{LCprob}
		\begin{cases}
			w''(\tau)=\Omega^2w(\tau), \quad &\tau\in[0,T], \\
			\frac{1}{2}|w'(\tau)|^2-E-\frac{\Omega^2}{2}|w(\tau)|^2=0, &\tau\in[0,T],\\
			w(0)=w_0, w'(0)=\dot{w}_0
		\end{cases}
	\end{equation}
	for suitable $w_0, \dot{w}_0, T$, $\Omega=\sqrt{2\mathcal{E}+2h}$, $E=\mu$. 
\end{prop}

We will refer to the time variable $\tau$ as the Levi-Civita time, and to the new reference system as the Levi-Civita plane. \\ 
By means of this regularisation, which is of indipendent interest and  will be used again in Section \ref{sec:general_stab}, one can infer the local existence and transversality of solutions of Problem (\ref{inprob}) in the vicinity of the homotetic brake orbits, in terms of both initial positions and velocities.

\begin{thmv}\label{thm esistenza int}
	Let us suppose that $\partial D$ is a closed curve of class $C^2$, $\boldsymbol{0}\in D$, parametrised by $\gamma(\xi):I\rightarrow \mathbb{R}^2$, and suppose that there is $\bx\in I$ such that condition \ref{cond Hom} is satisfied. Then there exist $\lambda_{\x0}(\bx)>0$ and  $0<\lambda_{\alpha}(\bar{\xi})<\pi/2$ such that for every $\x0\in [\bx-\lambda_{\x0}, \bx+\lambda_{\x0}]$ and  $\alpha\in[-\lambda_{\alpha},\lambda_{\alpha}]$ there are $T>0$ and $\xi_1\in I$ such that the problem 
	\begin{equation}\label{prob}
		\begin{cases}
			z''(s)=-\frac{\mu}{|z(s)|^3}z(s), \quad \frac{1}{2}|z'(s)|^2=\mathcal{E}+h+\frac{\mu}{|z(s)|}, \quad s\in[0,S]\\
			z(0)=\gamma(\x0)=\rho(\x0) e^{i\theta(\x0)}, z'(0)=\sqrt{2}\sqrt{\mathcal{E}+h+\frac{\mu}{\rho(\x0)}}e^{i(\theta(\bar{\xi})+\alpha)}
		\end{cases}
	\end{equation}
	admits the unique solution $z(s; \x0, \alpha)$. Moreover, $z(T;\x0,\alpha)=\gamma(\xi_1)\in\partial D$, and $z'(T;\x0,\alpha)\equiv  z'(T)$ is not tangent to $\partial D$. More precisely, there exists $\sigma>0$, depending on $\bar{\xi}$, such that, if we define $\beta$ such that $z'(T)=\sqrt{2}\sqrt{\mathcal{E}+h+\frac{\mu}{\rho(\xi_1)}}e^{i(\theta(\xi_1)+\beta)}$, we have $\beta\in[-\sigma,\sigma]$.
\end{thmv} 
\begin{rem}
	As in the case of the outer dynamics, one can assure that $z((0,T);\x0,\alpha)\in D$, namely, that $z(s; \x0,\alpha)$ does not intersect $\partial D$ for $s\in(0,T)$. This is in fact guaranteed by the validity of condition (\ref{cond Hom}(ii)), the continous dependence of problem (\ref{prob}) on the inital conditions and the fact that $0\in\dot{D}$. 
\end{rem}
\begin{notaz}\label{notazione} In the following, we will refer to the outer differential equation along with the energy conservation law with $(\Prob_E)$, while $(\Prob_I)$ and $(\Prob_{LC})$ will denote the inner differential problem respectively in the physical and in the Levi-Civita variables, with their own energy conservation conditions. More precisely, we will write $(\Prob_E)[z]$ if $z$ satisfies the outer differential equation with zero energy, and use the analogous notation for $(\Prob_I)$ and $(\Prob_{LC})$. 
\end{notaz}
\section{First return map}\label{sec:first_return}
\textcolor{black}{As in the previous Sections, }let us suppose that the boundary of the regular domain $D$ defined in Section \ref{sec: model} can be parametrised by a regular closed curve $\gamma:I\rightarrow\mathbb{R}^2$. Given some initial conditions $z_0^{(I)}, v_0^{(I)}, z_0^{(E)}$ and $v_0^{(E)}$, let us consider the solutions  $z_I(s)$ and $z_E(s)$ of the two systems 
\begin{equation}\label{EI problem}
	\begin{cases}
		(\Prob_I)[z(s)] &s\in[0,T_I]\\
		z_I(0)=z_0^{(I)}, z'_I(0)=v_0^{(I)}
	\end{cases}
	\begin{cases}
	(\Prob_E)[z(s)] &s\in[0,T_E]\\
		z_E(0)=z_0^{(E)}, z'_E(0)=v_0^{(E)}
	\end{cases}
\end{equation}
for some $T_I$, $T_E>0$.
Fixed $z_0\in\partial D$, $v_0\in\mathbb{R}^2$ such that it points towards the exterior of $D$, we want to describe (supposing that it exists) the trajectory obtained by the juxtaposition of an outer arc $z_E$ and the subsequent inner arc $z_I$, namely $z_{EI}(s)$ defined by
\begin{equation}\label{zEI def}
	z_{EI}(s)=
	\begin{cases}
		z_E(s)\quad &s\in[0,T_E)\\
		z_I(s) &s\in[T_E,T_E+T_I)\\
		z_E^{(1)}(s) &s=T_E+T_I,
	\end{cases}
\end{equation}
where the branches $z_E$, $z_I$ and $z_E^{(1)}$ are solution either of the outer or the inner problem and are connected by following the Snell's rule (see Section \ref{subs:snell}). In particular, we require $z_E(T_E)=z_I(T_E)$ and $z_I(T_E+T_I)=z_E^{(1)}(T_E+T_I)$, and, using the notation  
\begin{gather*}	
		z_E(T_E)=z_I(T_E)=z_1, \quad z_I(T_E+T_I)=z_E^{(1)}(T_E+T_I)=z_2\\
		v_1=\frac{ z'_E(T_E)}{| z'_E(T_E)|}	, \quad v_1'=\frac{ z'_I(T_E)}{| z'_I(T_E)|},\quad v_2=\frac{ z'_I(T_E+T_I)}{| z'_I(T_E+T_I)|}\quad v_2'=\frac{ z'_I(T_E+T_I)}{| {z'_E}^{(1)}(T_E+T_I)|}
\end{gather*}
we demand
\begin{equation}\label{extprob}
	\begin{cases}
		(\Prob_E)[z_E(s)] &s\in[0,T_E]\\
		z_E(s)\notin D, z_E(T_E)\in\partial D &s\in(0,T_E)\\
		z_E(0)=z_0, z'_E(0)=v_0 
	\end{cases}
\end{equation}
\begin{equation}\label{intprob}
	\begin{cases}
		(\Prob_I)[z_I(s)] &s\in[T_E,T_E+T_I]\\
		z_I(s)\in D, z_I(T_E+T_I)\in\partial D&s\in(T_E,T_E+T_I)\\
		\sqrt{V_E(z_1)}v_1\cdot e_1= 
		\sqrt{V_I(z_1)}v_1'\cdot e_1
	\end{cases}
\end{equation}
\begin{equation*}
	\begin{cases}
		(\Prob_E)[z_E^{(1)}(s)] &s\in[T_E+T_I,T_E+T_I+\tilde{T}]\\
		
		z_I(s)\notin D,  &s\in(T_E+T_I,T_E+T_I+\tilde{T}]\\
		\sqrt{V_I(z_2)}v_2\cdot e_2= 
		\sqrt{V_E(z_2)}v_2'\cdot e_2,  
	\end{cases}
\end{equation*}
for some $T_E,T_I,\tilde{T}>0$ and where $e_1$ and $e_2$ are the unit vectors tangent to $\partial D$ respectively in $z_1$ and $z_2$. \\

\subsection{Local  first return map}

We wish to construct the iteration map which expresses $(z_1,v_1)=(z_{EI}(T_E+T_I),\dot{z}_{EI}(T_E+T_I))$ as a function of $(z_0,v_0)$ in a suitable set of coordinates.\\
Let us suppose that the point $z_{EI}(0)=\gamma(\xi)\in\partial D$ is the starting point of the outer branch of $z_{EI}(t)$: then, denoting with $t(\xi)$ and $n(\xi)$ respectively the tangent and the outward-pointing normal unit vectors of $\gamma$ in $\xi$, the initial velocity $v$ can be expressed as $v=\sqrt{2V_E(\gamma(\xi))}(\cos{\alpha}~ n(\xi)+\sin{\alpha}~ t(\xi))$, where $\alpha\in[{-\pi\textcolor{black}{/2},\pi\textcolor{black}{/2}}]$ is the angle between $v$ and $n(\xi)$, positive if $v\cdot t(\xi)\geq0$ and negative otherwise. Then, once $\xi$ is fixed, the vector $v$ is completely determined by $\alpha$. We can then consider the map 
\begin{gather*}
	F:B\subset\left([0,2\pi]\times[{-\pi\textcolor{black}{/2},\pi\textcolor{black}{/2}}]\right)\rightarrow[0,2\pi]\times[{-\pi\textcolor{black}{/2},\pi\textcolor{black}{/2}}],\\ (\xi_0,\alpha_0)\mapsto(\xi_1,\alpha_1)=(\xi_1(\xi_0,\alpha_0),\alpha_1(\xi_0,\alpha_0)),
\end{gather*} 
where the pair $(\xi_1,\alpha_1)$ completely determines $(z_{EI}(T_E+T_I),z'_{EI}(T_E+T_I))$. \textcolor{black}{The determination of the domain of $F$, denoted with $B$, is a nontrivial problem, whose main issues are discussed in Remark \ref{oss buona definizione}. }

Although $F$ is not explicitely defined, taking together the properties of the solutions of Problem (\ref{EI problem}) and Snell's law (\ref{snell law def}), under some suitable hypotheses on $\partial D$, one can characterize one particular class of fixed points of $F$, deriving from one-periodic homotetic solutions of (\ref{EI problem}): 

\begin{rem}\label{oss fixed points}
	Initial conditions $z_{EI}(0)=\gamma(\bar{\xi})$, $z'_{EI}(0)=\sqrt{\rule{0pt}{2ex}2V_E(\bar{\xi})}\gamma(\bar{\xi})/|\gamma(\bar{\xi})|$ correspond to an homotetic solution of Problem (\ref{EI problem}) if and only if 
	\begin{equation}\label{fixedcond}
		\begin{split}
		&\gamma(\bar{\xi})\perp\dot{\gamma}(\bar{\xi})\\
		\text{and the segment }t\gamma(\bar{\xi}), \text{ }&t\in[0,\infty), \text{ does not intersect }\partial D\text{ for }t\neq1. 
		\end{split}
	\end{equation}

	Therefore, if condition (\ref{fixedcond}) holds, the pair $(\bar{\xi}, 0)$ is a fixed point for $F$, which we call \textit{homotetic}. 
\end{rem}
\begin{rem}\label{oss buona definizione}
	The conditions for $F$ to be globally defined on $[0,2\pi]\times[-\pi/2,\pi/2]$ are essentially two: 
	\begin{itemize}
		\item[(i)] the existence and uniqueness of the outer and inner arcs for any inital conditions; 
		\item[(ii)] the good definition of the refraction rule for every incoming arc: according to Remark \ref{rifl totale}, it is equivalent to require that for every inner arc, \textcolor{black}{if we denote with $\beta_1$ the angle between $z'_I(T_I+T_E)$ and the inward-pointing normal vector to $\partial D$ in $\gamma(\x1)$, we shall have } $|\beta_1|<\beta_{crit}=\arcsin(\sqrt{V_E(\gamma(\x1))/V_I(\gamma(\x1))})$.  
	\end{itemize}
Theorems \ref{thm esistenza ext} and \ref{thm esistenza int} provide sufficient conditions for $(i)$ to be satisfied, as well as proving the existence of a small neighborhood of the homotetic initial conditions for which the inner arc is arbitrarily transverse to $\partial D$. As a consequence, even though the global definition of the first return map $F$ can not be assured without additional requirement on $\gamma$, the hypotheses of the existence theorems guarantee that \textit{the map is locally well defined near to the homotetic solutions.}
\end{rem}
As we will see in some specific cases, there are particular conditions on which the homotetic solutions are not the only one-periodic solutions of Problem (\ref{EI problem}). On the other hand, the study of the stability of this particular class of points allows us to derive important informations on the behaviour of $F$.

\section{Stability analysis of the homotetic fixed points of $F$}\label{sec:general_stab}

\subsection{The Jacobian matix of $F$}\label{ssec: Jac 1}
Without loss of generality, let us assume that, in complex notation, $\bar{\xi}\in[0,2\pi]$ is such that $\gamma(\bar{\xi})=|\gamma(\bar{\xi})|e^{i\bar{\xi}}$ and $\dot{\gamma}(\bx)=|\dot{\gamma}(\bx)|ie^{i\bx}$. Then the point $\bar{p}=(\bx,0)$ is a fixed point for $F$, whose stability properties can be deduced from the spectral properties of the Jacobian matrix 
\begin{equation}\label{jac}
DF((\bar{\xi},0))=
\begin{pmatrix}
	\frac{\partial\xi_1}{\partial\xi_0}_{\vert_{\bar{p}}} & \frac{\partial\xi_1}{\partial\alpha_0}_{\vert_{\bar{p}}} \\ 
	\frac{\partial\alpha_1}{\partial\xi_0}_{\vert_{\bar{p}}} & \frac{\partial\alpha_1}{\partial\alpha_0}_{\vert_{\bar{p}}},
\end{pmatrix} 
\end{equation} 
which can be derived through the implicit function theorem, even though $F$ is not explicitely determined. \\
Let us consider a generic potential $V(z)$ and, once fixed $z_0,z_1\in\mathbb{R}^2$, consider the function $z(s)=z(s;z_0,z_1)$ which solves the fixed end problem (\ref{general prob}).
As already seen in Section \ref{subs:snell} and, in mor details, in Appendix \ref{sec:appA}, $z(t;z_0,z_1)=z(s(t);z_0,z_1)$ is a critical point for the Jacobi length  $L(y(t)$
with endpoints $y(0)=z_0$  $y(1)=z_1$. We denote $L(z(t;z_0,z_1))$ the value of this length. Note that neither the outer nor the inner arcs are global minimizers of the Jacobi length with fixed ends. Indeed, it can be proved (though is not relevant in this paper) that  the inner arc is a local minimizer while the outer one has Morse index one (cfr \cite{MR4061719,MR4058157}).
%\begin{equation}
%L(z(t;z_0,z_1))=\min\{L(y(t))\quad|\quad y(0)=z_0, y(1)=z_1\}, 
%\end{equation}
We recall that: 
\begin{itemize}
\item $t$ is the geodesic time, whose relation with the cinetic time $s$ is discussed in Section \ref{sec:variational} and Remark \ref{oss tempi};
\item $L(y(t))=\int_0^1\vert \dot{y}(t)\vert\sqrt{V(y(t))}dt$; 
\item $L=L(z(t;z_0,z_1))=\vert \dot{z}(t)\vert\sqrt{V(z(t))}=const$.
\end{itemize}
If we consider a generic unit vector $e$, recalling and generalizing (\ref{dirder}) the directional derivatives of $d(z_0,z_1)$ with respect to the first or second variable, denoted respectively with $v$ and $w$, can be written as
\begin{equation}\label{def der}
\begin{split}
	\partial_{e,v}d(z_0,z_1)=\nabla_{z_0}d(z_0,z_1)\cdot e, \quad &\partial_{e,w}d(z_0,z_1)=\nabla_{z_1}d(z_0,z_1)\cdot e\\
	\nabla_{z_0}d(z_0,z_1)=-\sqrt{V(z(0))}\frac{\dot{z}(0)}{\vert\dot{z}(0)\vert},\quad &\nabla_{z_1}d(z_0,z_1)=\sqrt{V(z(1))}\frac{\dot{z}(1)}{\vert\dot{z}(1)\vert}. 
\end{split}
\end{equation}
Let us now define the generating function 
\begin{equation*}
S(\xi_0,\xi_1)=d(\gamma(\xi_0),\gamma(\xi_1)),  
\end{equation*}
and define the tangent unit vectors $e_0=\dot{\gamma}(\x0)/|\dot{\gamma}(\x0)|$ and $e_1=\dot{\gamma}(\x1)/|\dot{\gamma}(\x1)|$. Hence, we have that the partial derivatives of $S$ with respect to $\x0$ and $\x1$ can be expressed as 
\begin{equation}\label{derS}
\begin{split}
	&\partial_{\xi_0}S(\xi_0,\xi_1)=\frac{d}{d\epsilon}d(\gamma(\xi_0+\epsilon \xi), \gamma(\xi_1))_{\vert_{\epsilon=0}}=\partial_{e_0,v}d(\gamma(\xi_0), \gamma(\xi_1))=\nabla_{z_0}d(\gamma({\xi_0}),\gamma(\xi_1))\cdot e_0, \\
	&\partial_{\xi_1}S(\xi_0,\xi_1)=\partial_{e_1,w}d(\gamma(\xi_0), \gamma(\xi_1))=\nabla_{z_1}d(\gamma({\xi_0}),\gamma(\xi_1))\cdot e_1. 
\end{split}
\end{equation}
Turning to the trajectory $z_{EI}(s)$ which describes a complete cycle exterior-interior, we can use the previous formulas to describe some geometric properties of the latter. Referring to (\ref{EI problem}) and further equations, define: 
\begin{itemize}
\item $\xi_0,\tilde{\xi},\xi_1\in[a,b]$ such that $\gamma({\xi_0})=z_0$, $\gamma({\tilde{\xi}})=z_1$, $\gamma({\xi_1})=z_2$; 
\item  $\alpha_0$ the angle between $v_0$ with $n(\xi_0)$; 
\item $\beta_0,\beta_1$ respectively the angles of $\dot{z}_E(T_E)$ and $\dot{z}_I(T_E)$ with $n(\tilde{\xi})$;
\item $\alpha_1',\alpha_1$ respectively the angles of $\dot{z}_I(T_E+T_I)$ and $\dot{z}_E^{(1)}(T_E+T_I)$ with $n(\xi_1)$. 
\end{itemize}
Then, from (\ref{def der}), (\ref{derS}) and (\ref{snell law def}), one finds the relations
\begin{equation}\label{6 eq}
\begin{split}
	&-\sqrt{V_E(\gamma({\xi_0}))}\sin{\alpha_0}=\partial_{\xi_0}S_E(\xi_0,\tilde{\xi}), \\
	&\sqrt{V_E(\gamma({\tilde{\xi}}))}\sin{\beta_0}=\partial_{\xi_1}S_E(\xi_0,\tilde{\xi}),\\
	&-\sqrt{V_I(\gamma({\tilde{\xi}}))}\sin{\beta_1}=\partial_{\xi_0}S_I(\tilde{\xi},\xi_1),\\
	&\sqrt{V_I(\gamma({\xi_1}))}\sin{\alpha_1'}=\partial_{\xi_1}S(\tilde{\xi},\xi_1),\\
	&\sqrt{V_E(\gamma({\tilde{\xi}}))}\sin{\alpha_0}=\sqrt{V_I(\gamma(\tilde{\xi}))}\sin{\beta_0}\\
	&\sqrt{V_I(\gamma({\xi_1}))}\sin{\alpha_1'}=\sqrt{V_E(\gamma(\xi_1))}\sin{\alpha_1},
\end{split}
\end{equation}
where $S_E$ and $S_I$ refer respectively to $d_E$ and $d_I$.
Removing $\beta_0, \beta_1$ and $\alpha_1'$ from (\ref{6 eq}), one obtains
\begin{equation}\label{3 eq}
\begin{split}
	&\partial_{\xi_0}S_E(\xi_0,\tilde{\xi})+\sqrt{V_E(\gamma(\xi_0))}\sin{\alpha_0}=0,\\
	&\partial_{\xi_1}S_E(\xi_0,\tilde{\xi})+\partial_{\xi_0}S_I(\tilde{\xi},\xi_1)=0,\\
	&\partial_{\xi_1}S_I(\tilde{\xi},\xi_1)-\sqrt{V_E(\gamma(\xi_1))}\sin{\alpha_1}=0.
\end{split}
\end{equation}
We can then define the function 
\begin{equation*}
\begin{split}
	&\Phi=\textcolor{black}{[\eta_\xi^-, \eta_\xi^+]\times[\eta_\alpha^-, \eta_\alpha^+]\times[\eta_\xi^-, \eta_\xi^+]\times[\eta_\xi^-, \eta_\xi^+]\times[\eta_\alpha^-, \eta_\alpha^+]}\rightarrow\mathbb{R}^3, \\
	&(\xi_0,\alpha_0,\tilde{\xi},\xi_1,\alpha_1)\mapsto
	\begin{pmatrix}
		\Phi_1(\xi_0,\alpha_0,\tilde{\xi},\xi_1,\alpha_1)\\
		\Phi_2(\xi_0,\alpha_0,\tilde{\xi},\xi_1,\alpha_1)\\
		\Phi_3(\xi_0,\alpha_0,\tilde{\xi},\xi_1,\alpha_1)
	\end{pmatrix}
	=
	\begin{pmatrix}
		\frac{\partial_{\xi_0}S_E(\xi_0,\tilde{\xi})}{\sqrt{V_E(\gamma(\xi_0))}}+\sin{\alpha_0}\\
		\partial_{\xi_1}S_E(\xi_0,\tilde{\xi})+\partial_{\xi_0}S_I(\tilde{\xi},\xi_1)\\
		\sin{\alpha_1}-\frac{\partial_{\xi_1}S_I(\tilde{\xi},\xi_1)}{\sqrt{V_E(\gamma(\xi_1))}}
	\end{pmatrix},
\end{split}
\end{equation*}
\textcolor{black}{where $[\eta_\xi^-,\eta_\xi^+]$ and $[\eta_\alpha^-,\eta_\alpha^+]$ are neighborhoods respectively of $\bx$ and $0$ such that the inner and outer dynamics are well defined (we remark that the existence of such neighborhoods is assured by Theorems \ref{thm esistenza ext} and \ref{thm esistenza int}). }\\
If $\xi_0, \tilde{\xi}$ and $\xi_1$ define respectively the inital, junction and final point of $z_{EI}(s)$, and $\alpha_0,\alpha_1$ are the angles of the initial and final velocity vectors of $z_{EI}(s)$ with the direction normal to $\partial D$ in the initial and final points, then, from (\ref{3 eq}), $\Phi((\xi_0,\alpha_0, \tilde{\xi}, \xi_1, \alpha_1))=0$. The point $\bar{q}$ which describes the homotetic solution defined in Remark \ref{oss fixed points}, which we call $\hat{z}_0(s)$,  is given by $\bar{q}=(\bar{\xi},0,\bar{\xi},\bar{\xi},0)$: clearly, $\Phi(\bar{q})=0$. \\
Under the hypothesis of nonsingularity of the matrix 
\begin{equation}\label{DyPhi}
D_{(\tilde{\xi},\xi_1,\alpha_1)}\Phi(\bar{q})=
\begin{pmatrix}
	\frac{\partial \Phi_1}{\partial\tilde{\xi}}_{\vert_{\bar{q}}}&\frac{\partial \Phi_1}{\partial\xi_1}_{\vert_{\bar{q}}}&\frac{\partial \Phi_1}{\partial\alpha_1}_{\vert_{\bar{q}}}\\
	\frac{\partial \Phi_2}{\partial\tilde{\xi}}_{\vert_{\bar{q}}}&\frac{\partial \Phi_2}{\partial\xi_1}_{\vert_{\bar{q}}}&\frac{\partial \Phi_2}{\partial\alpha_1}_{\vert_{\bar{q}}}\\
	\frac{\partial \Phi_3}{\partial\tilde{\xi}}_{\vert_{\bar{q}}}&\frac{\partial \Phi_3}{\partial\xi_1}_{\vert_{\bar{q}}}&\frac{\partial \Phi_3}{\partial\alpha_1}_{\vert_{\bar{q}}}
\end{pmatrix},
\end{equation}
which we will prove in Section \ref{ssec stab}, the implicit function theorem guarantees the existence of a function $\Psi:I_1\times J_1\rightarrow I_2\times I_3\times J_2$, $(\xi_0,\alpha_0)\mapsto(\tilde{\xi}(\xi_0,\alpha_0),\xi_1(\xi_0,\alpha_0),\alpha_1(\xi_0,\alpha_0))$, where $I_1,I_2,I_3$ and $J_1,J_2$ are suitable neighborhoods respectively of $\bar{\xi}$ and $0$, such that
\begin{equation*}
\forall (\xi_0,\alpha_0)\in I_1\times J_1\quad \Phi((\xi_0,\alpha_0,\Psi((\xi_0,\alpha_0))))=0.
\end{equation*} 
Moreover, defined 
\begin{equation*}
D_{(\xi_0,\alpha_0)}\Phi(\bar{q})=
\begin{pmatrix}
	\frac{\partial \Phi_1}{\partial\xi_0}_{\vert_{\bar{q}}}&\frac{\partial \Phi_1}{\partial\alpha_0}_{\vert_{\bar{q}}}\\
	\frac{\partial \Phi_2}{\partial\xi_0}_{\vert_{\bar{q}}}&\frac{\partial \Phi_2}{\partial\alpha_0}_{\vert_{\bar{q}}}\\
	\frac{\partial \Phi_3}{\partial\xi_0}_{\vert_{\bar{q}}}&\frac{\partial \Phi_3}{\partial\alpha_0}_{\vert_{\bar{q}}}
\end{pmatrix},\quad
D_{(\xi_0,\alpha_0)}\Psi(\bar p)=
\begin{pmatrix}
	\frac{\partial \tilde{\xi}}{\partial\xi_0}_{\vert_{(\bar{\xi},0)}}&\frac{\partial \tilde{\xi}}{\partial\alpha_0}_{\vert_{(\bar{\xi},0)}}\\
	\frac{\partial \xi_1}{\partial\xi_0}_{\vert_{(\bar{\xi},0)}}&\frac{\partial \xi_1}{\partial\alpha_0}_{\vert_{(\bar{\xi},0)}}\\
	\frac{\partial \alpha_1}{\partial\xi_0}_{\vert_{(\bar{\xi},0)}}&\frac{\partial \alpha_1}{\partial\alpha_0}_{\vert_{(\bar{\xi},0)}}
\end{pmatrix},
\end{equation*}
one has that 
\begin{equation*}
D_{(\xi_0,\alpha_0)}\Psi(\bar p)=-(D_{(\tilde{\xi},\xi_1,\alpha_1)}\Phi(\bar{q}))^{-1}D_{(\xi_0,\alpha_0)}\Phi(\bar{q}).
\end{equation*}
Recalling (\ref{jac}), we see that $DF(\bar p)$ is composed by the last two rows of $D_{(\xi_0,\alpha_0)}\Psi(\bar p)$. \\
To compute (\ref{DyPhi}), the second derivative of $S_E(\xi_0,\tilde{\xi})$ and $S_I(\tilde{\xi},\xi_1)$ computed in $\bar p$ are needed.

\subsection{Outer dynamics: computation of the derivatives of $\boldsymbol{S_E(\xi_0,\tilde{\xi})}$}\label{ssec: outer der}

Let us define $z_E^0(s)=z_E(s;\gamma(\bar{\xi}),\gamma(\bar{\xi}))$ the homotetic solution of problem 
(\ref{extprob}) (without loss of generality, suppose that it is defined in $[-T,T]$ for some $T>0$ to be determined). Recalling that, from the initial assumptions on $\bar{\xi}$, $\gamma({\bar{\xi}})=|\gamma(\bx)|e^{i\bx}$ and $\dot{\gamma}(\bar{\xi})=|\dot{\gamma}(\bx)|ie^{i\bx}$, we have that $z_0^E(s)=x_0^E(s)e^{i\bx}$, where $x_0^E(s):[-T,T]\rightarrow\mathbb{R}$ is a solution of the one-dimensional fixed-end problem
\begin{equation*}
\begin{cases}
	x_0^{E''}(s)=-\omega^2x_0^E(s), &s\in[-T,T]\\
	\frac{1}{2}\vert x_0^{E'}(s)\vert^2+\frac{\omega^2}{2}\vert x_0^E(s)\vert^2-\mathcal{E}=0, &s\in[-T,T]\\
	x_0^E(-T)=x_0^E(T)=|\gamma(\bx)|. 
\end{cases}
\end{equation*}
Then we have
\begin{equation}\label{omotetica esterna}
\begin{split}
	z_0^E(s)=\frac{\sqrt{2\mathcal{E}}}{\omega}\cos{(\omega s)}e^{i\bx}, \quad T=\frac{1}{\omega}\arccos{\left(\frac{\omega|\gamma(\bx)|}{\sqrt{2\mathcal{E}}}\right)},\\
	z_0^{E'}(-T)=-z_0^{E'}(T)=\sqrt{2\mathcal{E}-\omega^2|\gamma(\bx)|^2}e^{i\bx}; 
\end{split}
\end{equation}
taking into account (\ref{def der}) and the relations 
\begin{equation*}
\frac{d}{dt}=\frac{L}{\sqrt{2}V(z(t(s)))}\frac{d}{ds}\Rightarrow
\begin{cases}
	\sqrt{V_E(z_0^E(0))}\frac{\dot{z}(0)}{\vert\dot{z}(0)\vert}=\frac{1}{\sqrt{2}}z'(-T)\\
	\sqrt{V_E(z_0^E(1))}\frac{\dot{z}(1)}{\vert\dot{z}(1)\vert}=\frac{1}{\sqrt{2}}z'(T)
\end{cases}
\end{equation*}

one has
\begin{equation}\label{ext prime der}
\begin{split}
	&\partial_{\xi_0}S_E(\bar{\xi},\bar{\xi})=\nabla_{z_0}d_E(\gamma(\bar{\xi}),\gamma(\bar{\xi}))\cdot\dot{\gamma}(\bx)=-\frac{\sqrt{\rule{0pt}{2ex}2\E-\om|\gamma(\bx)|^2}}{\sqrt{2}}e^{i\bx}\cdot \dot{\gamma}(\bx)=0,\\
	&\partial_{\xi_1}S_E(\bar{\xi},\bar{\xi})=0. 
\end{split}
\end{equation}
As for the second derivatives, we have
\begin{equation}\label{second ext}
\begin{split}
	&\partial_{\xi_0}^2S_E(\bar{\xi},\bar{\xi})=\nabla_{z_0}^2d_E(\gamma({\bar{\xi}}),\gamma(\bar{\xi}))\dot{\gamma}(\bar{\xi})\cdot\dot{\gamma}(\bar{\xi})+\nabla_{z_0}d_E(\gamma(\bar{\xi}),\gamma(\bar{\xi}))\cdot\ddot{\gamma}(\bar{\xi})\\
	&\partial_{\xi_1}^2S_E(\bar{\xi},\bar{\xi})=\nabla_{z_1}^2d_E(\gamma({\bar{\xi}}),\gamma(\bar{\xi}))\dot{\gamma}(\bar{\xi})\cdot\dot{\gamma}(\bar{\xi})+\nabla_{z_1}d_E(\gamma(\bar{\xi}),\gamma(\bar{\xi}))\cdot\ddot{\gamma}(\bar{\xi})\\
	&\partial_{\xi_0,\xi_1}^2S_E(\bar{\xi},\bar{\xi})=\nabla_{z_0,z_1}^2d_E(\gamma({\bar{\xi}}),\gamma(\bar{\xi}))\dot{\gamma}(\bar{\xi})\cdot\dot{\gamma}(\bar{\xi})\\ &\partial_{\xi_1,\xi_0}^2S_E(\bar{\xi},\bar{\xi})=\nabla_{z_1,z_0}^2d_E(\gamma({\bar{\xi}}),\gamma(\bar{\xi}))\dot{\gamma}(\bar{\xi})\cdot\dot{\gamma}(\bar{\xi}), 
\end{split}
\end{equation} 
where, defining $\bar{e}=ie^{i\bx}$, 
\begin{equation}\label{secon ext 2}
\begin{split}
	\nabla_{z_0}^2d_E(\gamma(\bar{\xi}),\gamma(\bar{\xi}))\dot{\gamma}(\bar{\xi})&=|\dot{\gamma}(\bx)|\partial_{\bar{e},v}\left(\partial_{\bar{e},v}d_E\right)(\gamma({\bar{\xi})},\gamma(\bar{\xi}))=|\dot{\gamma}(\bx)|\partial_{\bar{e},v}\left(-\frac{1}{\sqrt{2}}z_0^{E'}(-T)\right)=\\
	&=-\frac{|\dot{\gamma}(\bx)|}{\sqrt{2}}\frac{d}{ds}\left(\partial_{\bar{e},v}z_0^E\right)(-T),
\end{split}
\end{equation}
and, similarly, 
\begin{equation}\label{second ext 3}
\begin{split}
	&\nabla_{z_1}^2d_E(\gamma(\bar{\xi}),\gamma(\bar{\xi}))\dot{\gamma}(\bar{\xi})=\frac{|\dot{\gamma}(\bx)|}{\sqrt{2}}\frac{d}{ds}\left(\partial_{\bar{e},w}z_0^E\right)(T), \\
	&\nabla_{z_0,z_1}^2d_E(\gamma(\bar{\xi}),\gamma(\bar{\xi}))\dot{\gamma}(\bar{\xi})=-\frac{|\dot{\gamma}(\bx)|}{\sqrt{2}}\frac{d}{ds}\left(\partial_{\bar{e},w}z_0^E\right)(-T),\\
	&\nabla_{z_1,z_0}^2d_E(\gamma(\bar{\xi}),\gamma(\bar{\xi}))\dot{\gamma}(\bar{\xi})=\frac{|\dot{\gamma}(\bx)|}{\sqrt{2}}\frac{d}{ds}\left(\partial_{\bar{e},v}z_0^E\right)(T). 
\end{split}
\end{equation}
The functions $\partial_{\bar{e},v}z_0^E(s)$ and $\partial_{\bar{e},w}z_0^E(s)$ are the first-order variations of $z_0^E(s)$ with respect to the variation respectively of its first and second endpoint along the unit vector $\bar{e}$, which is orthogonal to $z_0^E(s)$. If we define $\tilde{f}_0(s)=\partial_{\bar{e},v}z_0^E(s)$ and $\tilde{f}_1(s)=\partial_{\bar{e},w}z_0^E(s)$, we have that 
$\tilde{f}_0(s)=f_0(s)\bar{e}$ and $\tilde{f}_1(s)=f_1(s)\bar{e}$, with $f_0, f_1:[-T,T]\rightarrow\mathbb{R}$ to be determined. Consider $z(t)=z_0^E(t)+\tilde{f}_0(t)$ the geodesics obtained by varying the first endpoint of $z_0^E(t)$ in the direction of $\dot{\gamma}(\bar{\xi})$ expressed with respect to the geodesic time $t$: it solves the Euler-Lagrange equation with $\mathcal{L}=\vert\dot{z}(t)\vert^2V_E(z(t))$, namely, 
\begin{equation*}
\begin{split}
	0&=-\frac{d}{dt}\left(2\dot{z}(t)V(z(t))\right)+\vert\dot{z}(t)\vert^2\nabla V_E(z(t))=\\
	&=-\frac{d}{dt}\left(2(\dot{z}_0^E(t)+\dot{\tilde{f}}_0(t))V(z_0^E(t)+\tilde{f}_0(t))\right)+\vert\dot{z}_0^E(t)+\dot{\tilde{f}}_0(t)\vert^2\nabla V_E(z_0^E(t)+\tilde{f}_0(t))=\\
	&=-2\frac{d}{dt}\left(\dot{\tilde{f}}_0(t)V(z_0^E(t))\right)+\vert\dot{z}_0(t)\vert^2\nabla^2 V_E(z_0^E(t))\tilde f_0(t)=\\
	&=-2 \frac{\vert\dot{z}_0^E(t(s))\vert}{\sqrt{2V_E(z_0^E(t(s)))}}\frac{d}{ds}\left(V(z_0^E(t(s))) \frac{\vert\dot{z}_0^E(t(s))\vert}{\sqrt{2V_E(z_0^E(t(s)))}}\frac{d}{ds}\tilde{f}_0(t(s))\right)+\\
	&\quad+\vert\dot{z}_0(t(s))\vert^2\nabla^2 V_E(z_0^E(t(s)))\tilde f_0(t(s))\\
	&\Rightarrow \tilde{f}_0''(s)-\nabla^2V_E(z_0^E(s))\tilde{f}_0(s)=0, 
\end{split}
\end{equation*}
where we took only the first-order terms and used the transformation rules between $d/dt$ and $d/ds$, the conservation of $L=\vert\dot{z}_0^E(t(s))\vert\sqrt{V_E(z_0^E(t(s)))}$ and the Euler-Lagrange equation for $z_0^E(t)$.\\
Since $\tilde{f}_0(-T)=\bar{e}$ and $\tilde{f}_0(T)=\boldsymbol{0}$, $f_0(s)$ solves the one-dimensional system 
\begin{equation*}
\begin{cases}
	f_0''(s)=-\omega^2 f_0(s), \quad s\in[-T,T] \\
	f_0(-T)=1, f_0(T)=0,
\end{cases}
\end{equation*}
namely, recalling the definition of $T$ in (\ref{omotetica esterna}), 
\begin{equation*}
\tilde{f}_0(s)=\frac{1}{2}\left(\frac{\sqrt{2\mathcal{E}}}{\omega|\gamma(\bx)|}\cos{(\omega s)}-\frac{\sqrt{2\mathcal{E}}}{\sqrt{2\mathcal{E}-\omega^2|\gamma(\bx)|^2}}\sin{(\omega s)}\right)\bar{e}.
\end{equation*}
With the same reasoning and taking into account  that $\tilde{f}_1(-T)=\boldsymbol{0}$ and $\tilde{f}_1(T)=\bar{e}$, we have that $\tilde{f}_1(s)=\tilde{f}_0(-s)$, and then we can finally find
\begin{equation}\label{variazioni ext}
\begin{split}
	\frac{d}{ds}\partial_{\bar{e},v}z_0^E(-T)=\frac{\mathcal{E}-\om|\gamma(\bx)|^2}{|\gamma(\bx)|\sqrt{2\mathcal{E}-\omega^2|\gamma(\bx)|^2}}\bar{e},\quad &\frac{d}{ds}\partial_{\bar{e},v}z_0^E(T)=-\frac{\mathcal{E}}{|\gamma(\bx)|\sqrt{2\mathcal{E}-\omega^2|\gamma(\bx)|^2}}\bar{e},\\
	\frac{d}{ds}\partial_{\bar{e},w}z_0^E(-T)=\frac{\mathcal{E}}{|\gamma(\bx)|\sqrt{2\mathcal{E}-\omega^2|\gamma(\bx)|^2}}\bar{e},\quad &\frac{d}{ds}\partial_{\bar{e},w}z_0^E(T)=-\frac{\mathcal{E}-\om|\gamma(\bx)|^2}{|\gamma(\bx)|\sqrt{2\mathcal{E}-\omega^2|\gamma(\bx)|^2}}\bar{e}. 
\end{split}
\end{equation}
Taking together (\ref{second ext}), (\ref{secon ext 2}), (\ref{second ext 3}) and (\ref{variazioni ext}), one can find the analytical expressions of the second derivatives of $S_E(\xi_0,\xi_1)$, computed for $\xi_0=\xi_1=\bar{\xi}$: 
\begin{equation}\label{outer der def}
\begin{split}
	&\partial_{\xi_0}^2S_E(\bar{\xi},\bar{\xi})=\partial_{\xi_1}^2S(\bar{\xi},\bar{\xi})=-\frac{|\dot{\gamma}(\bx)|^2}{2|\gamma(\bx)|}\frac{\E-\om|\gamma(\bx)|^2}{\sqrt{V_E(\gamma(\bx))}}-\sqrt{V_E(\gamma(\bx))}e^{i\bx}\cdot\ddot{\gamma}(\bx),\\
	&\partial_{\xi_0,\xi_1}^2S_E(\bar{\xi},\bar{\xi})=\partial_{\xi_1,\xi_0}^2S(\bar{\xi},\bar{\xi})=-\frac{|\dot{\gamma}|^2}{2|\gamma(\bx)|}\frac{\E}{\sqrt{V_E(\gamma(\bx))}}. 
\end{split}
\end{equation} 
If $\gamma$ is parametrised by arc length, $|\dot \gamma(\bx)|=1$ and $\ddot{\gamma}(\bx)=-k(\bx)n(\bx)=-k(\bx)e^{i\bx}$, where $k(\bx)$ is the curvature of $\gamma$ in $\bx$: Eqs.(\ref{outer der def}) simplify then in 
\begin{equation}\label{outer second der arc}
\begin{split}
	&\partial_{\xi_0}^2S_E(\gamma(\bx),\gamma(\bx))=\partial_{\xi_1}^2S_E(\gamma(\bx),\gamma(\bx))=\frac{\E}{2|\gamma(\bx)|\sqrt{\rule{0pt}{2ex}V_E(\gamma(\bx))}}+\frac{\sqrt{\rule{0pt}{2ex}V_E(\gamma(\bx))}}{|\gamma(\bx)|}\left(|\gamma(\bx)|k(\bx)-1\right), \\
	&\partial_{\xi_0,\xi_1}^2S_E(\gamma(\bx),\gamma(\bx))=\partial_{\xi_1,\xi_0}^2S_E(\gamma(\bx),\gamma(\bx))=-\frac{\E}{2|\gamma(\bx)|\sqrt{\rule{0pt}{2ex}V_E(\gamma(\bx))}}. 
\end{split}
\end{equation}
Eq.(\ref{outer second der arc}) highlights that the second term in $\partial^2_{\xi_0}S(\gamma(\bx),\gamma(\bx))$ represents a perturbation of the homogeneus second derivative with respect to the circular case, where $\left(|\gamma(\bx)|k(\bx)-1\right)=0$ for every $\bx\in[0,2\pi]$. 

\subsection{Inner dynamics: computation of the derivatives of $S_I(\tilde{\xi},\xi_1)$}

With reference to the Notation \ref{notazione}, from Proposition \ref{lem levi civita} we know that the fixed ends inner problem 

\begin{equation}\label{inner LC initial}
\begin{cases}
	(\Prob_I)[z(s)] &s\in[0,T_I], \\
	z_I(0)=z_0^I, z_I(T_I)=z_1^I.
\end{cases}
\end{equation}
is conjugated, by means of the Levi-Civita transformations, to the regularised problem 
\begin{equation*}
\begin{cases}
	(\Prob_{LC})[w(\tau)] &\tau\in[-T,T],\\
	w(-T)=w_0, w(T)=w_1
\end{cases}
\end{equation*}
where $\Omega^2=2(\Eh)$, $E=\mu$, $w_0^2=z_0^I$, $w_1^2=z_1^I$ and $\tau=\tau(s)$ such that $\frac{d\tau}{ds}=\frac{1}{2|z(s)|}$.
In the following, we will work with the Levi-Civita variables, taking respectively for $w_0$ the negative determination of the square root of $z_0^I$ and for $w_1$ the positive determination of the square root of $z_1^I$, namely, in polar coordinates, 
\begin{equation}\label{determinazioni}
z_0^I=\vert z_0^I\vert e^{i\theta_0}\Rightarrow w_0=-\sqrt{\vert z_0^I\vert}e^{i\frac{\theta_0}{2}},\quad z_1^I=\vert z_1^I\vert e^{i\theta_1}\Rightarrow w_1=\sqrt{\vert z_1^I\vert}e^{i\frac{\theta_1}{2}}. 
\end{equation}
To compute the derivatives of $S_I(\gamma(\xi_0),\gamma(\xi_1))$, define then 
\begin{equation*}
L_I(z(t))=\int_0^1\vert\dot{z}(t)\vert\sqrt{V_I(z(t))}dt=\int_0^1\vert\dot{z}(t)\vert\sqrt{\mathcal{E}+h+\frac{\mu}{\vert z(t)\vert}}dt,
\end{equation*}
where $t\in[0,1]$ is the usual geodesic time. Passing to the Levi-Civita plane: 
\begin{equation*}
\begin{split}
	L_I(z)=2\int_0^1\vert\dot{w}(t)\vert\sqrt{2(\mathcal{E}+h)\frac{\vert w(t)\vert^2}{2}+\mu}dt=2\int_0^1\vert\dot{w}(t)\vert\sqrt{\frac{\Omega^2}{2}\vert w(t)\vert^2+E}=2\tilde{L}_I(w)
\end{split}
\end{equation*}
According to the choice for the initial and final point of $\omega(\tau)$ defined in (\ref{determinazioni}), in the Levi-Civita plane the function $S(\xi_0,\xi_1)$ can be written as 
\begin{equation*}
S_I(\xi_0,\xi_1)=d_I(\gamma(\xi_0),\gamma(\xi_1))=2\tilde{d}_I(\phi_-(\xi_0),\phi_+(\xi_1))=2\tilde{S}_I(\xi_0,\xi_1),
\end{equation*}
where $\tilde{d}_I$ is the distance associated to $\tilde{L}_I$ and $\phi_-(\xi), \phi_+(\xi)$ are defined in two neighborhoods respectively of $\xi_0$ and $\xi_1$ as follows: given $\epsilon>0$ and expressing $\gamma(\xi)$ in polar coordinates, namely, $\gamma(\xi)=\rho_{\gamma}(\xi) e^{i\theta_{\gamma}(\xi)}$: 
\begin{equation}\label{def phi}
\begin{split}
	\phi_-(\xi):[\xi_0-\epsilon,\xi_0+\epsilon]\rightarrow\mathbb{C}\simeq\mathbb{R}^2,\quad \phi_-(\xi)=-\sqrt{\rho_{\gamma}(\xi)}e^{i\frac{\theta_{\gamma}(\xi)}{2}} \\
	\phi_+(\xi):[\xi_1-\epsilon,\xi_1+\epsilon]\rightarrow\mathbb{C}\simeq\mathbb{R}^2,\quad \phi_+(\xi)=\sqrt{\rho_{\gamma}(\xi)}e^{i\frac{\theta_{\gamma}(\xi)}{2}} 
\end{split}
\end{equation}
Equations (\ref{def phi}) allow us to compute the transformed of $\gamma(\bar{\xi})$, $\dot{\gamma}(\bar{\xi})$, $\ddot{\gamma}(\bar{\xi})$, seen both as initial and ending point of our arc. Without loss of generalization, let us suppose that $\gamma(\bx)=|\gamma(\bx)|(1,0)$ and $\dot{\gamma}(\bx)=|\dot{\gamma}(\bx)|(0,1)$:  using the relation $\phi_{\pm}(\xi)^2=\gamma(\xi)$, one has
\begin{equation*}
\begin{aligned}
	&\phi_-(\bar{\xi})=\sqrt{|\gamma(\bx)|}(-1,0),\quad\phi_+(\bar{\xi})=\sqrt{|\gamma(\bx)|}(1,0)\\
	&\dot{\phi}_-(\bar{\xi})=\frac{|\dot{\gamma}(\bx)|}{2\sqrt{|\gamma(\bx)|}}(0,-1)=\frac{|\dot{\gamma}(\bx)|}{2\sqrt{|\gamma(\bx)|}}t_-(\bar{\xi}),\\
	&\dot{\phi}_+(\bar{\xi})=\frac{|\dot{\gamma}(\bx)|}{2\sqrt{|\gamma(\bx)|}}(0,1)=\frac{|\dot{\gamma}(\bx)|}{2\sqrt{|\gamma(\bx)|}}t_+(\bar{\xi})\\ 
\end{aligned}
\end{equation*}
where $t_-(\bx)=(0,-1)$ and $t_+(\bx)=(0,1)$, 
and $\ddot{\phi}_{\pm}(\bx)$ satisfy the equations $\ddot{\gamma}(\bx)=2(\dot{\phi}_{\pm}^2(\bx)+\phi_{\pm}(\bx)\ddot{\phi}_{\pm}(\bx))$. \\
In order to compute the derivatives of $S_I(\bar{\xi},\bar{\xi})$, we can use the same techniques used in Section \ref{ssec: outer der} for the outer dynamics, taking into account that, in the Levi-Civita plane, the starting and final point are different. \\
Let us start with the derivation of the homotetic equilibrium orbit: in the physical plane, it corresponds to the ejection-collision solution $\hat{z}_I(s)$ of the fixed-end problem 
\begin{equation*}
\begin{cases}
	(\Prob_I)[z(s)], &s\in[0,T_I], \\
	z(0)=z(T_I)=\gamma(\bx), 
\end{cases}
\end{equation*}
which corresponds, in the Levi-Civita variables, to the solution $w_0(\tau)$ of the problem
\begin{equation*}
\begin{cases}
	(\Prob_{LC})[w(\tau)], &\tau\in[-T,T], \\
	w(-T)=\phi_-(\bx),w(T)=\phi_+(\bx), 
\end{cases}
\end{equation*}
from which one obtains
\begin{equation*}
\begin{split}
	&w_0(\tau)=\frac{\sqrt{2E}}{\Omega}\sinh{(\Omega \tau)}(1,0), \quad T=\frac{1}{\Omega}\arcsinh{\left(\frac{\Omega\sqrt{|\gamma(\bx)|}}{\sqrt{2E}}\right)}\\
	&\Rightarrow w_0'(-T)=w_0'(T)=\sqrt{2E+\Omega^2|\gamma(\bx)|}(1,0). 
\end{split}
\end{equation*}
Proceeding as in (\ref{ext prime der}), one has then
\begin{equation*}
\partial_{\xi_0}S_I(\bar{\xi},\bar{\xi})=-\frac{2}{\sqrt{2}}w_0'(-T)\cdot \dot{\phi}_-(\bar{\xi})=0, \quad\partial_{\xi_1}S_I(\bar{\xi},\bar{\xi})=\frac{2}{\sqrt{2}}w_0'(T)\cdot \dot{\phi}_+(\bar{\xi})=0.
\end{equation*}
As for the second derivatives, taking into account that $\phi_{\pm}$ are not parametrised by arc length: 
\begin{equation}\label{inner sec der}
\begin{aligned}
	&\partial_{\xi_0}^2S_I(\bar{\xi},\bar{\xi})&&=2\partial_{\xi_0}^2\tilde{S}_I(\bar{\xi},\bar{\xi})=2\nabla_{w_0}^2\tilde{d}_I(\phi_-(\bar{\xi}),\phi_+(\bar{\xi}))\dot{\phi}_-(\bar{\xi})\cdot \dot{\phi}_-(\bar{\xi})+2\nabla_{w_0}\tilde{d}_I(\phi_-(\bar{\xi}),\phi_+(\bar{\xi}))\cdot\ddot{\phi}_-(\bar{\xi})=\\
	& &&=\frac{|\dot{\gamma}(\bx)|^2}{2|\gamma(\bx)|}\nabla_{w_0}^2\tilde{d}_I(\phi_-(\bar{\xi}),\phi_+(\bar{\xi}))t_-(\bar{\xi})\cdot t_-(\bar{\xi})+2\nabla_{w_0}\tilde{d}_I(\phi_-(\bar{\xi}),\phi_+(\bar{\xi}))\cdot\ddot{\phi}_-(\bar{\xi})=\\
	& &&=-\frac{|\dot{\gamma}(\bx)|^2}{2\sqrt{2}|\gamma(\bx)|}\frac{d}{d\tau}\left(\partial_{t_-(\bar{\xi}),v}w_0\right)(-T)\cdot t_-(\bar{\xi})-\sqrt{2}w_0'(-T)\cdot \ddot{\phi_-}(\bx), \\
	&\partial_{\xi_1}^2S_I(\bar{\xi},\bar{\xi})&&=\frac{|\dot{\gamma}(\bx)|^2}{2\sqrt{2}|\gamma(\bx)|}\frac{d}{d\tau}\left(\partial_{t_+(\bar{\xi}),w}w_0\right)(T)\cdot t_+(\bar{\xi})+\sqrt{2}w_0'(T)\cdot \ddot{\phi_+}(\bx)\\
	&\partial_{\xi_0,\xi_1}^2S_I(\bar{\xi},\bar{\xi})&&=-\frac{|\dot{\gamma}(\bx)|^2}{2\sqrt{2}|\gamma(\bx)|}\frac{d}{d\tau}\left(\partial_{t_+(\bar{\xi}),w}w_0\right)(-T)\cdot t_-(\bar{\xi}),\\
	&\partial_{\xi_1,\xi_0}^2S_I(\bar{\xi},\bar{\xi})&&=\frac{|\dot{\gamma}(\bx)|^2}{2\sqrt{2}|\gamma(\bx)|}\frac{d}{d\tau}\left(\partial_{t_-(\bar{\xi}),v}w_0\right)(T)\cdot t_+(\bar{\xi}). 
\end{aligned}
\end{equation}
We can compute the variations $\partial_{t_-(\bar{\xi}),v}w_0(\tau)$, $\partial_{t_+(\bar{\xi}),w}w_0(\tau)$ as in Section \ref{ssec: outer der}: by imposing $\partial_{t_-(\bar{\xi}),v}w_0(\tau)=\tilde{g}_0(\tau)=g_0(\tau)t_-({\bar{\xi}})$ and $\partial_{t_+(\bar{\xi}),w}w_0(\tau)=\tilde{g}_1(\tau)=g_1(\tau)t_+({\bar{\xi}})$ we have that $g_0(\tau)$ and $g_1(\tau)$ are solutions of the two one-dimensional systems
\begin{equation*}
\begin{cases}
	g_0''(\tau)=\Omega^2 g_0(\tau), \tau\in[-T,T]\\
	g_0(-T)=1, g_0(T)=0,
\end{cases}
\begin{cases}
	g_1''(\tau)=\Omega^2 g_1(\tau), \tau\in[-T,T]\\
	g_1(-T)=0, g_1(T)=1,
\end{cases}
\end{equation*}
and then we obtain
\begin{equation*}
\begin{aligned}
	&\frac{d}{d\tau}(\partial_{t_-(\bar{\xi}),v}w_0)(-T)=-\frac{E+\Omega^2|\gamma(\bx)|}{\sqrt{\rule{0pt}{2ex}|\gamma(\bx)|}\sqrt{\rule{0pt}{2ex}2E+\Omega^2|\gamma(\bx)|}}t_-({\bar{\xi}}),\\ &\frac{d}{d\tau}(\partial_{t_-(\bar{\xi}),v}w_0)(T)=-\frac{E}{\sqrt{\rule{0pt}{2ex}|\gamma(\bx)|}\sqrt{\rule{0pt}{2ex}2E+\Omega^2|\gamma(\bx)|}}t_-({\bar{\xi}}),\\
	&\frac{d}{d\tau}(\partial_{t_+(\bar{\xi}),w}w_0)(-T)=\frac{E}{\sqrt{\rule{0pt}{2ex}|\gamma(\bx)|}\sqrt{\rule{0pt}{2ex}2E+\Omega^2|\gamma(\bx)|}}t_+({\bar{\xi}}),\\ &\frac{d}{d\tau}(\partial_{t_+(\bar{\xi}),w}w_0)(T)=\frac{E+\Omega^2|\gamma(\bx)|}{\sqrt{\rule{0pt}{2ex}|\gamma(\bx)|}\rule{0pt}{2ex}\sqrt{\rule{0pt}{2ex}2E+\Omega^2|\gamma(\bx)|}}t_+({\bar{\xi}}). 
\end{aligned}
\end{equation*}
Then, taking into account (\ref{inner sec der}) and recalling that $E=\mu$, $\Omega^2=\mathcal{E}+h$, we finally obtain (recall that we are assuming that $\gamma(\bx)\parallel(1,0)$)
\begin{equation}\label{inner der fin}
\begin{aligned}
	&\partial_{\xi_0}^2S_I(\bar{\xi},\bar{\xi})&&=\frac{|\dot{\gamma}(\bx)|^2}{4|\gamma(\bx)|^2}\frac{\mu+2(\Eh)|\gamma(\bx)|}{\sqrt{\rule{0pt}{2ex}V_I(\gamma(\bx))}}-2\sqrt{|\gamma(\bx)|}\sqrt{V_I(\gamma(\bx))}(1,0)\cdot\ddot{\phi}_-(\bx),\\
	&\partial_{\xi_1}^2S_I(\bar{\xi},\bar{\xi})&&=\frac{|\dot{\gamma}(\bx)|^2}{4|\gamma(\bx)|^2}\frac{\mu+2(\Eh)|\gamma(\bx)|}{\sqrt{\rule{0pt}{2ex}V_I(\gamma(\bx))}}+2\sqrt{|\gamma(\bx)|}\sqrt{V_I(\gamma(\bx))}(1,0)\cdot\ddot{\phi}_+(\bx),\\
	&\partial_{\xi_0,\xi_1}^2S_I(\bar{\xi},\bar{\xi})&&=\partial_{\xi_1,\xi_0}^2S_I(\bar{\xi},\bar{\xi})=\frac{|\dot{\gamma}(\bx)|^2}{4|\gamma(\bx)|^2}\frac{\mu}{\sqrt{\rule{0pt}{2ex}V_I(\gamma(\bar{\xi}))}}.
\end{aligned}
\end{equation}
If $\gamma(\xi)$ is parametrised by arc length, $|\dot{\gamma}(\bx)|=1$ and $\ddot{\gamma}(\bx)k(\bx)=(-1,0)$, then 
\begin{equation*}
\ddot{\phi}_-(\bx)=\frac{1}{2\sqrt{\rule{0pt}{2ex}|\gamma(\bx)|}}\left(k(\bx)-\frac{1}{2|\gamma(\bx)|}\right)(1,0)=-\ddot{\phi}_+,
\end{equation*}
and Eqs.(\ref{inner der fin}) simplify as
\begin{equation}\label{inner der arc}
\begin{aligned}
	&\partial_{\x0}^2S_I(\bx,\bx)=\partial_{\x1}^2S_I(\bx,\bx)=-\frac{\mu}{4|\gamma(\bx)|^2\sqrt{\rule{0pt}{2ex}V_I(\gamma(\bx))}}-\sqrt{V_I(\gamma(\bx))}\left(k(\bx)-\frac{1}{|\gamma(\bx)|}\right),\\
	&\partial_{\x0,\x1}^2S_I(\bx,\bx)=\partial_{\x1,\x0}^2S_I(\bx,\bx)=\frac{\mu}{4|\gamma(\bx)|^2\sqrt{\rule{0pt}{2ex}V_I(\gamma(\bx))}}
\end{aligned}
\end{equation}

\subsection{Stability properties of $(\bar{\xi},0)$}\label{ssec stab}

Let us now suppose that $\gamma(\bx)$ is parametrised by arc length (the general case can be treated in the same way, taking into account the explicit expression of $\ddot{\gamma}(\bar{\xi})$): taking together (\ref{outer second der arc}) and (\ref{inner der arc}), one can see that they can be written in the form
\begin{equation}\label{E0,I0}
\begin{aligned}
	&\partial_{\xi_0}^2S_E(\bar{\xi},\bar{\xi})=\partial_{\xi_1}^2S_E(\bar{\xi},\bar{\xi})=E_0+\varepsilon_E,&&\\
	&\partial_{\xi_0,\xi_1}^2S_E(\bar{\xi},\bar{\xi})=\partial_{\xi_1,\xi_0}^2S_E(\bar{\xi},\bar{\xi})=-E_0,&&\\
	&E_0=\frac{\mathcal{E}}{2|\gamma(\bx)|\sqrt{\rule{0pt}{2ex}V_E( \gamma(\bar{\xi}))}}, &&\varepsilon_E=(|\gamma(\bx)|k(\bar{\xi})-1)\frac{\sqrt{\rule{0pt}{2ex}V_E(\gamma(\bar{\xi}))}}{|\gamma(\bx)|},\\
	&\partial_{\xi_0}^2S_I(\bar{\xi},\bar{\xi})=\partial_{\xi_1}^2S_I(\bar{\xi},\bar{\xi})=I_0+\varepsilon_I,&&\\
	&\partial_{\xi_0,\xi_1}^2S_I(\bar{\xi},\bar{\xi})=\partial_{\xi_1,\xi_0}^2S_I(\bar{\xi},\bar{\xi})=-I_0,&&\\
	&I_0=-\frac{\mu}{4|\gamma(\bx)|^2\sqrt{\rule{0pt}{2ex}V_I( \gamma(\bar{\xi}))}}, &&\varepsilon_I=-\left(k(\bar{\xi})-\frac{1}{|\gamma(\bx)|}\right)\sqrt{\rule{0pt}{2ex}V_I(\gamma(\bar{\xi}))},
\end{aligned}
\end{equation}
The terms $\varepsilon_{E\backslash I}$ can be seen as the perturbations induced to the second derivatives when the domain's boundary $\partial D$ is not a circle. Turning to the matrices defined in Section \ref{ssec: Jac 1}, we have that 
\begin{equation*}
D_{(\tilde{\xi},\xi_1,\alpha_1)}\Phi(\bar{q})=
\begin{pmatrix}
	\frac{\partial_{\xi_0,\xi_1}^2S_E(\bar{\xi},\bar{\xi})}{\sqrt{V_E(\gamma(\bar{\xi}))}} & 0 &0\\
	\partial_{\xi_1}^2S_E(\bar{\xi},\bar{\xi})+\partial_{\xi_0}^2S_I(\bar{\xi},\bar{\xi})
	& \partial_{\xi_0,\xi_1}^2S_I(\bar{\xi},\bar{\xi})&0\\
	-\frac{\partial_{\xi_0,\xi_1}S_I(\bar{\xi},\bar{\xi})}{\sqrt{V_E(\gamma(\bar{\xi}))}}&-\frac{\partial_{\xi_1}^2S_I(\bar{\xi},\bar{\xi})}{\sqrt{V_E(\gamma(\bar{\xi}))}}&1 
\end{pmatrix},
\end{equation*}
whose determinant is given by 
\begin{equation*}
\det{\left(D_{(\tilde{\xi},\xi_1,\alpha_1)}\Phi(\bar{q})\right)}=-\frac{\mathcal{E}\mu}{8|\gamma(\bx)|^3\sqrt{\rule{0pt}{2ex}V_I(\gamma(\bar{\xi}))}V_E(\gamma(\bar{\xi}))}<0
\end{equation*}
in the Hill's region $\mathcal{H}$. The implicit function theorem can be then applied and we have that there exist $I_1,I_2,I_2$, $J_1,J_2$ neighborhoods respectively of $\bar{\xi}$ and $0$ and there is a function $\Psi:I_1\times J_1\rightarrow I_2\times I_3\times J_2$ such that for every $(\xi_0,\alpha_0)\in I_1\times J_1$ one has $\Phi(\xi_0,\alpha_0,\Psi(\xi_0,\alpha_0))=0$. Moreover 
\begin{equation*}
D_{(\xi_0,\alpha_0)}\Psi(\bar p)=-\left(D_{(\tilde{\xi},\xi_1,\alpha_1)}\Phi(\bar{q})\right)^{-1}D_{(\xi_0,\alpha_0)}\Phi(\bar{q}).
\end{equation*}
The function $F: (\xi_0, \alpha_0)\mapsto(\xi_1(\xi_0,\alpha_0),\alpha_1(\xi_0,\alpha_0))$ is given by the last two components of $\Psi$, then $DF(\bar p)$ is composed by the last two rows of $D_{(\xi_0,\alpha_0)}\Psi(\bar p)$. Direct computations show that 
\begin{equation}\label{jac gen}
DF(\bar p)=\begin{pmatrix}
	A_{11}&A_{12}\\
	A_{21}&A_{22}
\end{pmatrix},
\end{equation}
where
\begin{equation*}
\begin{split}
	&A_{11}=1+\frac{2\varepsilon_E+\varepsilon_I}{I_0}+\frac{\varepsilon_E(\varepsilon_E+\varepsilon_I+I_0)}{E_0 I_0},\\
	&A_{12}=\sqrt{V_E(\gamma(\bar{\xi}))}\left(\frac{1}{I_0}+\frac{1}{E_0}\right)+\sqrt{V_E(\gamma(\bar{\xi}))}\frac{\varepsilon_E+\varepsilon_I}{E_0 I_0},\\
	&A_{21}=\frac{2\varepsilon_E(\varepsilon_I+I_0)+\varepsilon_I(\varepsilon_I+2I_0)}{I_0\sqrt{\rule{0pt}{2ex}V_E(\gamma(\bar{\xi}))}}+\frac{\varepsilon_E[\varepsilon_E(\varepsilon_I+I_0)+\varepsilon_I(\varepsilon_I+2I_0)]}{E_0I_0\sqrt{\rule{0pt}{2ex}V_E(\gamma(\bar{\xi}))}}\\
	&A_{22}=1+\frac{\varepsilon_E}{E_0}+\frac{\varepsilon_I(2I_0+\varepsilon_I+E_0+\varepsilon_E)}{E_0 I_0}
\end{split}
\end{equation*}
The stability properties of the equilibrium in $(\bar{\xi},0)$ can be studied looking at the eigenvalues of $DF(\bar p)$, see \cite{MR3293130}: let us denote them with $\lambda_1$ and $\lambda_2$. 
Direct computations show that $\det{(DF(\bar p))}=1$: this is a completely general fact, as the map $F$ describes a conservative system, and, from an algebraic point of view, implies that $\lambda_1\lambda_2=1$. Therefore we can have two cases: 
\begin{itemize}
\item $\lambda_1,\lambda_2\in\mathbb{R}\Rightarrow \lambda_1=1/\lambda_2$: if $\lambda_1\neq1$, then $(\bar{\xi},0)$ is an unstable saddle; 
\item $\lambda_1,\lambda_2\in\mathbb{C}/\mathbb{R}\Rightarrow\lambda_1=\overline{\lambda}_2$ and $\lambda_1,\lambda_2\in S^1$: then $(\bar{\xi},0)$ is a stable center.  
\end{itemize} 
We can distnguish between the two cases by considering the characteristic polynomial of $DF((
\bar p)$. 
\begin{rem}\label{deltastab}
Denoted by $p(\lambda)=a\lambda^2+b\lambda+c$ the characteristic polynomial of $DF(\bar{\xi},0)$, let $\Delta=b^2-4ac$ its discriminant. Then 
\begin{itemize}
	\item if $\Delta>0\Rightarrow(\bar{\xi},0)$ is a saddle for $F$;
	\item if $\Delta<0\Rightarrow(\bar{\xi},0)$ is a center for $F$; 
\end{itemize}
\end{rem}
The value of $\Delta$ with respect to the physical quantities of the problem can be directly computed: it results that 
\begin{equation*}
\Delta=A B C D,
\end{equation*}
where
\begin{equation*}
\begin{split}
	&A=\frac{16}{\mathcal{E}^2\mu^2}\left(\sqrt{V_I(\gamma(\bar{\xi})}-\sqrt{V_E(\gamma(\bar{\xi}))}\right)\left(|\gamma(\bx)|k(\bar{\xi})-1\right), \\
	&B=\mathcal{E}-\left(|\gamma(\bx)|k(\bar{\xi})-1\right)\left(\sqrt{V_I(\gamma(\bar{\xi}))}-\sqrt{V_E(\gamma(\bar{\xi}))}\right)\sqrt{V_E(\gamma(\bar{\xi}))},\\ 
	&C=-\mu\sqrt{V_E(\gamma({\bar{\xi}}))}+2|\gamma(	\bx)|B\sqrt{V_I(\gamma(\bx))}, \\
	&D=\mu+2|\gamma(\bx)|		\left(|\gamma(\bx)|k(\bar{\xi})-1\right)\sqrt{V_I(\gamma(\bar{\xi}))}\left(\sqrt{V_I(\gamma(\bar{\xi}))}-\sqrt{V_E(\gamma(\bar{\xi}))}\right).
\end{split}
\end{equation*}

\section{A direct investigation: elliptic domains}\label{sec:alliptic_domains}
When the expression of $\gamma(\xi)$ is given, the general theory developed in Section \ref{sec:general_stab} can be used to study the effective stability of the fixed points of the map $F$. In this Section we investigate the existence and stability of equilibrium orbits for our dynamical system when $D$ is an elliptic domain. Let us suppose that $\partial D$ is an ellipse with semimajor axis $a=1$ and eccentricity $0\leq e<1$. Denoted by $b=a\sqrt{1-e^2}=\sqrt{1-e^2}$ the semiminor axis, one can parametrise $\partial D$ as 
\begin{equation*}
	\gamma(\xi)=(\cos{\xi},b \sin{\xi}),\text{ }\xi\in[0,2\pi],  
\end{equation*}
which can be written as 
\begin{equation*}
	\begin{aligned}
		&\gamma(\theta)=(1+f(e,\theta))e^{i\theta}, \\
		&f(e,\theta)=\frac{\sqrt{1-e^2}}{(1-e^2\cos^2{\theta})}-1, \quad \theta\in[0,2\pi], \quad e\in[0,1). 
	\end{aligned}
\end{equation*}
From direct computations and from Remark \ref{oss fixed points}, one has that the orbit with initial conditions $z(0)=\gamma(\bx),\text{ }z'(0)=\sqrt{V_E(z(0))}z(0)/|z(0)|$ is an homotetic equilibrium orbit if and only if $\bx=k\pi/2$, $\textcolor{black}{k\in\{0,1,2,3,4\}}$: due to the symmetry of the problem, we can restrict our study to the two cases $\bx_0=0$ and $\bx_1=\pi/2$. 
We have that 
\begin{equation*}
	\begin{aligned}
		&\gamma(0)=(1,0), &&\dot{\gamma}(0)=(0,b), &&\ddot{\gamma}(0)=(-1,0)\\
		&\gamma(\pi/2)=(0,b), &&\dot{\gamma}(\pi/2)=(-1,0), &&\ddot{\gamma}(\pi/2)=(0,-b):
	\end{aligned}
\end{equation*}
The stabilty properties of the $F$-fixed points $(\bx_0,0)$ and $(\bx_1,0)$  can be deduced as in Section \ref{sec:general_stab}: in particular, from Eqs.(\ref{outer der def}) and (\ref{inner sec der}), one obtains
\begin{equation}\label{ellisse0}
	\begin{aligned}
		&\partial_{\xi_0}^2S_I(\bx_0,0)=\partial_{\xi_1}^2S_I(\bx_0,0)=I_0^{(0)}+\varepsilon_I^{(0)}, \quad
		\partial_{\xi_0,\xi_1}^2S_I(\bx_0,0)=\partial_{\xi_1,\x0}^2S_I(\bx_0,0)=-I_0^{(0)}, \\
		&I_0^{(0)}=-\frac{(1-e^2)\mu}{4\sqrt{\rule{0pt}{2ex}V_I(\gamma(\bx_0))}}, \quad \varepsilon_I^{(0)}=-e^2\sqrt{V_I(\gamma(\bx_0))},\\
		&\partial_{\xi_0}^2S_E(\bx_0,0)=\partial_{\xi_1}^2S_E(\bx_0,0)=E_0^{0}+\varepsilon_E^{(0)}, \quad
		\partial_{\xi_0,\xi_1}^2S_E(\bx_0,0)=\partial_{\xi_1}^2S_E(\bx_0,0)=-E_0^{(0)}, \\
		&E_0^{(0)}=\frac{(1-e^2)\mathcal{E}}{2\sqrt{\rule{0pt}{2ex}V_E(\gamma(\bx_0))}}, \quad \varepsilon_E^{(0)}=e^2\sqrt{V_E(\gamma(\bx_0))},\\
	\end{aligned}
\end{equation}

\begin{equation}\label{ellisse90}
	\begin{aligned}
		&\partial_{\xi_0}^2S_I(\bx_1,0)=\partial_{\xi_1}^2S_I(\bx_1,0)=I_0^{(1)}+\varepsilon_I^{(1)}, \quad
		\partial_{\xi_0,\xi_1}^2S_I(\bx_1,0)=\partial_{\xi_1}^2S_I(\bx_1,0)=-I_0^{(1)}, \\
		&I_0^{(1)}=-\frac{\mu}{4(1-e^2)\sqrt{\rule{0pt}{2ex}V_I(\gamma(\bx_1))}}, \quad \varepsilon_I^{(1)}=\frac{e^2}{\sqrt{1-e^2}}\sqrt{\rule{0pt}{2ex}V_I(\gamma(\bx_1))},\\
		&\partial_{\xi_0}^2S_E(\bx_1,0)=\partial_{\xi_1}^2S_E(\bx_1,0)=E_0^{(1)}+\varepsilon_E^{(1)}, \quad \partial_{\xi_0,\xi_1}^2S_E(\bx_1,0)=\partial_{\xi_1}^2S_E(\bx_1,0)=-E_0^{(1)}, \\
		&E_0^{(1)}=\frac{\mathcal{E}}{2\sqrt{1-e^2}\sqrt{\rule{0pt}{2ex}V_E(\gamma(\bx_1))}}, \quad \varepsilon_E^{(1)}=-\frac{e^2}{\sqrt{1-e^2}}\sqrt{V_E(\gamma(\bx_1))},\\
	\end{aligned}
\end{equation}
	and then we have 
	\begin{equation}\label{delta01}
		\begin{aligned}
			&\Delta^{(0)}=A^{(0)}B^{(0)}C^{(0)}D^{(0)}, \qquad\qquad\qquad\qquad\qquad\qquad\qquad \Delta^{(1)}=A^{(1)}B^{(1)}C^{(1)}D^{(1)},\\
			&A^{(0)}=-\frac{16}{\E^2\mu^2}\frac{e^2}{1-e^2}\left(\sqrt{V_E(\gamma(\bar{\xi}_0))}-\sqrt{V_I(\gamma(\bar{\xi}_0))}\right)\\ 
			&A^{(1)}=\frac{16e^2}{\E^2\mu^2}\left(\sqrt{V_E(\gamma(\bar{\xi}_1))}-\sqrt{V_I(\gamma(\bar{\xi}_1))}\right)\\
			&B^{(0)}=\mu+2\frac{e^2}{1-e^2}\sqrt{V_I(\gamma(\bar{\xi}_0))}\left(\sqrt{V_I(\gamma(\bar{\xi}_0))}-\sqrt{V_E(\gamma(\bar{\xi}_0))}\right)\\
			&B^{(1)}=\mu-2e^2\sqrt{1-e^2}\sqrt{V_I(\gamma(\bar{\xi}_1))}\left(\sqrt{V_I(\gamma(\bar{\xi}_1))}-\sqrt{V_E(\gamma(\bar{\xi}_1))}\right)\\
			&C^{(0)}=\E+\frac{e^2}{1-e^2}\sqrt{V_E(\gamma(\bar{\xi}_0))}\left(\sqrt{V_E(\gamma(\bar{\xi}_0))}-\sqrt{V_I(\gamma(\bar{\xi}_0))}\right)\\
			&C^{(1)}=\E-e^2\sqrt{V_E(\gamma(\bar{\xi}_1))}\left(\sqrt{V_E(\gamma(\bar{\xi}_1))}-\sqrt{V_I(\gamma(\bar{\xi}_1))}\right)\\
			&D^{(0)}=-\mu\sqrt{V_E(\gamma(\bar{\xi}_0))}+2\sqrt{V_I(\gamma(\bar{\xi}_0))}C^{(0)}\\
			&D^{(1)}=-\mu\sqrt{V_E(\gamma(\bar{\xi}_1))}+2\sqrt{1-e^2}\sqrt{V_I(\gamma(\bar{\xi}_1))}C^{(1)}
		\end{aligned}
	\end{equation}		
\subsection{Asymptotic behaviours}\label{ssec: ecc piccola}
It is convenient to start the study of the elliptic case by investigating some of the properties of the first return map on a circular domain. When $e=0$, for every $\bar{\xi}\in[0,2\pi]$ the pair $(\bx,0)$ is an homotetic fixed point for $F$, with 
\begin{equation*}
	DF(\bx,0)
=\begin{pmatrix}
	1 & \sqrt{\E-\frac{\om}{2}}\left(\frac{2}{\E}\sqrt{\E-\frac{\om}{2}}-\frac{4}{\mu}\sqrt{\Eh+\mu}\right)\\0&1
\end{pmatrix}.
\end{equation*}
This is consistent with the expression of $F$ for a circular domain: when $D$ is a disk of radius $1$, from the central symmetry of both the domain and the inner and outer potentials one has that $F$ is a rigid translation of the form 
\begin{equation*}
	F_{circ}(\x0,\alpha_0)=(\x0+\theta(\alpha_0), \alpha_0),   
\end{equation*}
and, as a consequence, $\Delta_{circ}=0$ for every homotetic point $(\bx,0)$. The circular case represents then a degenerate case for the study of the linear stability of the homotetic points; nevertheless, the possibility to compute the explicit expression of $F_{circ}$ allows to study directly the map: considering the phase space $(\xi, \alpha)$, one has that the set $[0,2\pi]\times\{0\}$ is the invariant set containing all the homotetic points, and that all the \textcolor{black}{orbits} of $F_{circ}$ lie on the invariant lines $[0,2\pi]\times\{\bar{\alpha}\}$, where the value of $\theta(\bar{\alpha})$ determines their nature. The systematical study of the circular case, in a more convenient variational setting, is one of the subject of a further work \cite{IreneSusNew}. \\
Let us suppose that  $e>0$ and small. Recalling that $b=\sqrt{1-e^2}$, the expression of $\Delta^{(0)}$ and $\Delta^{(1)}$ in Eqs.(\ref{delta01}) can be expanded in Taylor series around $e=0$, obtaining, from direct computations,
\begin{gather}\label{espansione ecc piccole}
	\Delta^{(0)}=f_2e^2+f_4e^4+\mathcal{O}(e^6), \quad\Delta^{(1)}=g_2e^2+g_4e^4+\mathcal{O}(e^6),\\
	f_2=-g_2=-\frac{4\left(\sqrt{\E-\om/2}-\sqrt{\Eh+\mu}\right)\left(2\E\sqrt{\Eh+\mu}-\mu\sqrt{\E-\om/2}\right)}{\mu\E}. 
\end{gather}
Hence, when $e$ is sufficiently small, the sign of $\Delta^{(0)}$ and $\Delta^{(1)}$ is determined by the quantity $\left(2\E\sqrt{\Eh+\mu}-\mu\sqrt{\E-\om/2}\right)$. \\
Let us now suppose to fix the parameters related to te external dynamics, namely, $\E$ and $\omega$, and to let vary $\mu$ and $h$. If $e$ is small enough, $\Delta^{(0)}$ and $\Delta^{(1)}$ have opposite sign; in particular: 
\begin{itemize}
	\item if $\frac{\sqrt{\Eh+\mu}}{\mu}<\frac{\sqrt{2\E-\om}}{2\sqrt{2}\E}$, $\Delta^{(0)}<0$ and $\Delta^{(1)}>0$. Then, from  Remark \ref{deltastab}, one has that $(0,0)$ is a stable center and $(\pi/2,0)$ is an unstable saddle for $F$; 
	\item if $\frac{\sqrt{\Eh+\mu}}{\mu}>\frac{\sqrt{2\E-\om}}{2\sqrt{2}\E}$, for the same reasoning $(0,0)$ is a saddle and $(\pi/2,0)$ is a center.  
\end{itemize}
Fixing $\E$ and $\omega$, one has also: 
\begin{equation*}
\lim_{h\to\infty}\Delta^{(0)}=\lim_{\mu\to\infty}\Delta^{(0)}=\lim_{h\to\infty}\Delta^{(1)}=\lim_{\mu\to\infty}\Delta^{(1)}=\infty. 
\end{equation*}
As a final investigation on the asyntotical behaviour of $\Delta^{(0)}$ and $\Delta^{(1)}$, let us suppose to fix the physical parameters related to the inner dynamics and analyse the sign of the discriminants for $\E\to\infty$. From direct computations, one has 
\begin{equation*}
	\begin{aligned}
	\ell_0=\lim_{\E\to\infty}\Delta^{(0)}=\frac{(b^2-1 ) (2 h + 2 \mu + \om) (2 (b^2-1) h - 2 \mu + (b^2-1 ) \om)}{b^4\mu^2}\\
	\ell_1=\lim_{\E\to\infty}\Delta^{(1)}=\frac{b (b^2-1) (2 b h + 2 \mu + b^3 \om) (2 (b^2-1) h + 
		b (2 \mu + b (b^2-1) \om))}{\mu^2}. 
	\end{aligned}
\end{equation*}
In particular, it results $\ell_0>0$ for every fixed $0<b<1$ and $h,\mu, \omega>0$ and 
\begin{equation*}
	\ell_1>0 \Leftrightarrow 0<b<1 \text{ and }0<\mu<\bar{\bar{\mu}}=\frac{(b^2-1)(2h+b^2\om)}{2b}. 
\end{equation*}

Taking together the above considerations, one can give some general results, which hold for small eccentricity or for high values of $h$, $\mu$ or $\E$. 
\begin{prop}\label{eccpiccola}
	 For every $\E,\omega>0$ with $\om>2\E$ we have:  
	\begin{enumerate}
	\item[I)] for every fixed $h, \mu>0$: 
		\begin{itemize}
			\item[Ia)] if $\frac{\sqrt{\E+h+\mu}}{\mu}<\frac{\sqrt{2\E-\om}}{2\sqrt{2}\E}$, then there is $\bar{e}\in(0,1)$ such that, for every $e\in(o,\bar{e})$: 
 $\bar{z}_0$ is stable and $\bar{z}_{\pi/2}$ is unstable;
 			\item[Ib)] if $\frac{\sqrt{\E+h+\mu}}{\mu}>\frac{\sqrt{2\E-\om}}{2\sqrt{2}\E}$,  then there is $\bar{e}\in(0,1)$ such that, for every $e\in(o,\bar{e})$: 
$\bar{z}_0$ is unstable and $\bar{z}_{\pi/2}$ is stable.;
		\end{itemize}
	\item[II)] for all fixed $e\in(0,1)$, $h>0$, there is $\bar{\mu}>0$ such that for every  $\mu>\bar{\mu}$ the homotetic fixed points $(0,0)$ and $(\pi/2,0)$ are saddles; 
	\item[III)] for all fixed $e\in(0,1)$, $\mu>0$, there is $\bar{h}>0$ such that for every  $h>\bar{h}$ the homotetic fixed points $(0,0)$ and $(\pi/2,0)$ are saddles. 
	\end{enumerate}
For all fixed $e\in(0,1)$, $  h, \omega>0$ there are $\bar{\E}>0$ and $\bar{\bar{\mu}}>0$ such that, if $\E>\bar{\E}$: 
\begin{enumerate}
	\item[IVa)] if $\mu>\bar{\bar{\mu}}$, $(0,0)$ is a saddle and $(\pi/2,0)$ is a center; 
	\item[IVb)] if $0<\mu<\bar{\bar{\mu}}$, $(0,0)$ and $(\pi/2,0)$ are both saddles. 
\end{enumerate}
  
With the additional hypothesis that the $F$ is well defined on the whole ellipse, in cases (II) and (III), as well as (IVb), there must be at least a stable fixed point with $\xi_0\in(0,\pi/2)$; hence, by symmetry, $F$ admits at least $4$ stable period one non-homotetic fixed points. 
\end{prop}

Proposition \ref{eccpiccola}(I) provides an approximated relation between $h$ and $\mu$ through which one can find two regimes in the parameters such that, for $e$ sufficiently small, the stability of the homotetic fixed points can be easily deduced. In particular, there is a curve which, for $e$ sufficiently small, divides the two cases (Ia) and (Ib). As all the involved quantities are positive, one has
\begin{equation*}
	\frac{\sqrt{\Eh+\mu}}{\mu}>\frac{\sqrt{2\E-\om}}{2\sqrt{2}\E} \Longleftrightarrow h>\frac{2\E-\om}{8\E^2}\mu^2-\mu-\E=p(\mu). 
\end{equation*}
When $\E$ and $\omega$ are fixed, as well as $e$ small enough, the 2-degree polynomial $p(\mu)$ describes then a parabola on the plane $(h,\mu)$ such that, for fixed $\mu$, if $h>>p(\mu)$, then we are in case (Ib); on the contrary, if $h<<p(\mu)$, the case (Ia) is verified. \\
We stress that this behavour holds only for $e$ small enough for $f_2$ and $g_2$ to be the dominant terms in the expansions of $\Delta^{(0)}$ and $\Delta^{(1)}$. Moreover, one can verify that $f_4,g_4=\mathcal{O}(h^2\sqrt{\mu})$, and that all the further terms of the Taylor expansion are of the order of positive powers of $\mu$ and $h$: as a consequence, for every $e>0$, eventually the two parameters would be too large to use the above approximation.\\

\begin{figure}
	\centering
	\includegraphics[width=0.8\linewidth]{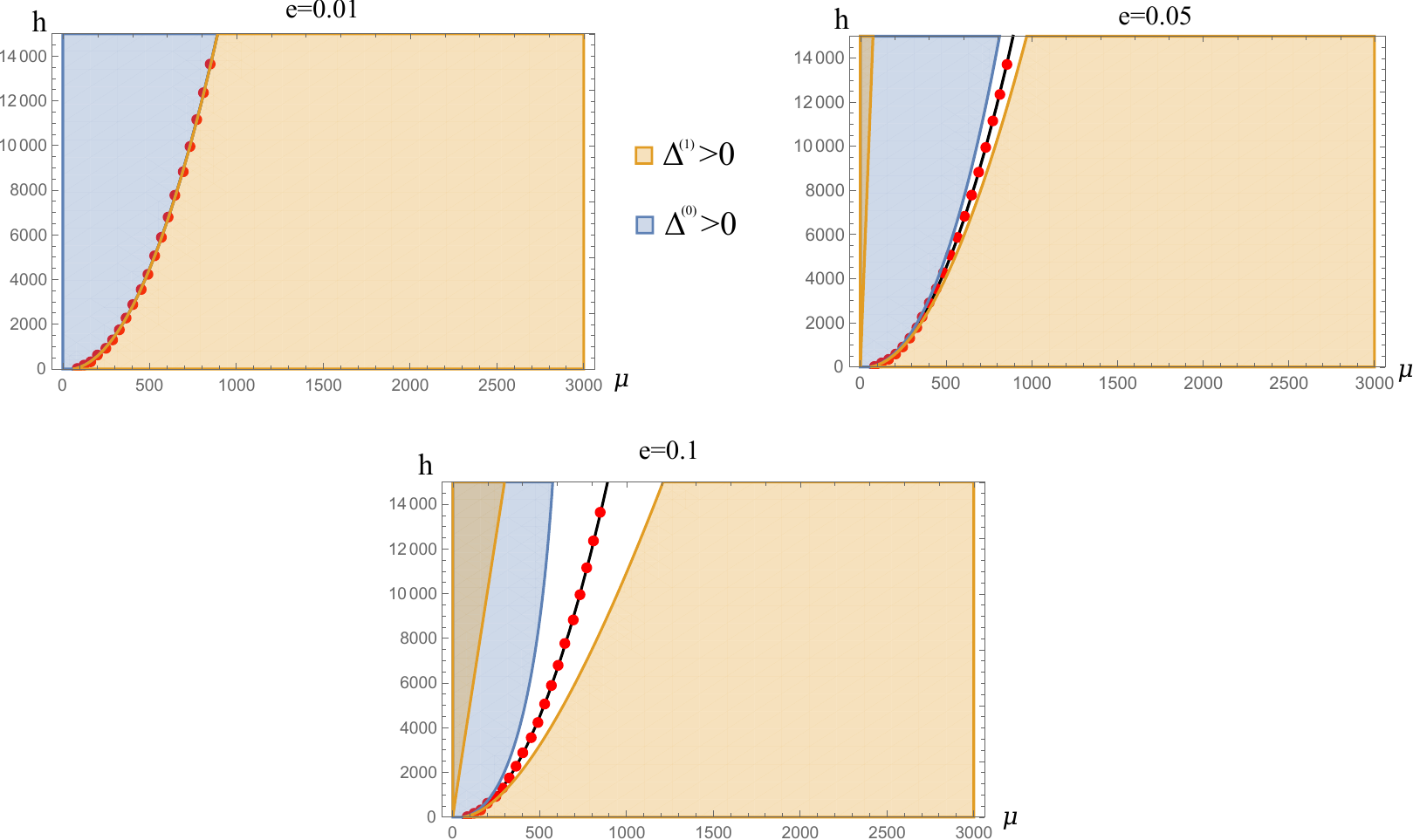}
	\caption{Sign of $\Delta^{(0)}$ and $\Delta^{(1)}$ in the $(\mu,h)-$plane for $\E=10$ and $\omega=2$. The red dotted curve represents the parabola $h=p(\mu)$. }
	\label{fig:parabole}
\end{figure}

Figure \ref{fig:parabole} gives a comparison between the parabola $p(\mu)$ (red dots) and the effective curves of change of sign for $\Delta^{(0)}$ and $\Delta^{(1)}$ in the $(\mu,h)$ plane for $\E=10$, $\omega=2$ and increasing eccentricities. As one can see, for very small eccentricities the approxmation fiven by $p(\mu)$ is very good even for extremely high values of $\mu$ and $h$; on the other hand, the increase of the eccentricity and of the two inner parameters made this approximation worse. \\
Moving to moderate and high eccentricities, the behaviour of the signs of $\Delta^{(0)}$ and $\Delta^{(1)}$ becomes more complex: to give an example of this, Figure  \ref{fig:h-mu-ecc-alte} shows the sign of both the discriinants as functions of $h$ and $\mu$ and for fixed $\E$, $\omega$ and eccentricity of the ellipse. It is present a reminiscence of the original parabola $p(\mu)$, which tends to widen for increasing eccentricity, while other sign-changing curves, deriving by the influence of the higher order terms in (\ref{espansione ecc piccole}), are present.

\begin{figure}
	\centering
	\includegraphics[width=0.7\linewidth]{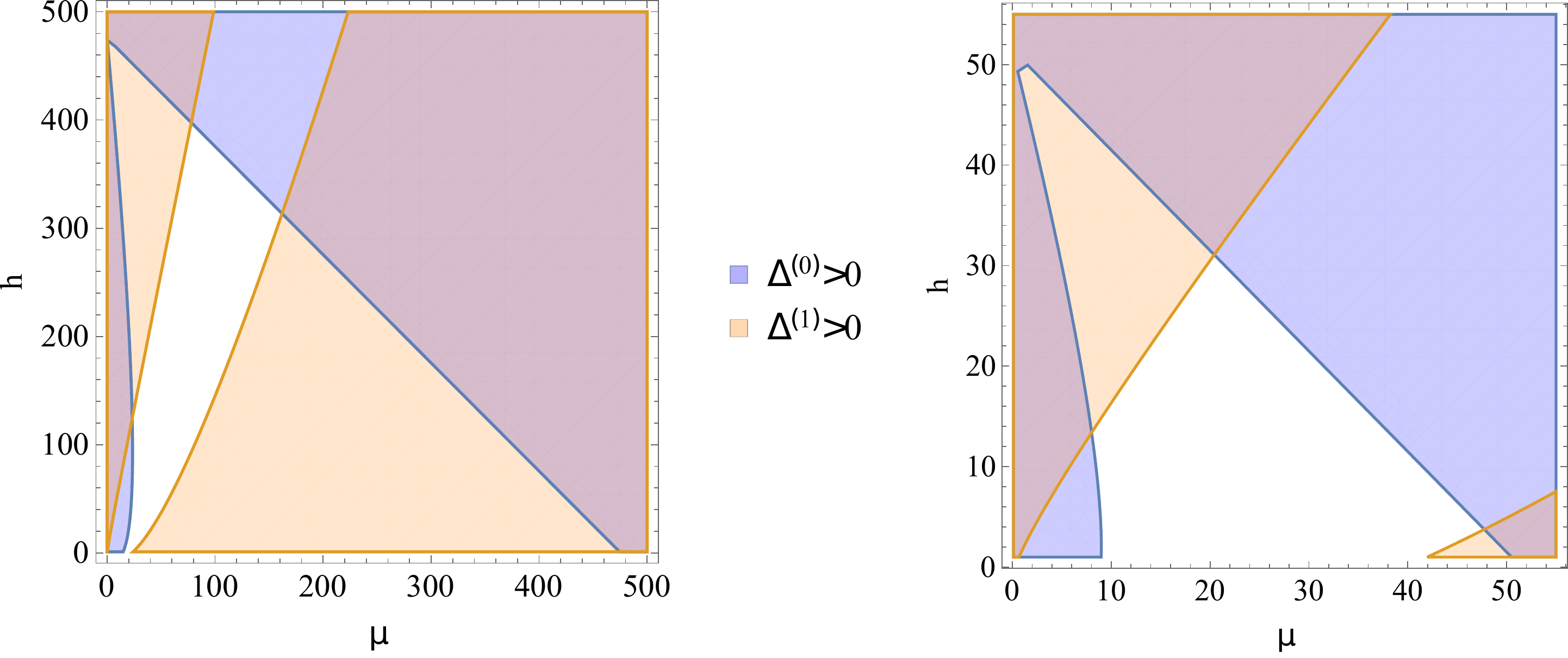}
	\caption{Sign of $\Delta^{(0)}$ and $\Delta^{(1)}$ in the $(\mu,h)-$plane for $\E=2.5$, $\omega=\sqrt{2}$ and $e=0.3$ (left) or $e=0.5$ (right).}
	\label{fig:h-mu-ecc-alte}
\end{figure}

\subsection{Arising of 2-periodic brake orbits}\label{ssec: brake}

As already seen in Section \ref{ssec: ecc piccola}, the existence of non homotetic $1-$ periodic points of $F$ can be deduced by the signs of $\Delta^{(0)}$ and $\Delta^{(1)}$, namely, by the stability properties of the homotetic points. On the other hand, other analytical techniques can be used to assure the existence of particular periodic orbits with period greater than $1$. It is the case of the so-called non homotetic\textit{brake orbits}, namely, 2-periodic orbits with homotetic outer arcs (see Figure \ref{fig:brake}), whose existence can be proved for suitable regimes of the parameters through an application of the shooting method. 
\begin{figure}[h!]
	\centering
	\includegraphics[width=0.5\linewidth]{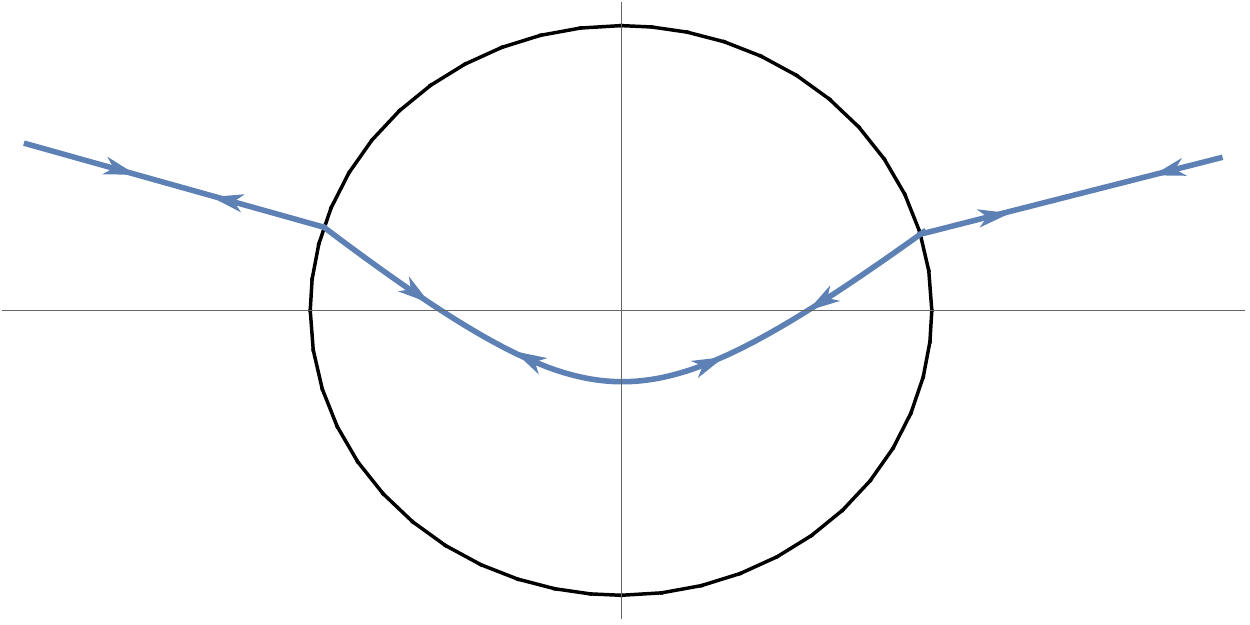}
	\caption{Example of 2-periodic brake orbit: the homotetic outer arcs are connected by an inner hyperbola. }
	\label{fig:brake}
\end{figure}
The existence of brake orbits is equivalent to the existence of non-homotetic zeros for the \textit{Free Fall} map, which quantifies the scattering with respect to the radial direction of the trajectory after entering the domain. Given $\theta\in[0,2\pi] $, consider the homotetic outer arc with initial points $(p_0,v_0)$, where $p_0$ is the intersection beween the ellipse and the radial straight line of inclination $\theta$, while $v_0$ is the outward-pointing radial vector in the direction of $\theta$ and such that $|v_0|^2/2-V_E(p_0)=0$: if we denote, as in the previous Sections, with $(p_1,v_1)$ the position and velocity vectors after two consecutive outer and inner crossings (with the respective refractions), the free fall map $\theta\mapsto\delta(\theta)$ returns the angle $\delta$ between $v_1$ and $p_1$ (see Figure \ref{fig:free-fall}).\\
\begin{figure}
	\centering
	\includegraphics[width=.5\linewidth]{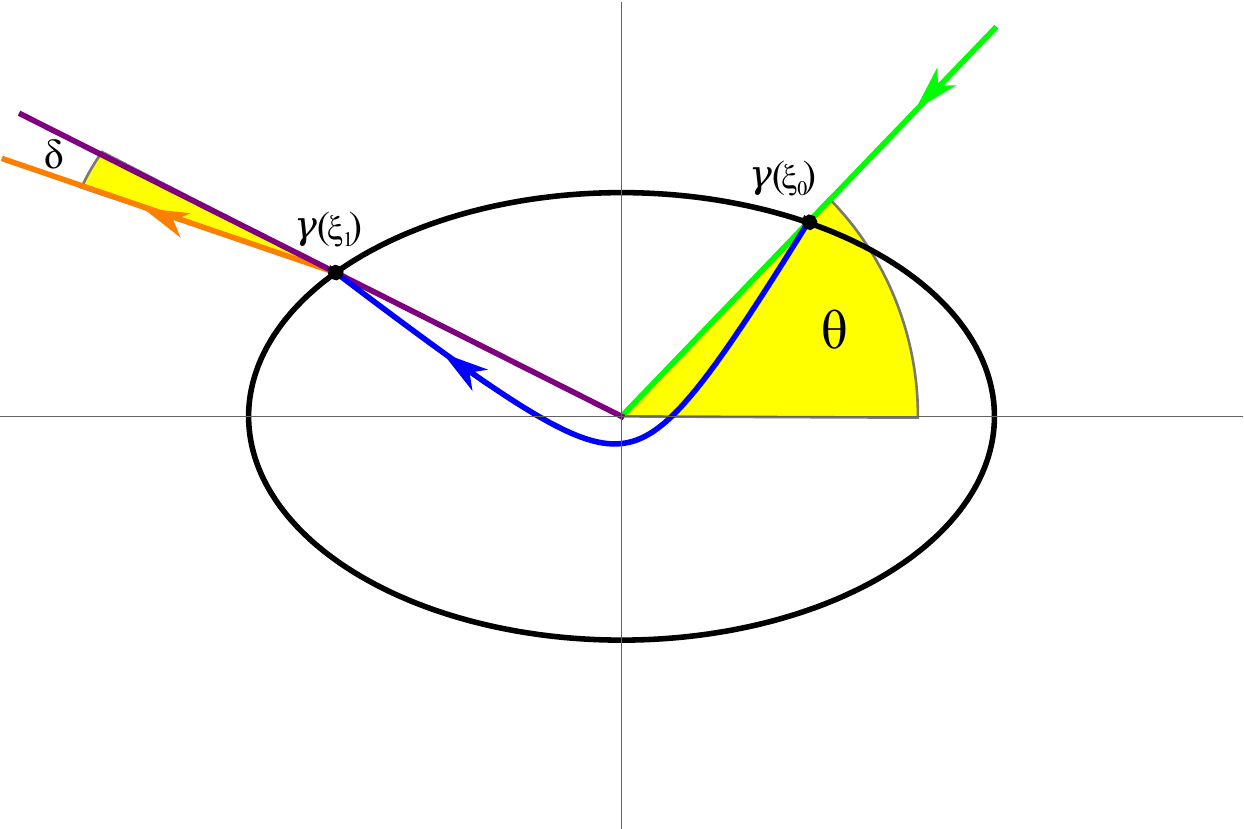}
	\caption{Free fall map on the ellipse.}
	\label{fig:free-fall}
\end{figure}
%\begin{oss}
If we consider general domains whose boundary intersects orthogonally the axes, as in the case of the ellipse, Theorems \ref{thm esistenza ext} and \ref{thm esistenza int} guarantee that the Free Fall map is well defined in suitable neighborhoods of the homotetic orbits in the horizontal and vertical directions. Nevertheless, as the construction of $\delta(\theta)$ only involves the refraction rule and the inner dynamics, under suitable hypotheses on $\partial D$ one can assure that it is well defined globally on $[0,2\pi]$; in particular, it is sufficient to require: 
\begin{equation}\label{FF cond}
	\begin{aligned}
	&(I)\text{ the well definition of the inner dynamics globally on }\partial D; \\
	&(II)\text{ a global transversality property of }\partial D\text{ with respect to the radial directions, namely}\\
		&\qquad\qquad\qquad\qquad\qquad\qquad\qquad\qquad\forall\xi\in I\quad \gamma(\xi)\nparallel\dot{\gamma(\xi)}. 
	\end{aligned}
\end{equation}
When $\partial D$ satisfies the above properties, one can continously extend  $\delta(\theta)$ even in the case that the first return map $F$ is not well defined (see Remark \ref{rifl totale}): it suffices to impose $\delta(\theta)=\pi/2$ whenever the inner angle $\beta_1$ is greater or equal to $\beta_{crit}=\arcsin(\sqrt{V_E(p_1)/V_I(p_1)})$ and $\delta(\theta)=-\pi/2$ when $\beta_1\leq-\beta_{crit}$. As a consequence, the function $\delta$ results to be a \textcolor{black}{continous} function of $\theta\in[0,2\pi]$, \textcolor{black}{differentiable whenever $|\beta_1|<\beta_{crit}$, as in neighborhoods of homotetic solutions}. Moreover, condition (II) assures that whenever $\delta(\theta)=0$ the Free Fall map is well-defined, since the refracted outer arc is not tangent to $\partial D$. \\
While the geometrical implications of condition (II) are rather immediate, condition (I) is implied by takig a particular class of domains characterized by a convexity property with respect to the hyperbol\ae. In particular, we shall give the following definition. 
\begin{defn}
	We say that the domain $D$ is \textbf{convex for hyperbol\ae~ for fixed } $\boldsymbol{h},\boldsymbol{\E}$\textbf{ and} $\boldsymbol{\mu}$ if every Keplerian hyperbola with energy $\Eh$ and central mass $\mu$ intersects $\partial D$ at most in two points. \\
	The domain $D$ is \textbf{convex for hyperbol\ae} if the previous condition holds for every positive $\E, h$ and $\mu$. 
\end{defn}
The connection between the Free Fall map and the brake orbits is straightforward: $(\cos\bar{\theta}, \sin\bar{\theta})$ is the direction of a 2-periodic brake orbit if and only if $\delta(\bar{\theta})=0$ and, denoting with $\xi_ {\bar{\theta}}$ the parameter in $I$ such that $\gamma(\xi_ {\bar{\theta}})$ has polar angle $\bar\theta$, $\gamma(\xi_ {\bar{\theta}})\not\perp\dot{\gamma}(\xi_ {\bar{\theta}})$. \\ 
	Let us remark that, by the properties of the ellipse, one has that $\delta(k\pi/2)=0$ for $k=0,1,2,3$, and that condition (II) is trivially true. The following Proposition shows that, when the eccentricity is small enough the elliptical domains are also convex by hyperbol\ae, leading to the conclusion that, in these cases, the Free Fall map is globally well denfined.  
\begin{prop}
	If $D$ is an ellipse parametrised by $(\cos{(\xi)}, b\sin(\xi))$ with eccentricity $e\in[0,1/\sqrt{2})$, then it is convex by hyperbol\ae.
\end{prop}
\begin{proof}
	
	Let us start by fixing the ellipse's eccentricity $e$ and the parameters $\E, h$ and $\mu$ and by taking the associated family $\mathcal{F}$ of hyperbol\ae, which is continous with respect to variations of the angular momentum and rotations of the axes. Denoting with $\ell$ the absolute value of the angular momentum of a Keplerian hyperbola $\mathcal{K}$ in such a family and with $r_p$ its minimal distance from the origin, one has that (see e.g. \cite{celletti2010stability})
	\begin{equation*}
		r_p=\frac{\ell^2}{\mu\left(1+\sqrt{1+\frac{2(\Eh)\ell^2}{\mu^2}}\right)}, 
	\end{equation*} 
	which is continuous and strictly increasing for $\ell\geq0$. The distance at the pericenter is then $0$ when $\ell=0$ (homotetic orbit) and varies continuosly with $\ell$. Moreover, since for the ellipse Theorem \ref{thm esistenza int} is true for every $\xi\in[0,2\pi]$, for $\ell$ small enough the hyperbol\ae~  of $\mathcal{F}$ intersect $\partial D$ exactly twice.
	Let us now fix a direction in $\R^2$, and consider only the hyperbol\ae~ in $\mathcal{F}$ whose  transverse axis is in the chosen direction, denoting them with $\mathcal F'$. 
	\begin{figure}
		\centering
		\includegraphics[width=0.5\linewidth]{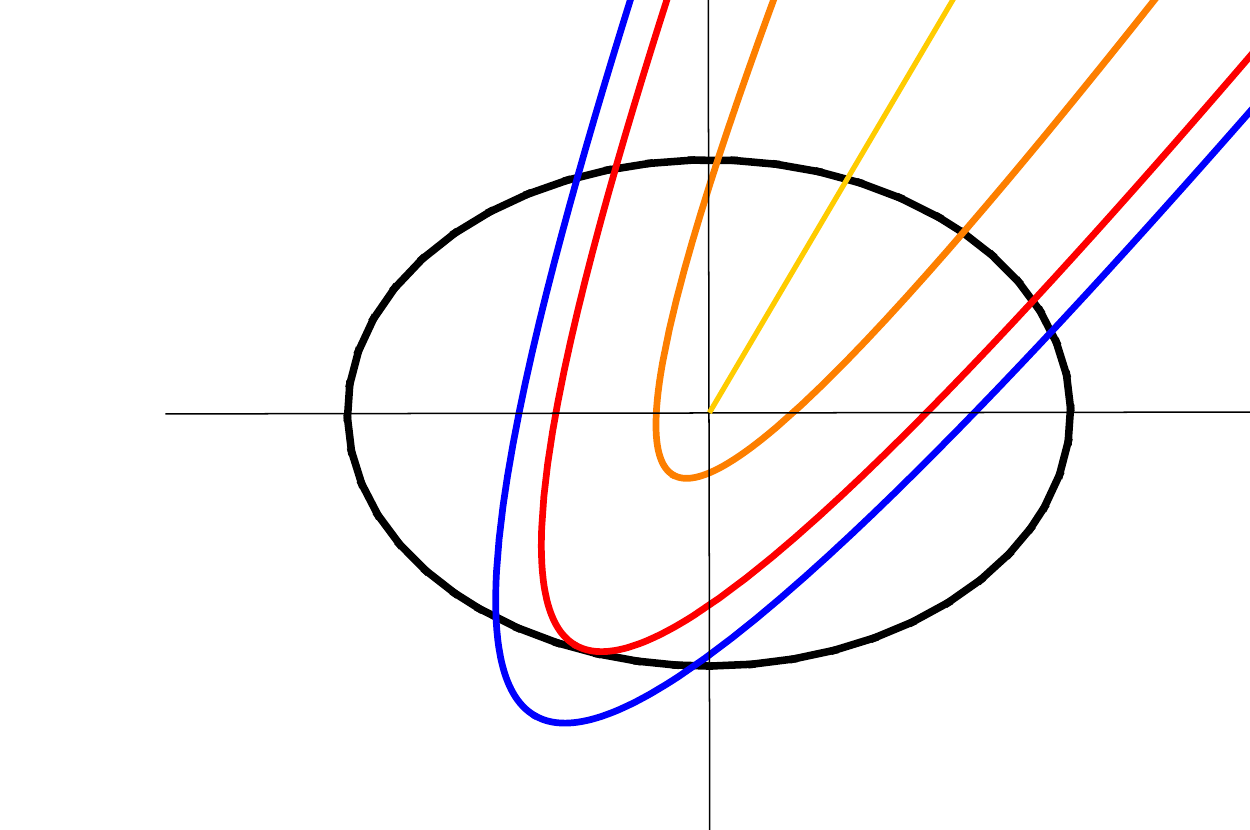}
		\caption{Nested family of Keplerian hyperbol\ae~ $\mathcal F'$ for fixed $\E, h, \mu$ and fixed tranverse axis. As $\ell$ increases, the hyperbol\ae~ move from the inner ones (orange) to the outer ones (blue). }
		\label{fig:iperboli-tg}
	\end{figure}	
		As the eccentricity of a Keplerian hyperbol\ae, whose expression is 
	\begin{equation*}
		e_{hyp}=\sqrt{1+\frac{2(\Eh)\ell^2}{\mu^2}},
	\end{equation*}
	is strictly increasing in $\ell$, such hyperbol\ae ~ are nested as in Figure \ref{fig:iperboli-tg}. Let us suppose that there exists a Keplerian hyperbola in $\mathcal F'$ which intersects $\partial D$ in four points $p_A, p_B, p_C$ and $p_D$. For the previous considerations on the continous dependence and monotonicity of $r_p$ and $e_{hyp}$ on $\ell$, there exists $\bar \ell$ such that the corresponding hyperbola $\bar{\mathcal K}$ in $\mathcal F'$ is tangent from inside to $\partial D$ in a point $p_T$; define $r_T=|p_T|$.  This implies that, denoted with $k_{hyp}(p)$ and $k_{ell}(p)$  respectively the curvatures with respect to the inward-pointing normal vector of $\bar{\mathcal{K}}$ and of $\partial D$ in a point $p$, we obtain
	\begin{equation*}\label{cond 4 int}
		k_{hyp}(p_T)\geq k_{ell}(p_T), 
	\end{equation*}
	which is a necessary condition for the family $\mathcal F'$ to admit an hyperbola which intersects $\partial D$ four times. 
	The ellipse's curvature is always bounded from below by $b$, while one can compute $k_{hyp}(p_T)$ by parametrising $\bar{\mathcal K}(s)$ by the cinetic time and recalling that, for a generic curve  $r(t)$, the curvature is given by 
	\begin{equation*}
		k_r(t)=\frac{\vert\vert r'(t)\wedge r''(t)\vert\vert}{\vert\vert r'(t)\vert\vert^3}.  
	\end{equation*}
	Observing that $p_T\in\partial D$, which implies $r_T\geq b$, one obtains that, if $\bar s$ is such that $\bar{\mathcal{K}}(\bar s)=p_T$, 
		\begin{equation*}
			k_{hyp}(p_T)=\frac{||\bar{\mathcal K}'(\bar s)\wedge\bar{\mathcal K}''(\bar s)||}{||\bar{\mathcal K}'(\bar s)||^3}\leq\frac{||\bar{\mathcal K}''(\bar s)||}{||\bar{\mathcal K}'(\bar s)||^2}=\frac{\mu}{2r_T(r_T(\Eh)+\mu)}\leq\frac{\mu}{2b(b(\Eh)+\mu)}
		\end{equation*}
	Taking together the bounds obtained for $k_{ell}(p_T)$ and $k_{hyp}(p_T)$, one can find a necessary condition for the family $\mathcal{F}'$ to admit hyerbol\ae~ with four intersection points with $\partial D$, given by
	\begin{equation*}
		\frac{\mu}{2b(b(\Eh)+\mu)}\geq b. 
	\end{equation*}
It is then sufficient to require
	\begin{equation}\label{cond conv hyp}
	\frac{\mu}{2b(b(\Eh)+\mu)}< b \Longleftrightarrow   2b^2\left(\frac{\Eh}{\mu}b+1\right)>1
	\end{equation}
	to ensure that $\mathcal F'$ does not admit hyperbol\ae~ of such kind. It is straightforward to observe that (\ref{cond conv hyp}) is trivially satisfied for every $\Eh>0$ ad $\mu>0$ whenever $2b^2>1$, namely $e\in[0,1/\sqrt{2})$. This reasoning can be repeated for every fixed direction for the axis and for every $\E, h, \mu>0$; in particular, it holds also when two of the four points of the original hyperbola (the blue one in Figure \ref{fig:iperboli-tg}) coincide: it is in fact trivially true when $p_C=p_D$, and, if $p_A=p_D$ or $p_B=p_C$, one can take a lower $\ell$ to retrieve the original case. Then the convexity for hyperbol\ae~ is proved whenever $e\in[0,1/\sqrt{2})$.

\end{proof}

Although deriving the explicit expression of $\delta(\theta)$ goes beyond the extent of this study, the values of its derivatives computed at the homotetic points, which can be found by making use of Eqs.(\ref{ellisse0}, \ref{ellisse90}), along with the global good definition of the Free Fall map, provide a sufficient condition for the existence of the brake orbits in the elliptic case.
\begin{thm}\label{teorema brake}
	For every $\E>0$, $\omega>0$ such that $\om>2\E$, $e\in(0,1/\sqrt{2})$ there are $\bar{h}>0$, $\bar{\mu}>0$ such that, if $h>\bar{h}$ and $\mu>\bar{\mu}$, the first return map $F$ admits at least four $2-$periodic brake orbits. 
\end{thm}
\begin{proof}
	By symmetry, it is sufficient to prove that, for $\E, \omega, b$ satisfying the hypotheses of the Theorem, there are $\bar{h}>0$ and $\bar{\mu}>0$ such that, if $h>\bar{h}$ and $\mu>\bar{\mu}$, then $\exists\bar{\theta}\in(0,\pi/2)$ such that $\delta(\bar{\theta})=0$. To this end, we want to find a regime for the parameters such that $\delta'(0)>0$ and $\delta'(\pi/2)>0$. \\
	Recall the definitions of $(\x0,\alpha_0)$, $(\x1,\alpha_1)$ used in Section \ref{sec:first_return}, suppose to work in a neighborhood of $\theta=\pi/2$, and consider the $6-$dimensional variable 
	\begin{equation*}
		\boldsymbol{q}=(\theta,\x0,\alpha_0,\x1,\alpha_1,\delta)\in(0,\pi)\times(0,\pi)\times\left[-\frac{\pi}{2},\frac{\pi}{2}\right]\times(0,\pi)\times\left[-\frac{\pi}{2},\frac{\pi}{2}\right]\times[0,2\pi]
	\end{equation*}  
From elementary geometric considerations and recalling the refraction relations, one has that, defined $\bar{\boldsymbol{q}}=(\pi/2,\pi/2,0,\pi/2,0,0)$, $\Phi(\bar{\boldsymbol{q}})=0$, where

\begin{equation*}
	\Phi(q)=
	\begin{pmatrix}
		\frac{1}{b}\cot\x0-\cot\theta\\
		\alpha_0+\theta-\arccot(b\cot\x0)\\
		\partial_aS_I(\x0,\x1)+\sqrt{V_E(\gamma(\x0))}\sin\alpha_0\\
		\partial_bS_I(\x0,\x1)-\sqrt{V_E(\gamma(\x1))}\sin\alpha_1\\
		\delta+\alpha_1+\arccot(b\cot\x1)-\arccot\left(\frac{\cot\x1}{b}\right), 
	\end{pmatrix}
\end{equation*}
\textcolor{black}{is defined in a suitable neighborhood of $\bar{\boldsymbol{q}}$}. As a consequence, 
\begin{equation*}
	\begin{aligned}
	M=D_{(\x0,\alpha_0,\x1,\alpha_1,\delta)}\Phi(\bar{\boldsymbol{q}})&=
	\begin{pmatrix}
		-\frac{1}{b}&0&0&0&0&\\
		-b&1&0&0&0&\\
		\partial^2_aS_I(\bx_1,\bx_1)&\sqrt{\rule{0pt}{2ex}V_E(\gamma(\bx_1))}&\partial_{ab}S_I(\bx_1,\bx_1)&0&0\\
		\partial_{ab}S_I(\bx_1,\bx_1)&0&\partial^2_{b}S_I(\bx_1,\bx_1)&-\sqrt{\rule{0pt}{2ex}V_E(\gamma(\bx_1))}&0\\
		0&0&b-\frac{1}{b}&1&1
	\end{pmatrix}\\
&\Rightarrow\det(M)=\frac{\mu}{4b^3\sqrt{\Eh+\mu/b}}\sqrt{\E-\frac{\om}{2}b^2}>0. 
\end{aligned}
\end{equation*}
Applying then the implicit function theorem, $\delta'(\pi/2)$ can be computed as the last component of the vector
\begin{equation*}
	-M^{-1}D_\theta\Phi(\bar{\boldsymbol{q}})=-M^{-1}
	\begin{pmatrix}
		1\\1\\0\\0\\0
	\end{pmatrix},
\end{equation*}
obtaining
\begin{equation*}
\delta'(\pi/2)=-\frac{\left(\sqrt{\rule{0pt}{2ex}V_E(\gamma(\bx_1))}b^2+\epsilon_I^{(1)}b-\sqrt{\rule{0pt}{2ex}V_E(\gamma(\bx_1))}\right)\left(\sqrt{\rule{0pt}{2ex}V_E(\gamma(\bx_1))}b^2+(2I_0^{(1)}+\epsilon_I^{(1)})b-\sqrt{\rule{0pt}{2ex}V_E(\gamma(\bx_1))}\right)}{bI_0^{(1)}\sqrt{\rule{0pt}{2ex}V_E(\gamma(\bx_1))}}. 
\end{equation*}
With the same reasoning and taking $\theta\in(-\pi/2,\pi/2)$, one gets
\begin{equation*}
	\delta'(0)=\frac{\left(\sqrt{\rule{0pt}{2ex}V_E(\gamma(\bx_0))}+\epsilon_I^{(0)}b-\sqrt{\rule{0pt}{2ex}V_E(\gamma(\bx_0))}b^2\right)\left(\sqrt{\rule{0pt}{2ex}V_E(\gamma(\bx_0))}b^2-(2I_0^{(0)}+\epsilon_I^{(0)})b-\sqrt{\rule{0pt}{2ex}V_E(\gamma(\bx_0))}\right)}{bI_0^{(0)}\sqrt{\rule{0pt}{2ex}V_E(\gamma(\bx_0))}}. 
\end{equation*}
Taking $b=\sqrt{1-e^2}$, direct computations show that, if 
\begin{equation*}
	\mu>\bar{\mu}=\frac{e^2(1-e^2)^{3/2}(2\E-e^2\om)}{(2e^2-1)^2}
\end{equation*}
and
\begin{equation*}
	\begin{aligned}
	h>\bar{h}=&\frac{1}{4}\left(-2\E-\frac{(4e^2-2)\mu}{e^2\sqrt{1-e^2}}-(1-e^2)\om\right)+\\
	&+\sqrt{\frac{(2\E-(1-e^2)\om)(4\mu-e^2\sqrt{1-e^2}((1-e^2)\om-2\E))}{e^2\sqrt{1-e^2}}}
	\end{aligned}
\end{equation*}
then $\delta'(0)>0$ and $\delta'(\pi/2)>0$, and the statement is proved.
\end{proof}
\begin{rem}
Notice that Theorem \ref{teorema brake} can be extended to general domains $D$ with boundary $\partial D=supp(\tilde{\gamma})$, provided that: 
	\begin{itemize}
		\item conditions (\ref{FF cond}) hold; 
		\item $\gamma$ shares the symmetry properties of the ellipse; 
		\item in the vicinity of the intersections between the coordinate axes and $\partial D$, $\tilde{\gamma}$ and  $\gamma(\xi)=(\cos{\xi},b\sin{\xi})$ are equal up to the second order, namely: 
		\begin{equation*}
			\begin{split}
				&\tilde{\gamma}(\xi)=(\cos{\xi}+f(\xi),b\sin{\xi}+g(\xi)), \\
				f(\bx_{0\backslash1})=g(\bx_{0\backslash1})&=f'(\bx_{0\backslash1})=g'(\bx_{0\backslash1})=f''(\bx_{0\backslash1})=g''(\bx_{0\backslash1})=0. 
			\end{split}
		\end{equation*}  
	\end{itemize}  
As a matter of fact, one has that, locally around $\pi/2$ (the reasoning for $0$ is the same), the vector $\boldsymbol{q}$ defined as in the Theorem satisfies the relation $\tilde{\Phi}(\boldsymbol{q})$, with
\begin{align}
	\tilde{\Phi}(\boldsymbol{q})=
\begin{pmatrix}
	\frac{\cos{\x0}+f(\x0)}{b\sin{\x0}+g(\x0)}-\cot{\theta}\\
	\alpha_0+\theta-\arccot\left(\frac{b\cos{\x0}+g'(\x0)}{\sin{\x0}-f'(\x0)}\right)\\
	\partial_aS_I(\x0,\x1)+\sqrt{V_E(\gamma(\x0))}\sin\alpha_0\\
	\partial_bS_I(\x0,\x1)-\sqrt{V_E(\gamma(\x1))}\sin\alpha_1\\
	\delta+\alpha_1+\arccot\left(\frac{b\cos{\x1}+g'(\x1)}{\sin{\x1}-f'(\x1)}\right)-\arccot\left(\frac{\cos{\x1}+f(\x1)}{b\sin{\x1}+g(\x1)}\right), 
\end{pmatrix}
\end{align}
whose derivatives with respect to all the variables, computed in $\bar{\boldsymbol{q}}$, are the same as in the Theorem. \\

\end{rem}

\begin{esempio}
	To make the reasoning quantitative, let us now consider the case $\E=2.5, \omega=\sqrt{2}, \mu=2$ and $e=0.1$. 
	\begin{figure}
		\centering
		\includegraphics[width=.5\textwidth]{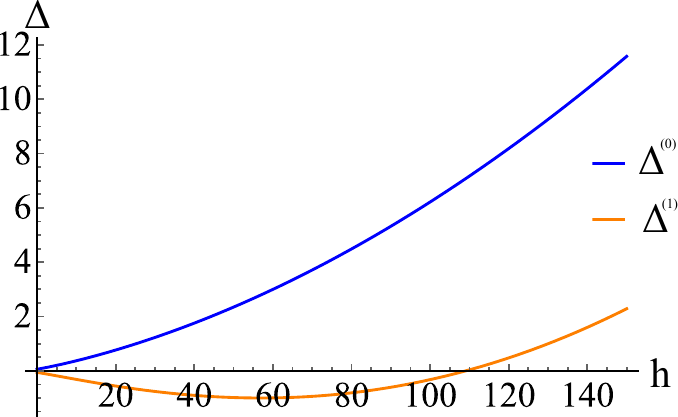}
		\caption{Values of $\Delta^{(0)}$ and $\Delta^{(1)}$ as a function of $h$, with $\E=2.5, \omega=\sqrt{2}, \mu=2, e=0.1$. }
		\label{fig:deltah}
	\end{figure}
	Figure \ref{fig:deltah} shows the signs of $\Delta^{(0)}$ and $\Delta^{(1)}$ as a function of $h$. One can see that, while $(0,0)$ is always an unstable saddle, there is a bifurcation value of $h$ for which $(\pi/2,0)$ changes its stability, whose value is precisely $h_{bif}=109.091$. 
	\begin{figure}
		\centering
		\includegraphics[width=.8\linewidth]{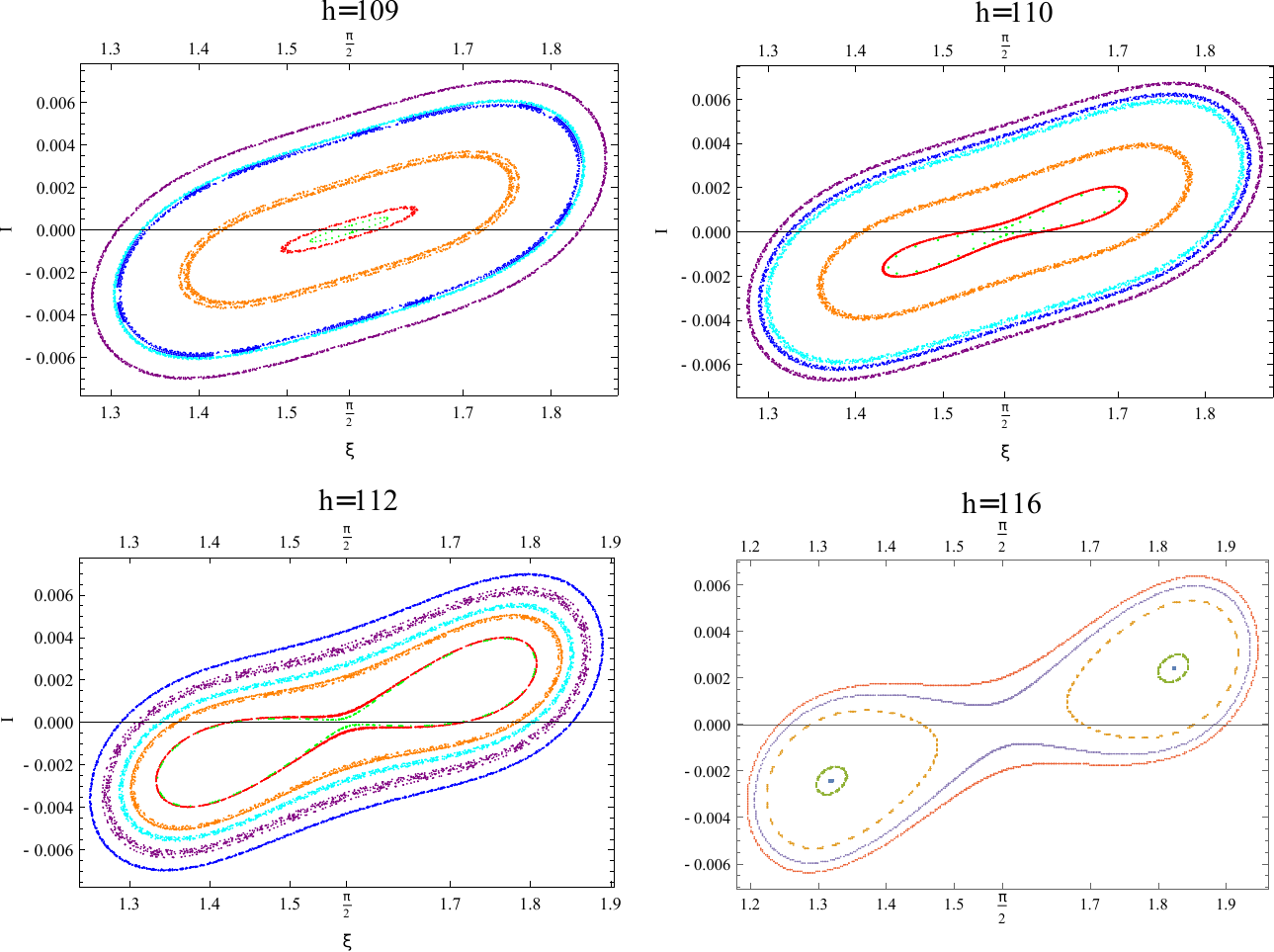}
		\caption{Orbits of $F$ in a neighborhood of the homotetic fixed point $(\pi/2,0)$ for $\E=2.5, \omega=\sqrt{2},\mu=2, e=0.1$ and different values of $h$. The transition of the fixed point from center to saddle is evident. Bottom-Right: the 2-periodic fixed point is detected as the 2-points blue \textcolor{black}{orbit}.}
		\label{fig:transizione-brake}
	\end{figure}
Figure \ref{fig:transizione-brake} shows the transition of $(\pi/2,0)$ from center to saddle, with the concurrent arising of a two periodic orbit. With reference to Theorem \ref{teorema brake}, we have in this case $\bar{\mu}\simeq0.0511$, while $\bar{h}=h_{bif}$: the treshold value for the existence of the 2-periodic brake orbits is then equal to the one for the change of stability of the homotetic equilibrium point, underlying the concurrence of the two phenomena. \\
\begin{figure}
	\centering
	\includegraphics[width=1\linewidth]{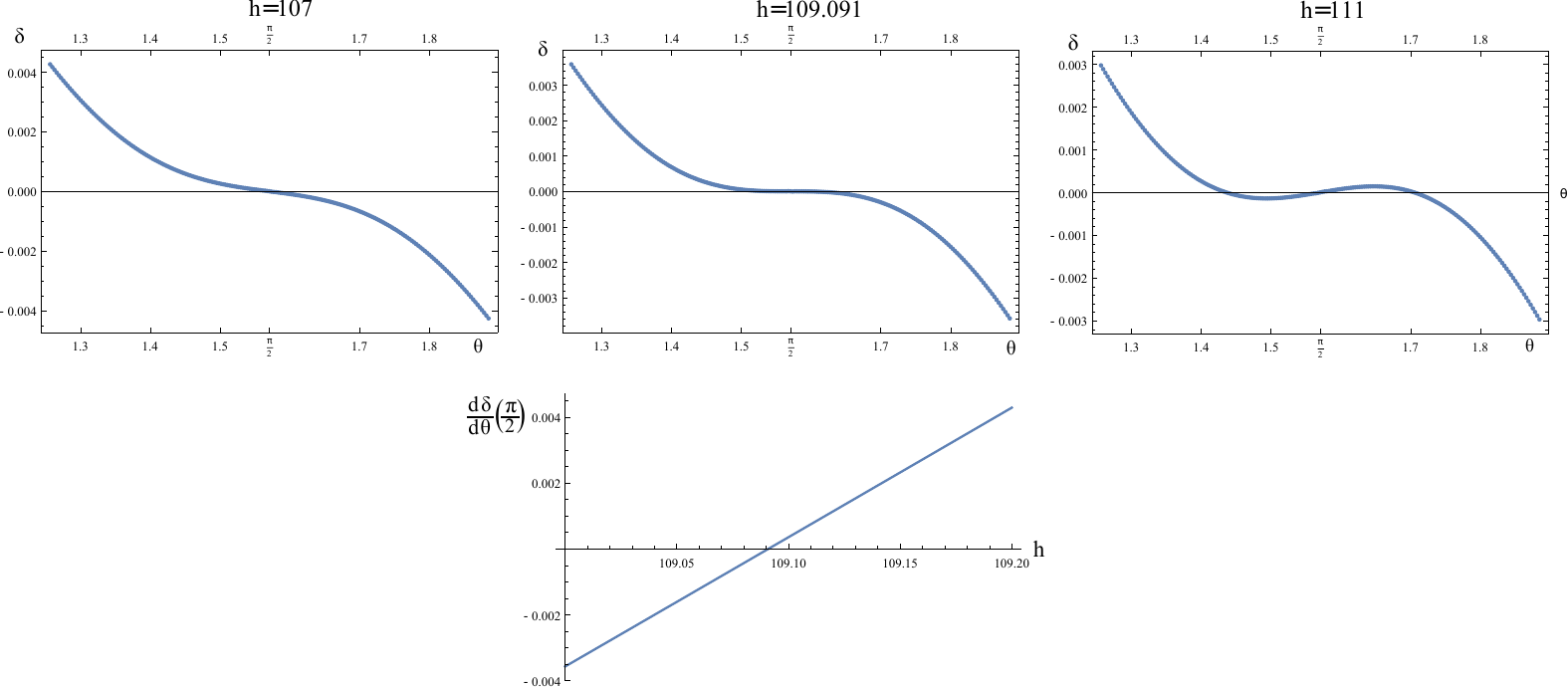}
	\caption{Plot of the free fall map (top) and its derivative in $\pi/2$ (bottom) as a function of $h$. The other parameters are $\E=2.5, \omega=\sqrt{2}, \mu=2, e=0.1.$ }
	\label{fig:free-fall-grafici}
\end{figure}
The direct study of the Free Fall map corroborates these findings. As a matter of fact, Figure \ref{fig:free-fall-grafici} shows the plots of $\delta(\theta)$ in a neighborhood of $\theta=\pi/2$ for different values of $h$ (before $h_{bif}$, at $h_{bif}$ and after it), along with the value of $\delta'(\pi/2)$ as a function of $h$: as one can see, before the bifurcation value the free fall map is strictly decreasing, while for $h=h_{bif}$ it has an inflection point with zero derivative at $\pi/2$. After the bifurcation value, two zeros, corresponding precisely to the brake orbits values of $\theta$, appear.

\end{esempio}

\section{Numerical simulations}\label{ssec: numerics}
	As already pointed out in Section \ref{ssec: brake}, the validity of the analytical investigations can be corroborated by a direct comparison with the plots of the map $F$ in specific cases, which highlights the variety of the behavuiours of the dynamics for different values of the involved parameters.\\
	This Section aims to gather cases of interest for the dynamics, underlying the effective role of the bifurcations in the change of stability and the subsequent arising or disappearence of new periodic points for $F$, as well as the potential presence of diffusive \textcolor{black}{orbits}, that represents a strong signal of caoticity.  \\
	All the below simulations are performed by considering $D$ as an ellipse centered in the origin, with semiaxes $a=1$ and $b=\sqrt{1-e^2}$, for different values of $e$. The routine is implemented in \textit{Mathematica}, and involves the numerical integration for the outer problem in its original form and, in order to avoid the numerical instability due to the presence of the possible singularity, of the inner problem in its regularised formulation.

\begin{figure}
	\centering
	\includegraphics[height=.23\textheight]{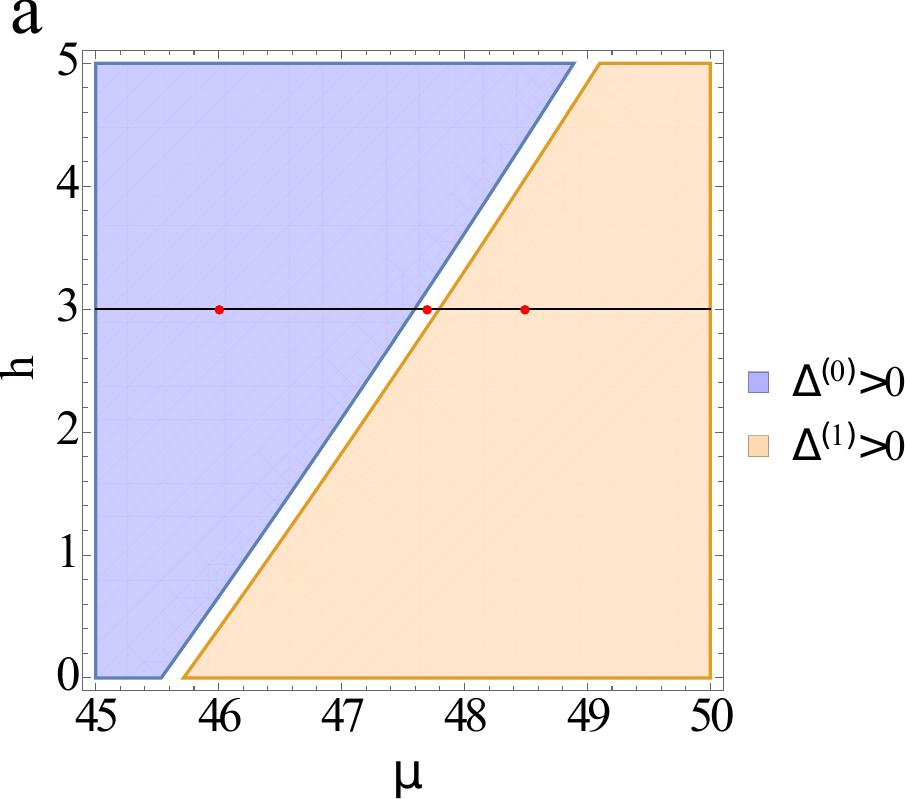}
	\includegraphics[height=0.23\textheight]{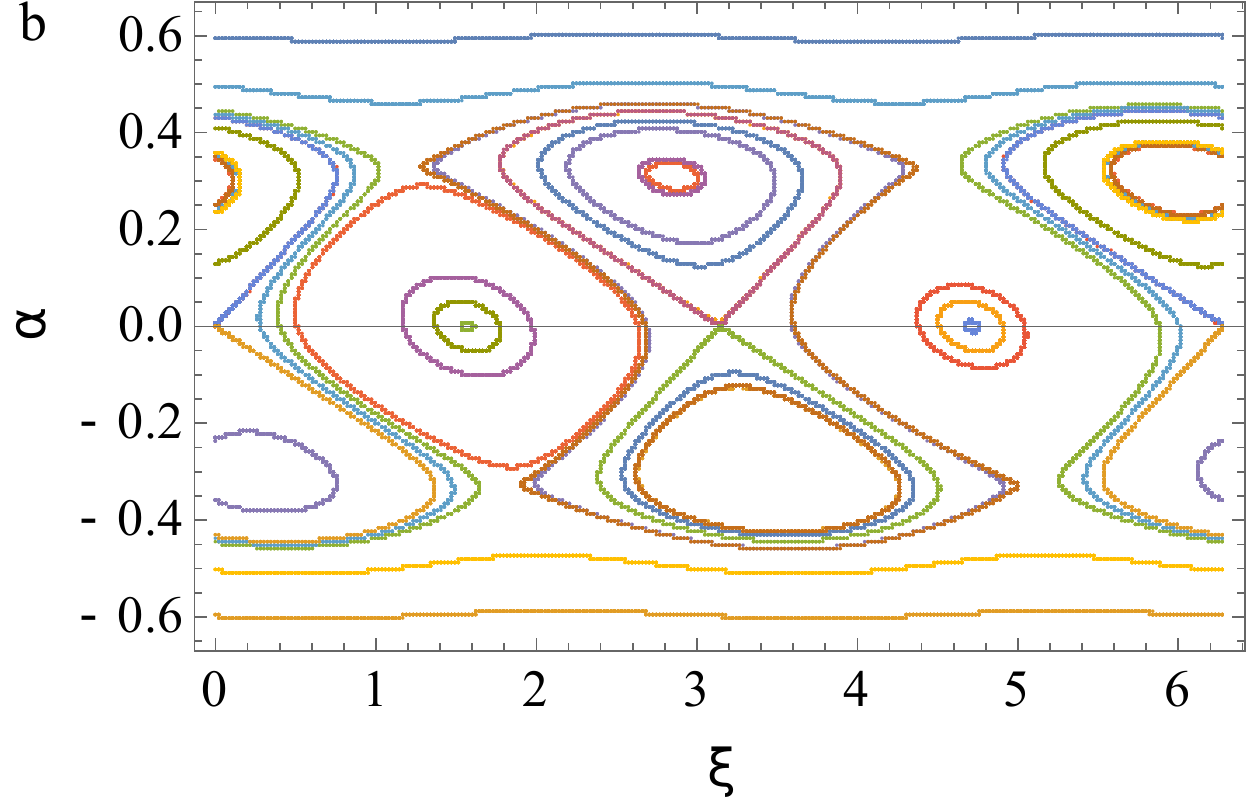}
	\includegraphics[height=0.23\textheight]{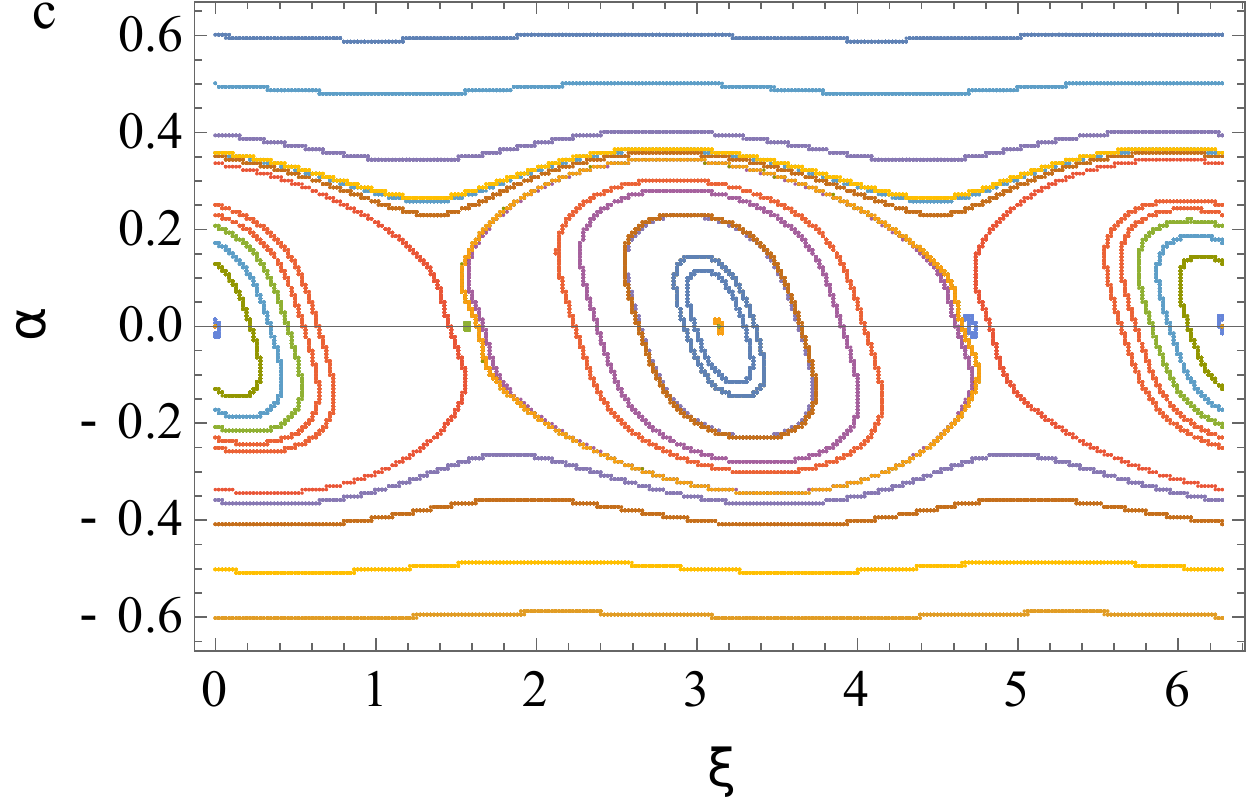}
	\includegraphics[height=0.23\textheight]{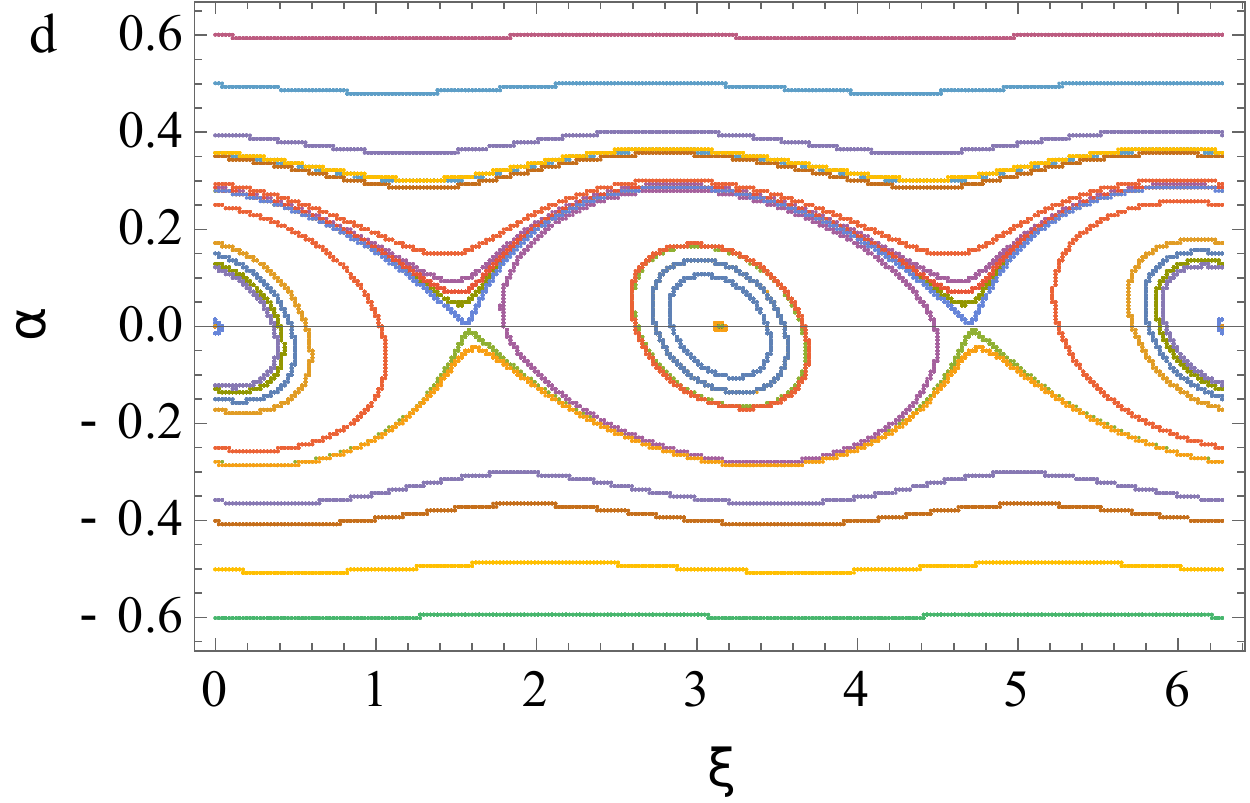}
	\caption{Bifurcations and plot of $F$ for a small eccentricity in a neigborhood of the axis $\alpha=0$. \textbf{(a)} Sign of $\Delta^{(0)}$ and $\Delta^{(1)}$ for $\E=9$, $\omega=1$, $e=0.03$ as a function of $h$ and $\mu$. The red dots correspond to $h=3$, $\mu=46$ \textbf{(b)}, $h=3$, $\mu=47.7$ \textbf{(c)} and $h=3$, $\mu=48.5$ \textbf{(d)}. }  
	\label{fig:e=0.03}
\end{figure}
	Figure \ref{fig:e=0.03}  shows the transition of the map through different stability regimes as the parameters modify the sign of $\Delta^{(0)}$ and $\Delta^{(1)}$. The changes of stability of $(0,0)$ between \textbf{(b)} and \textbf{(c)} and of $(\pi/2,0)$ between \textbf{(c)} and \textbf{(d)} are consistent with the plot of the discriminants scketched in \textbf{(a)}. In this case, where the eccentricity is small and the parameters are not much different from each other, the maps results to be regular also in the vicinity of the fixed points. We observe that in all the considered cases, for high values of $\alpha$ the map results essentially in a rotation on the ellipse, with small oscillations in $\alpha$. A noticeable fact is represented by the complexive number of stable and unstable equilibria in each regime, which is the same even in the case of generation of non-trivial fixed points for $\alpha\neq0$: 
	\begin{itemize}
		\item in the case \textbf{(b)} the saddle nature of $(0,0)$ and $(0,\pi)$ give rise to four non-homotetic stable fixed points, whose presence are balanced by four non-homotetic saddles in the vicinity of $(\pi/2,0)$ and $(3\pi/2,0)$; 
		\item in the case \textbf{(c)}, all the homotetic fixed points result to be stable; although the stable equilibrium points generated by the saddles in $(0,0)$ and $(\pi, 0)$ disappear with their change of stability, the saddles near to $(\pi,0)$ and $(3\pi/2,0)$ still remain, leading to four stable and four unstable points; 
		\item in the case \textbf{(d)}, the stability of the homotetic points is balanced, and no other equilibrium points are detected. 
	\end{itemize}
This non-trivial fact is coherent with the results one can obtain by applying the theory of the topological degree to the study of the stability of the fixed points in a discrete dynamical system (cf. \textcolor{black}{\cite{neumann1977generalizations}}), although the rigorous application of such theory would require the good definition and non-degeneration of $F$ on the whole ellipse. \\
In view of the approximation given in Section \ref{ssec: ecc piccola}, it is reasoneable to think that for small eccentricities and small values of the physical parameters the dynamics induced by $F$ does not differ much to the one sketched in Figure \ref{fig:e=0.03}. Nevertheless, when de ellipse becomes more eccentric or the parameters differ much from each others, a variety of behaviours can manifest, including the presence of diffusive \textcolor{black}{orbits}, that are strong indicators of chaos. 

\begin{figure}
	\centering
	\includegraphics[height=.22\textheight]{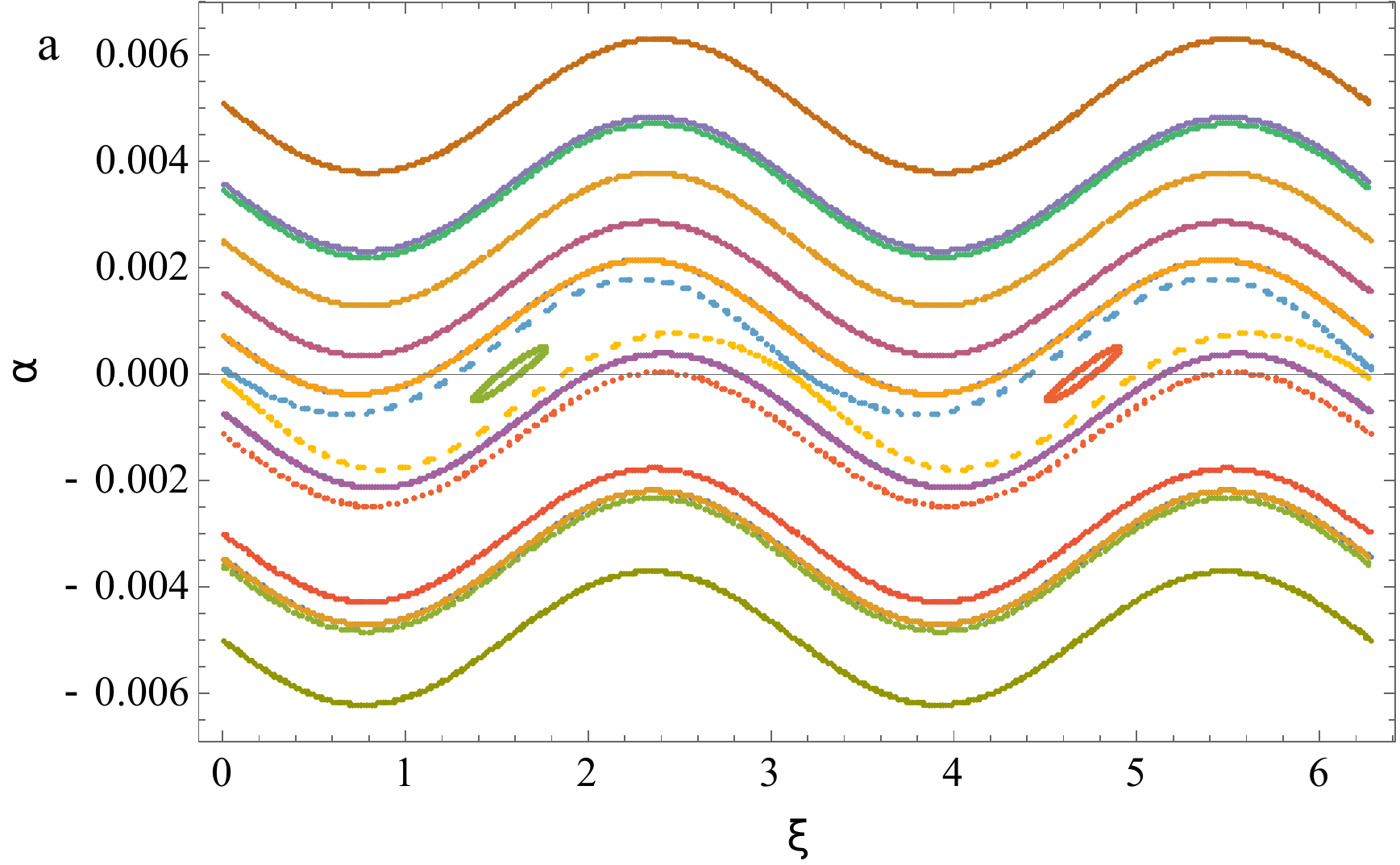}
	\includegraphics[height=.22\textheight]{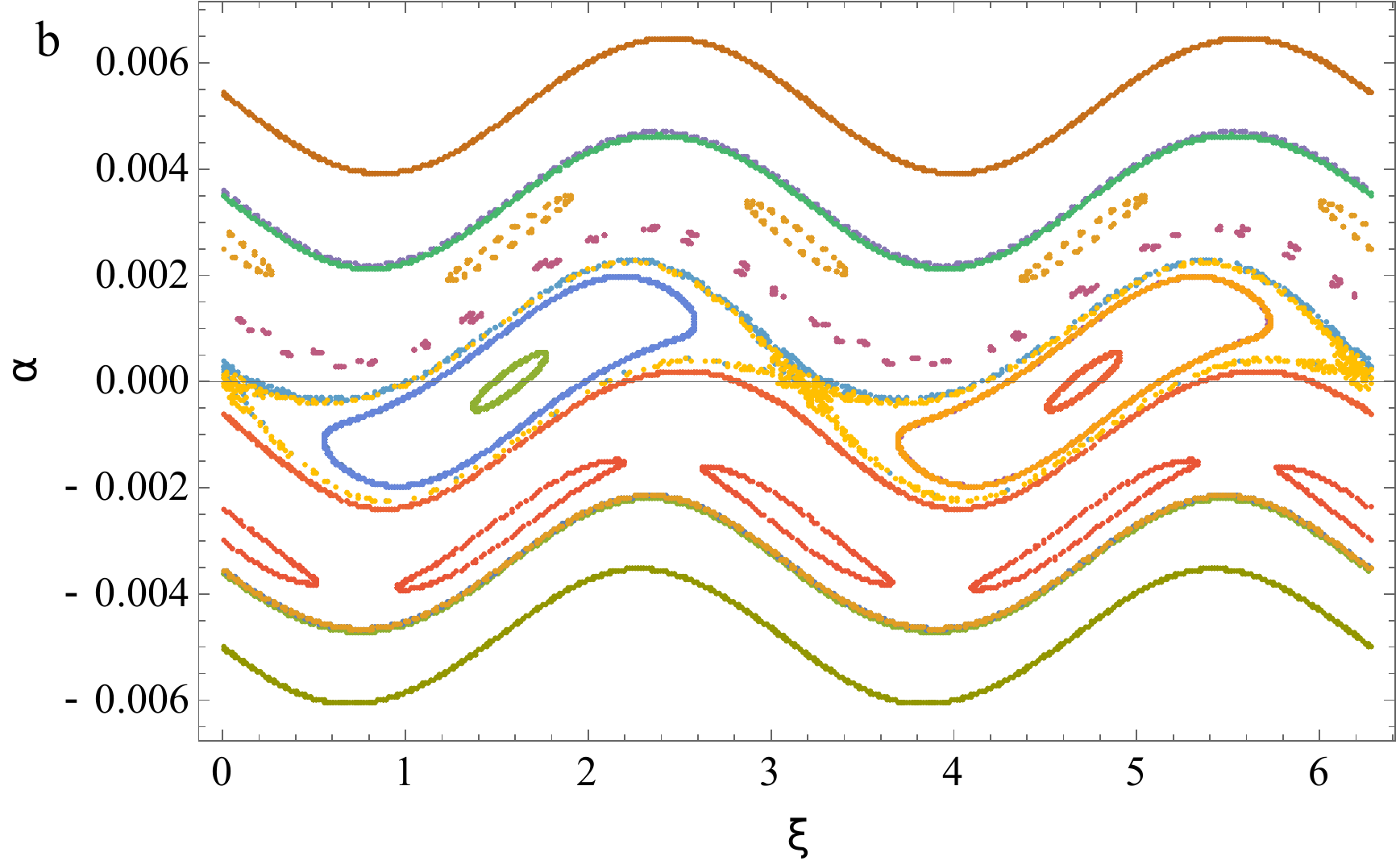}
	\includegraphics[height=.22\textheight]{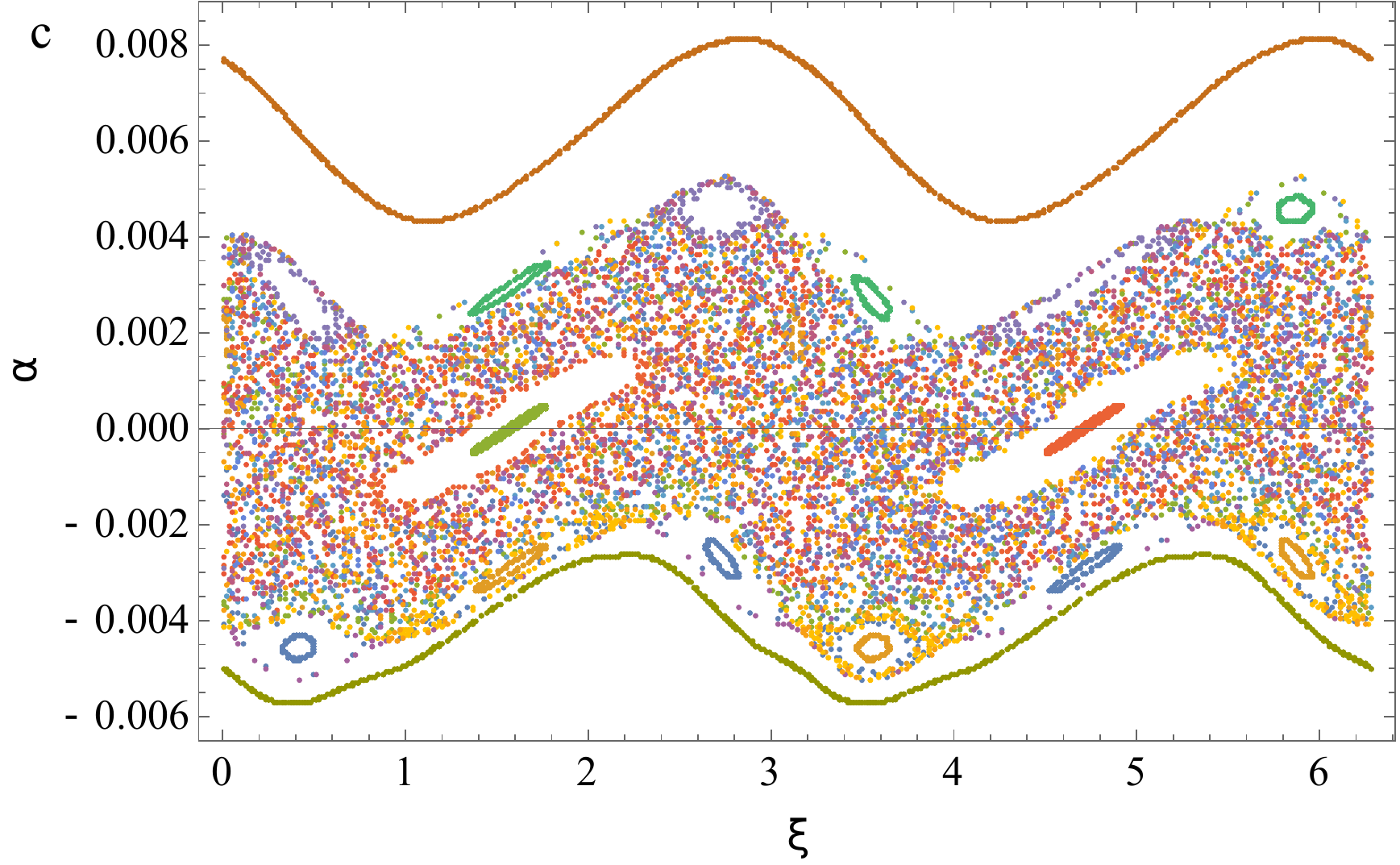}
	\caption{Plots of $F$ for $e=0.05$, $\E=20, \omega=1, \mu=0.13$ and $h=1$ \textbf{(a)}, $h=10$ \textbf{(b)}, $h=40$ \textbf{(c)}. The chaotic behaviour around $(0,0)$ and $(\pi, 0)$ is evident even for very small eccentricities. }
	\label{fig:e05-h2}
\end{figure}
	
Figure \ref{fig:e05-h2} shows the transition of $F$ for $e=0.05$, $\E=20, \omega=1, \mu=0.13$ and $h=1$ \textbf{(a)}, $h=10$ \textbf{(b)}, $h=40$ \textbf{(c)}, namely, for $e$ very small but with a high difference in magnitude between $h$ and $\mu$. In the considered regime, direct computations assure that $\Delta^{(0)}>0$ and $\Delta^{(1)}<0$, leading to the conclusion $(\pi/2,0)$ is a center and $(0, 0)$ is an unstable saddle. For increasing values of $h$, the saddle \textcolor{black}{orbits} around $(0,0)$ tend to diffuse, leading finally a chaotic cloud which surrounds the two stability islands. As in the case of Figure \ref{fig:e=0.03}, the chaotic region is bounded by invariant curves which induce oscillating rotations on the ellipse. Furthermore, periodic orbits of period 4 \textbf{(b)} and 3 \textbf{(c)} are detectable. \\
The other factor which can induce chaotic behhaviour is the increasing eccentricity of the domain. 

\begin{figure}
	\centering
	\includegraphics[height=.23\textheight]{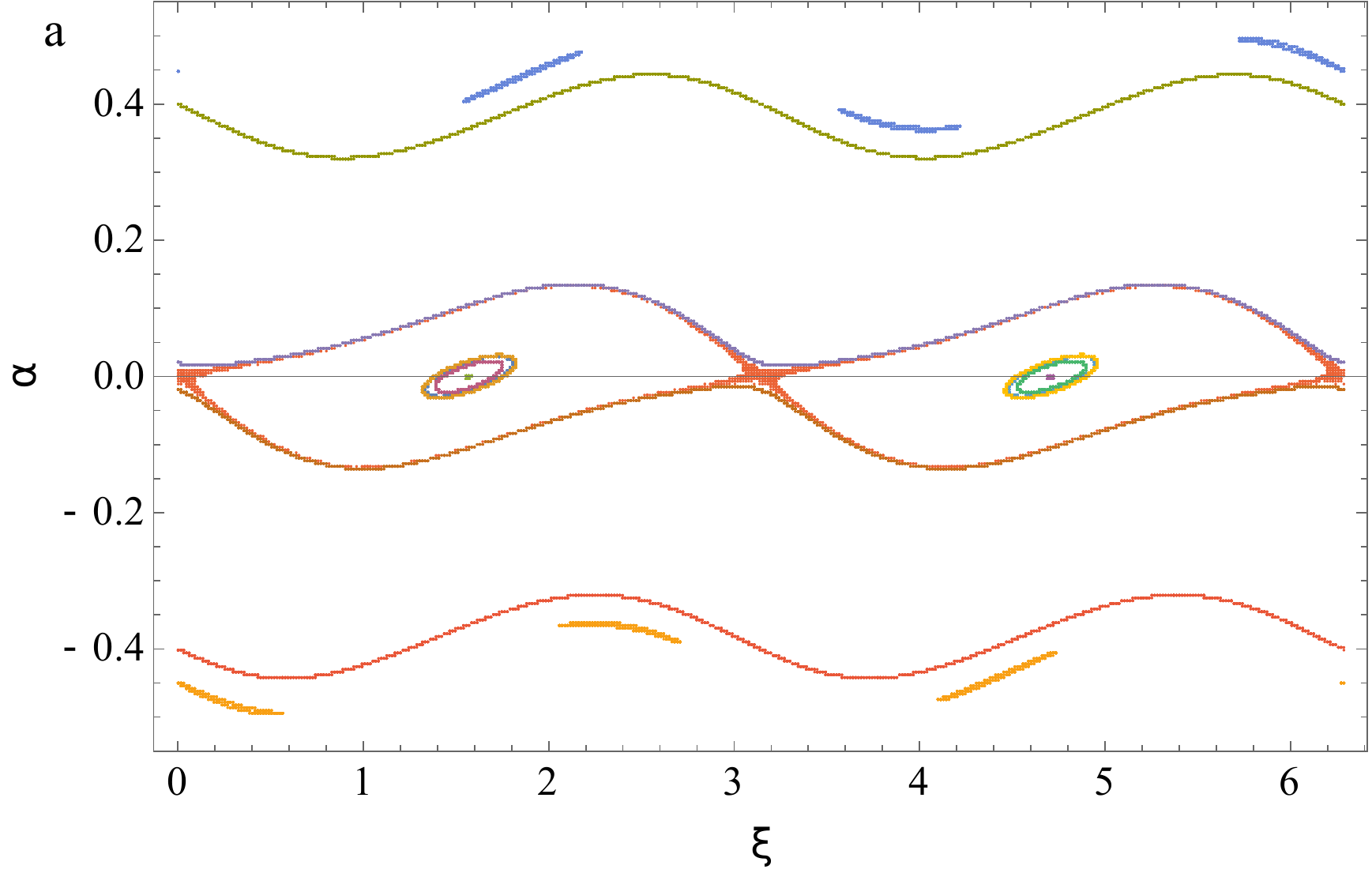}
	\includegraphics[height=.23\textheight]{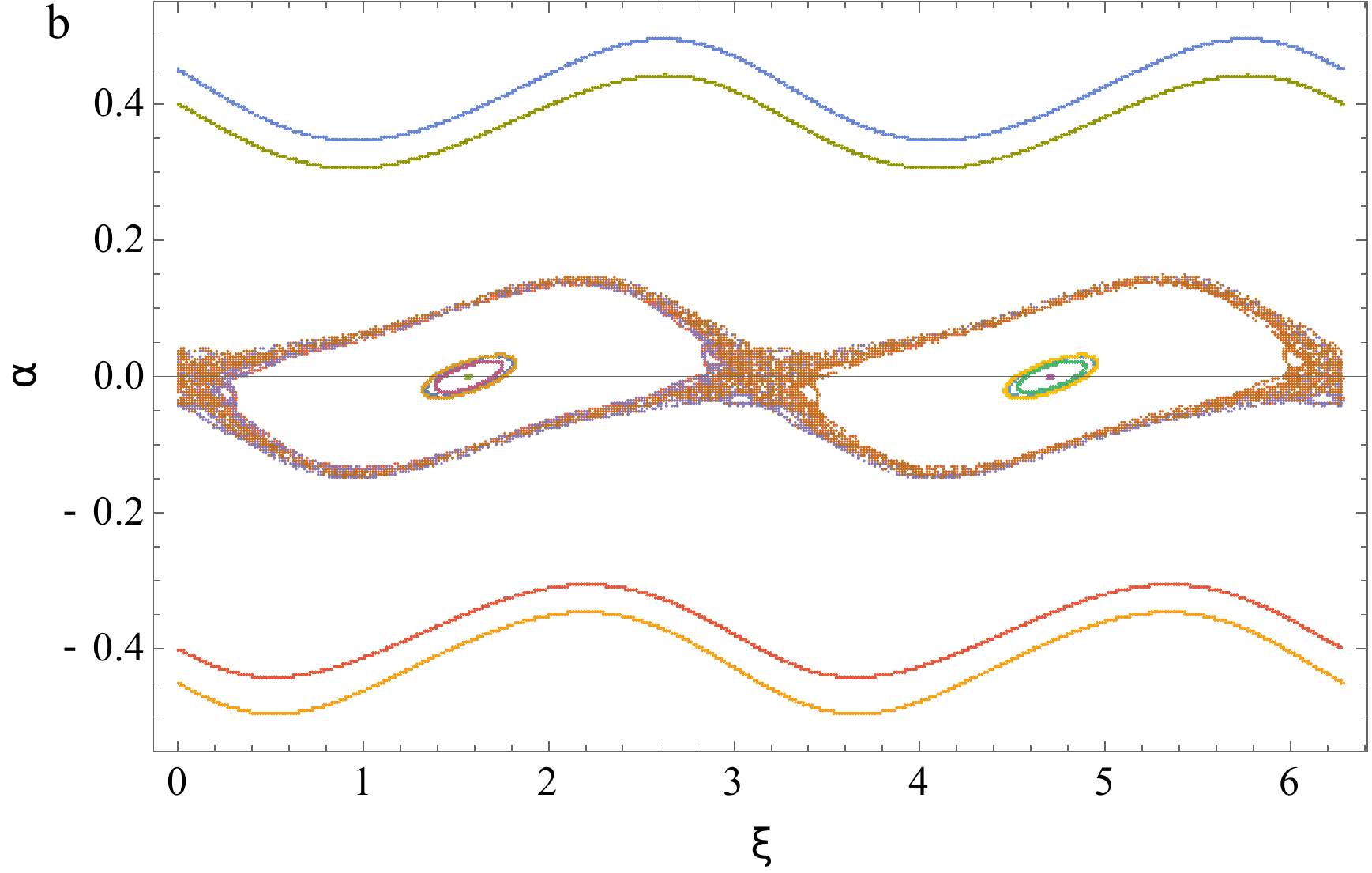}
	\includegraphics[height=.23\textheight]{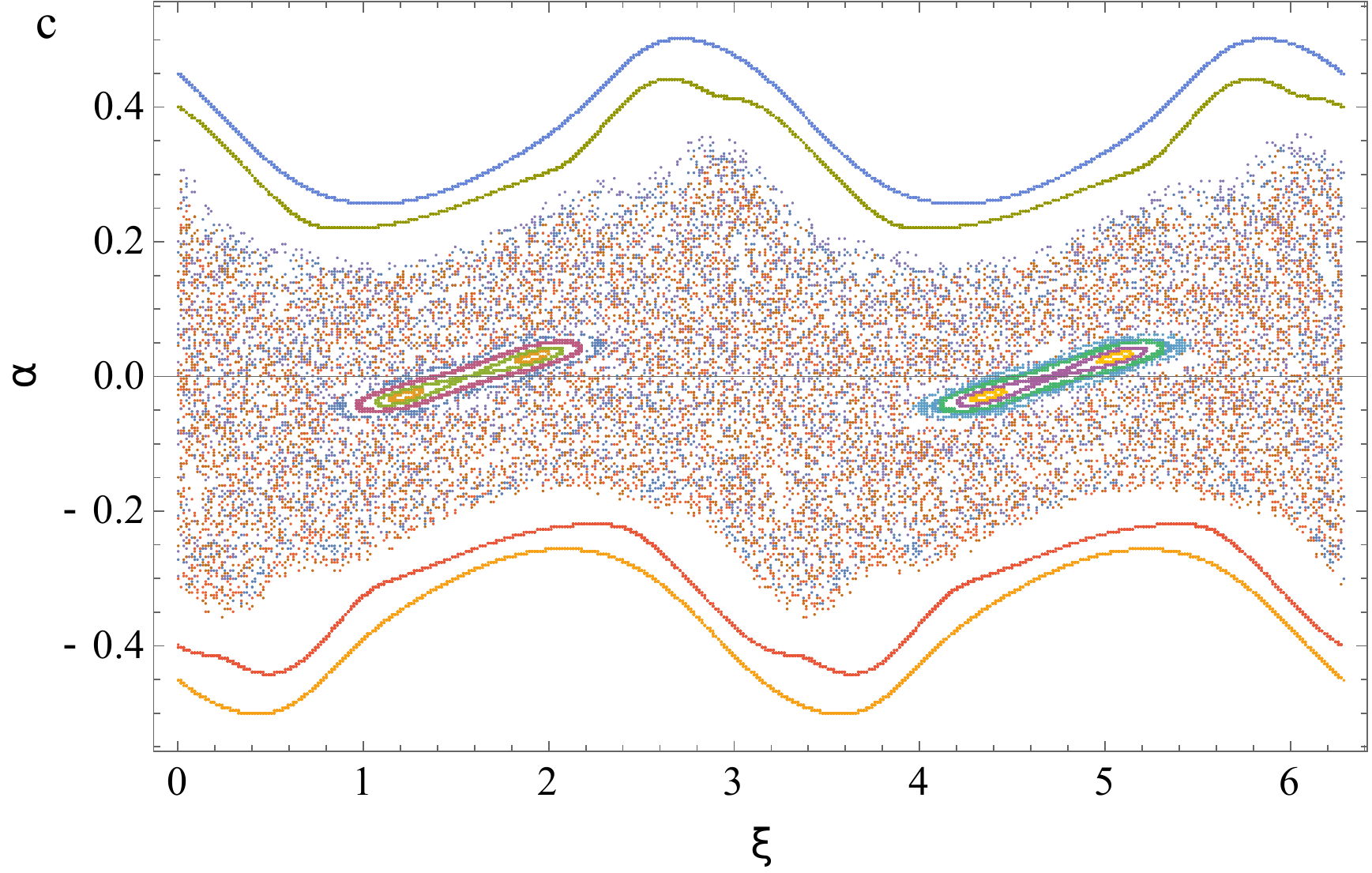}
	\includegraphics[height=.23\textheight]{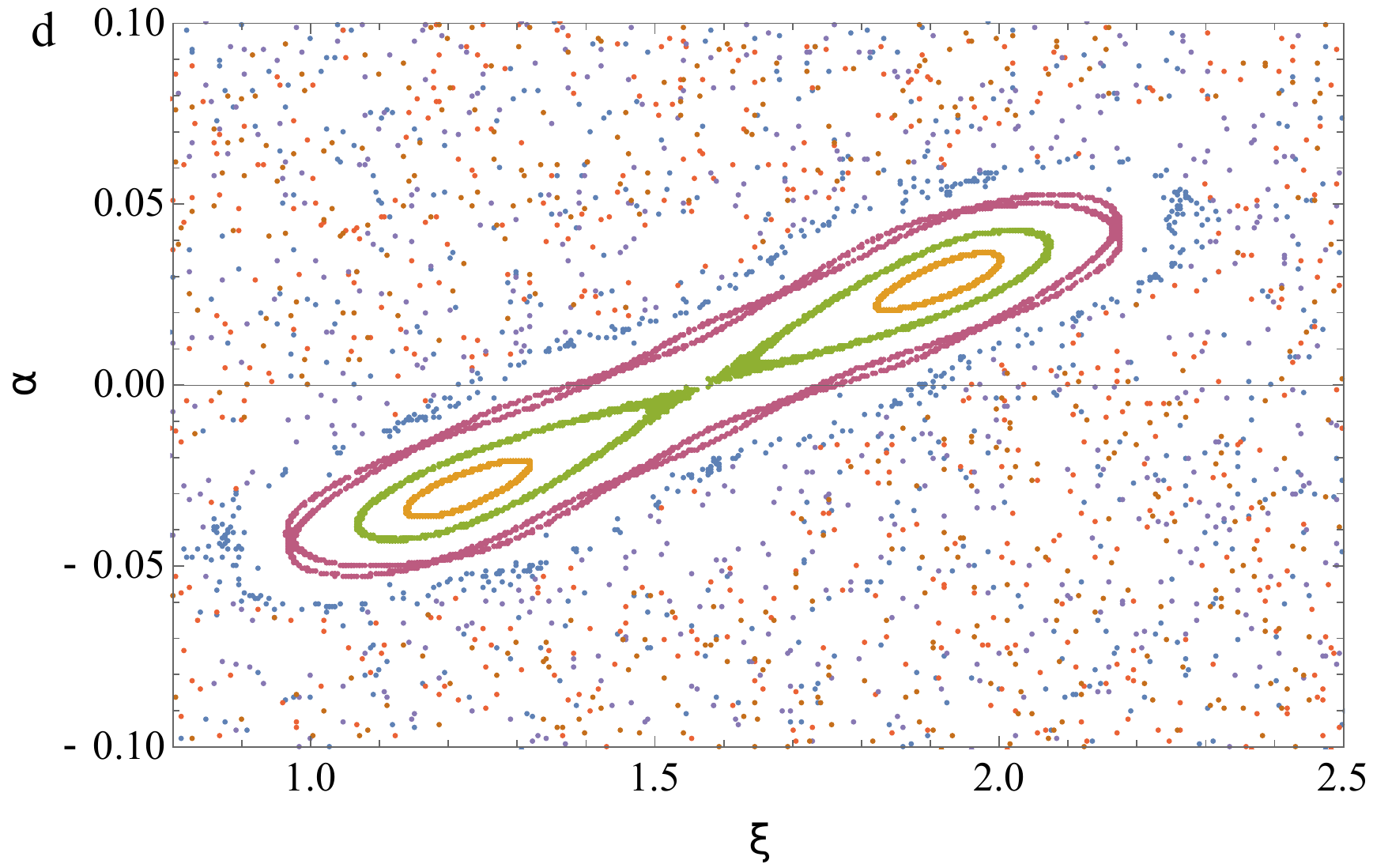}
	\caption{Plots of $F$ for $e=0.3$, $\E=2.5, \omega=\sqrt{2}, \mu=1$ and $h=0.1$ \textbf{(a)}, $h=1$ \textbf{(b)}, $h=7$ \textbf{(c)}. \textbf{(d)}: refining of \textbf{(c)} in a neighborhood of $(\pi/2,0)$. }
	\label{fig:e3-h-01}
\end{figure}

Figure \ref{fig:e3-h-01} illustrates how for moderate values of the eccentricity the system could have diffusive \textcolor{black}{orbits} around the unstable fixed points, even for small values of the physical parameters. This case is analogous to Figure \ref{fig:transizione-brake}, where the transition of $(\pi/2,0)$ from center to saddle produce two 2-periodic brake orbits. \\

\appendix
\section{Preliminaries: the variational approach and the generalized Snell's law}\label{sec:appA}
\subsection{The variational approach}
\textcolor{black}{Recalling the definitions given in Section \ref{sec:variational}}, let us define the quantity 
	\begin{equation}\label{lambda}
		\lambda^2=\frac{\int_0^1V(z(t))dt}{\frac{1}{2}\int_0^1\vert \dot{z}(t)\vert^2dt}>0.
	\end{equation}
As the following Proposition shows, the critical points of the Maupertuis functional $M$ are reparametrisations of solutions of (\ref{general prob}).
\begin{prop}\label{prop Mau}
	Let $z(t)\in H_{z_0,z_1}$ be such that $M(z)>0$ and $\forall v\in H^1_0([0,1],\mathbb{R}^2)$ $dM(z)[v]=0$. Then $z(s)=u(t(s))$ with $t(s)=\lambda s$ is a classical solution of 
	\begin{equation}\label{prop sys}
		\begin{cases}
			\ddot{z}(t)=\nabla V(z(t)) \quad &t\in[0,1/\lambda]\\\frac{1}{2}\vert \dot{z}(t)\vert^2-V(z(t))=0 &t\in[0,1/\lambda]\\
			z(0)=z_0, \quad z(1/\lambda)=z_1.
		\end{cases}
	\end{equation}
\end{prop}
\begin{proof}
	$\forall v\in H_0^1([0,1],\mathbb{R}^2)$ one has that
	\begin{equation*}
		0=\int_0^1<\dot{z}(t),\dot{v}(t)>dt\int_0^1V(z(t))dt+\frac{1}{2}\int^1_0\vert\dot{z}(t)\vert^2dt\int_0^1<\nabla V(z(t)),v(t)>dt.
	\end{equation*}
	Since $M(z)>0\Rightarrow$ we can divide for $\frac{1}{2}\int_0^1\vert\dot{z}(t)\vert^2 dt>0$ and obtain
	\begin{equation*}
		0=\lambda^2\int_0^1<\dot{z}(t),\dot{v}(t)>dt\int_0^1V(z(t))dt+\int_0^1<\nabla V(z(t)),v(t)>dt,
	\end{equation*}
	then $z(t)$ is a weak solution (and, by regularity, strong), of $\lambda^2\ddot{z}(t)=\nabla V(z(t))$. Reparametrising $z(s)=z(t(s))$ with $t(s)=\lambda s$, and defining $'=d/ds$, one has $z''(s)=\nabla V(z(s))$ for $s\in[0,1/\lambda]$. As for the energy conservation, from $\lambda^2\ddot{u}(t)=\nabla V(u(t))$ there is $k\in\mathbb{R}$ such that 
	\begin{equation}\label{kdef}
		\frac{\lambda^2}{2}\vert\dot{z}(t)\vert^2-V(z(t))=k, 
	\end{equation}
	and, comparing (\ref{lambda}) with (\ref{kdef}), one obtains $k=0$.
\end{proof}
 
\begin{rem}\label{L^2=2M}
	By H\"older inequality, 
	\begin{equation*}
		L^2(z)\leq\int_0^1\vert\dot{z}(t)\vert^2V(z(t))dt=2M(z), 
	\end{equation*}
	on the other hand, if $z(t)$ is a solution of (\ref{prop sys}), from \ref{kdef} we have that $V(z(t))=\frac{\lambda^2}{2}\vert\dot{z}(t)\vert$ for all $t\in[0,1]$, then the equality holds and  $L^2(z)=2M(z)$. This means that, in order to find solutions of (\ref{general prob}), finding critical points of $M(z)$ at a positive level is equivalent to finding critical points of $L(z)$.
\end{rem}

\begin{rem}[Relation between geodesic and cinetic time]\label{oss tempi}
	Suppose that $z(t)\in H_0$ minimizes $L(z)$, then from Remark \ref{L^2=2M} it minimizes $L^2=2M$ at a positive level. From the energy conservation, one has that $\mathcal{L}=\vert\dot{z}(t)\vert^2V(z(t))=const>0,$ then $L=\vert\dot{z}(t)\vert\sqrt{V(z(t))}=const>0.$ Moreover, as $z(t)$ minimizes $M$, the Euler-Lagrange equation holds for the Lagrangian function $\mathcal{L}$: 
	\begin{equation*}	
		0=\frac{d}{dt}\left(\frac{\partial}{\partial \dot{z}}\mathcal{L}\right)-\frac{\partial}{\partial z}\mathcal{L}=\frac{d}{dt}\left(2\dot{z}(t)V(z(t))\right)-\vert \dot{z}(t)\vert^2\nabla V(z(t)).	
	\end{equation*}
	Consider now the reparametrisation $t=t(s)$ such that 
	\begin{equation*}
		\frac{d}{dt}=\frac{L}{\sqrt{2}V(z(t(s)))}\frac{d}{ds},
	\end{equation*}
	then
	\begin{equation*}
		\begin{split}
			0=&\frac{L}{\sqrt{2}V(z(t(s)))}\frac{d}{ds}\left(2V(z(t(s)))\frac{L}{\sqrt{2}V(z(t(s)))}\frac{d}{ds}z(t(s))\right)-\frac{L^2}{V(z(t(s)))}\nabla V(z(t(s)))\\
			&\Leftrightarrow z''(s)=\nabla V(z(s)),
		\end{split}
	\end{equation*}
	namely $s$ is the cinetic time. One can compute the cinetic period $T$ as 
	\begin{equation*}
		T=\int_0^T ds=\int_0^1 \frac{L}{\sqrt{2}V(z(t))}dt=\frac{1}{\sqrt{2}L}\int_0^1\vert\dot{z}(t)\vert dt,
	\end{equation*}
	while the relation between $L$ and the Lagrangian action $\mathcal{A}$ is given by
	\begin{equation*}
		\mathcal{A}=\int_0^T\left(\frac{1}{2}\vert z'(s)\vert^2+V(z(s))\right)ds=\int_0^1(2V(z(t)))\frac{L}{\sqrt{2}V(z(t))}dt=\sqrt{2}L.	\end{equation*}
\end{rem}

\subsection{The generalized Snell's law}\label{ssec: appA2}

Let us suppose that $\mathbb{R}^2$ is divided into two regions by an interface $\Sigma$, defined as the trace of a regular curve $\sigma(\xi)$, $\xi\in I\subset\mathbb{R}$. Let us call $A$ and $B$ the two regions of $\mathbb{R}^2$, such that $\mathbb{R}^2=A\cup B\cup\Sigma$. Suppose that in $A$ and in $B$ two generic Riemannian metrics are defined, such that 
\begin{equation}\label{psA}
	\begin{split}
		\forall z\in A,\text{}\forall v,w\in T_z\mathbb{R}^2\quad &<v,w>_{A,z}=\sum_{i,j=1}^2a_{ij}(z)v_iw_j, \\
		\forall z\in B,\text{}\forall v,w\in T_z\mathbb{R}^2\quad &<v,w>_{B,z}=\sum_{i,j=1}^2b_{ij}(z)v_iw_j,
	\end{split}
\end{equation}
\textcolor{black}{where the coefficients $a_{ij}$ and $b_{ij}$ are of class $C^2$: therefore, we can assume without loss of generality that for any pair of points in $\bar A$ (respectively in $\bar B$) there is a unique geodesic which connects them, provided they are sufficiently close.} We can define the respective Jacobi lengths and distances (from now on, all the indices in the sums go from $1$ to $2$):
\begin{equation*}
	\begin{split}
		L_A(\alpha)=\int_0^1\vert\dot{\alpha}(t)\vert_{A,\alpha(t)} dt=\int_0^1\sqrt{\sum_{i,j}a_{ij}(\alpha(t))\dot{\alpha}_i(t)\dot{\alpha}_j(t)}dt\\
		L_B(\alpha)=\int_0^1\vert\dot{\alpha}(t)\vert_{B,\alpha(t)} dt=\int_0^1\sqrt{\sum_{i,j}b_{ij}(\alpha(t))\dot{\alpha}_i(t)\dot{\alpha}_j(t)}dt\\
		d_{A\backslash B}(z_0,z_1)=\min\{L_{A\backslash B}(\alpha)\quad\vert\quad\alpha(0)=z_0,\text{}\alpha(1)=z_1\}
	\end{split}
\end{equation*}
Fixed $z_A\in A$ and $z_B\in B$, we want to find $\bar{z}\in\Sigma$ such that 
\begin{equation}\label{stat}
	d_A(z_A,\bar{z})+d_B(\bar{z},z_B)=\min_{z\in\Sigma}\left(d_A(z_A,z)+d_B(z,z_B)\right). 
\end{equation}
This means that $\bar{z}$ fulfills the stationarity condition
\begin{equation*}
	\forall e\in T_z\Sigma\quad \partial_{e,z}(d_A(z_A,z)+d_B(z,z_B))=\partial_{e,w}d_A(z_A,z)+\partial_{e,v}d_B(z,z_B)=0, 
\end{equation*}
where: \begin{itemize}
	\item  $\partial_{e,v}d(z_0,z_1)=\frac{d}{d\epsilon}d(z_0+\epsilon e,z_1)_{\vert_{\epsilon=0}}$
	\item $\partial_{e,w}d(z_0,z_1)=\frac{d}{d\epsilon}d(z_0,z_1+\epsilon e)_{\vert_{\epsilon=0}}$
\end{itemize}
To compute the directional derivatives of $d_A$, suppose that the curve $\alpha^A(t;z_0,z_1)$ minimizes $L_A(\alpha)$ over all the curves $\alpha(t)$ such that $\alpha(0)=z_0$ and $\alpha(1)=z_1$;  from the minimization properties of $\alpha^A$, $d_A(z_A,z)=L_A(\alpha^A)$, so that $\partial_{e,w}d_A(z_A,z)=\partial_{e,w}L_A(\alpha^A)$. Generalizing Remark \ref{L^2=2M}, $\alpha^A$ solves the Euler-Lagrange equation with $\mathcal{L}=\sum_{i,j}a_{i,j}(\alpha^A(t))\dot{\alpha}^A_i(t)\dot{\alpha}^A_j(t)$, namely, for $k=1,2$, 
\begin{equation}\label{E-L snell}
	0=-\frac{d}{dt}\left(2\sum_ia_{ik}(\alpha^A)\dot{\alpha}^A_i\right)_k+\left(\sum_{i,j}\frac{\partial a_{ij}}{\partial x_k}(\alpha^A)\dot{\alpha}^A_i\dot{\alpha}^A_j\right)_k. 
\end{equation}
Define $e$ as the unit vector tangent to $\Sigma$ in $z$ (namely, $e=\dot{\sigma}(\bar{\xi})/\vert\dot{\sigma}(\bar{\xi})\vert$, with $\sigma(\bar{\xi})=z$) and consider 
\begin{equation*}
	\partial_{e,w}\alpha^A(t)=\partial_{e,w}\alpha^A(t;z_0,z_1)=\frac{d}{d\epsilon}\alpha^A(t;z_0,z_1+\epsilon e)_{\vert_{\epsilon=0}}. 
\end{equation*} 
One can easily observe that $\partial_{e,w}\alpha^A(0)=\boldsymbol{0}$ and $\partial_{e,w}\alpha^A(1)=e$. Computing the directional derivative of $L^2_A(\alpha^A)$, one obtains 
\begin{equation}\label{derL}
	\begin{split}
		2L_A(\alpha^A)\partial_{e,w}L_A(\alpha^A)&=\partial_{e,w}L^2_A(\alpha^A)=\partial_{e,w}\int_0^1\sum_{i,j}a_{ij}(\alpha^A)\dot{\alpha}^A_i\dot{\alpha}^A_j=\\
		%&=\int_0^1\partial_{e,w}\sum_{i,j}a_{ij}(\alpha^A)\dot{\alpha}^A_i\dot{\alpha}^A_j=\\
		&=\int_0^1\sum_{i,j,k}\frac{\partial a_{ij}}{\partial x_k}(\alpha^A)\dot{\alpha}^A_i\dot{\alpha}^A_j\partial_{e,w}\alpha^A_k+2\sum_{i,j}a_{ij}(\alpha^A)\dot{\alpha}^A_i\partial_{e,w}\dot{\alpha}^A_j.
	\end{split}
\end{equation}
Multiplying (\ref{E-L snell}) with $\partial_{e,w}\alpha^A=(\partial_{e,w}\alpha^A_k)_k$ and integrating for $t\in[0,1]$:
\begin{equation*}
	\begin{split}
		0&=\sum_k\int_0^1\left[-\frac{d}{dt}\left(2\sum_ia_{ik}(\alpha^A)\dot{\alpha}^A_i\right)\partial_{e,w}\alpha^A_k+\sum_{i,j}\frac{\partial a_{ij}}{\partial x_k}(\alpha^A)\dot{\alpha}^A_i\dot{\alpha}^A_j\partial_{e,w}\alpha^A_k\right]=\\
		%&=\left[-2\sum_{i,k}a_{ik}(\alpha^A)\dot{\alpha}^A_i\partial_{e,w}\alpha^A_k\right]_0^1+\\
		%&\quad\quad\quad\quad\quad+\left[2\sum_{i,k}a_{ik}(\alpha^A)\dot{\alpha}^A_i\partial_{e,w}\dot{\alpha}^A_k+\sum_{i,j,k}\frac{\partial a_{ij}}{\partial x_k}(\alpha^A)\dot{\alpha}^A_i\dot{\alpha}^A_j\partial_{e,w}\alpha^A_k\right]=\\
		&=-2\sum_{i,k}a_{ik}(\alpha^A(1))\dot{\alpha}^A_ie_k+2L_A(\alpha^A)\partial_{e,w}L_A(\alpha^A),
	\end{split}
\end{equation*}
where, in the last equality, we used (\ref{derL}) and the fact that $\partial_{e,w}\alpha^A(1)=e$. Recalling (\ref{psA}), and repeating the analogous computations for $\partial_{e,v}d_B(z,z_B)$, one obtains
\begin{equation}\label{dirder}
	\partial_{e,w}d_A(z_A,z)=\frac{<\dot{\alpha}^A(1),e>_{A,z}}{\vert\dot{\alpha}^A(1)\vert_{A,z}},\quad\partial_{e,v}d_B(z,z_B)=-\frac{<\dot{\alpha}^B(0),e>_{B,z}}{\vert\dot{\alpha}^B(0)\vert_{B,z}}.
\end{equation}
Comparing (\ref{stat}) and (\ref{dirder}), one can finally find the junction condition
\begin{equation}\label{snell law}
	\frac{<\dot{\alpha}^A(1),e>_{A,\bar{z}}}{\vert\dot{\alpha}^A(1)\vert_{A,\bar{z}}}=\frac{<\dot{\alpha}^B(0),e>_{B,\bar{z}}}{\vert\dot{\alpha}^B(0)\vert_{B,\bar{z}}},
\end{equation}
which can be intepreted as a conservation law for the tangential component of the velocity vector before and after crossing the interface $\Sigma$. \\
Looking at the potentials defined in (\ref{potential}), the two metrics are expressed by 
$a_{ij}(z)=V_I(z)\delta_{ij}$ and $b_{ij}(z)=V_E(z)\delta_{ij}$: in this case, in view of the regularity of both the potentials far from the origin, the uniqueness of the geodesics is guaranteed by the existence of two suitable strongly convex neighborhoods of $z_A$ and $z_B$, if they are near enough to $\Sigma$. Denoting with $\cdot$ 
the scalar product in the Euclidean metric of $\mathbb{R}^2$, 
\begin{equation*}
	<v,w>_{A,z}=V_I(z)v\cdot w,\quad <v,w>_{B,z}=V_E(z)v\cdot w,  
\end{equation*} 
and, if we define $z_E(t)=\alpha^A(t)$ and $z_I(t)=\alpha^B(t)$, equation (\ref{snell law}) becomes
\begin{equation}\label{snell law 2}
	\sqrt{V_E(\bar{z})}\frac{\dot{z}_E(1)}{\vert\dot{z}_E(1)\vert}\cdot e=\sqrt{V_I(\bar{z})}\frac{\dot{z}_I(0)}{\vert\dot{z}_I(0)\vert}\cdot e. 
\end{equation}

\section{Local existence of inner and outer arcs}\label{appB}

This Appendix presents an extensive discussion, along with the proofs, of the framework which leads to the results enlisted in Section \ref{sec:variational}. 
%, where we assumed 
% that $D$ is contained in the Hill's region 
%\begin{equation*}
%	\mathcal{H}=\left\{p\in\R^2\text{ }\Bigg|\text{ }\sqrt{\E-\frac{\om}{2}|p|^2}>0\right\}, 
%\end{equation*}
%and that $\partial D$ is parametrised as the trace of a curve $\gamma:  I=[a,b]\mapsto\mathbb{R}^2 $, with $\gamma\in C^2(I)$. We consider neighborhoods ofpoints in $\gamma$ such that 
%\begin{equation}\label{cond Hom app}
%	\begin{aligned}
%		\bx\in I \text{ such that}:\text{ } &(i)\text{ }\gamma{(\bx)}\nparallel\dot{\gamma}{(\bx)}\\
%		&(ii)\text{ }\text{defined }s=t\gamma{(\bx)},\text{ } t\in[0,\infty),\text{ } supp(s)\cap\partial D=\{\gamma{(\bx)}\}. 
%	\end{aligned}
%\end{equation}

\subsection{Local existence of the outer arcs}
Recalling Notation \ref{notazione}, we want to ensure the existence, under some suitable hypotheses on the initial conditions, of solutions to the problem 
\begin{equation}\label{eq:outer_dynamics app}
	\begin{cases}
		(\Prob_E)[y(s)]& s\in[0,T]\\ 
		y(0)=y_0, y'(0)=\mathbf{v}_0	
	\end{cases}
\end{equation}
such that $y(T)\in\partial D$. 

\begin{lemma}\label{lemma esistenza outer app}
	Suppose that $\bx\in I$ satisfies condition (\ref{cond Hom}). Then there exist $\delta_0(\bar{\xi}), \delta_1(\bar{\xi}), \rho(\bar{\xi})>0$ such that for every $\xi_0\in I$, $\dot{\theta}_0\in\mathbb{R}$ with $\vert\bar{\xi}-\xi_0\vert<\delta_0(\bar{\xi})$ and $\vert\dot{\theta}_0\vert<\rho(\bar{\xi})$, defined the unit vectors (in exponential notation) $\hat{x}_1=\gamma({\bar{\xi}})/\vert\gamma(\bar{\xi})\vert$ and $\hat{x}_2=i\hat{x}_1$, there exist $T>0$ and $\xi_1\in I$ such that the problem 
	\begin{equation*}
		\begin{cases} 
			(\Prob_E)[y(s)] & s\in[0,T]\\ 
			y(0)=\gamma(\xi_0), y'(0)=\dot{r}_0 \hat{x}_1+\dot{\theta}_0 \hat{x}_2, 
		\end{cases}
	\end{equation*}
	with $\dot{r}_0=\dot{r}_0(\dot{\theta}_0)=\sqrt{2\mathcal{E}-\omega^2\vert\gamma(\xi_0)\vert-\dot{\theta}_0^2}$, admits the unique solution $y(s;\xi_0, \dot{\theta}_0)$ and $y(T;\xi_0, \dot{\theta}_0)=\gamma(\xi_1)\in\partial D$. Moreover, $\vert\bar{\xi}-\xi_1\vert<\delta_1(\bar{\xi})$. 
\end{lemma}

The proof relies on a transversality argument, standard in detecting one side Poincar\'e sections, based upon regularity of solutions of Cauchy's problems and the implicit function theorem (see, e.g. the similar construction in \cite{ST2012}).

\begin{rem}\label{rem appB}
	The validity of condition (\ref{cond Hom}(ii)) in a neigborhood of $\bx$ entails that the point $\gamma(\xi_1)$ defined as in Lemma \ref{lemma esistenza outer app} is such that, for every $s\in(0,T(\x0, \dot{\theta}_0))$, $y(s; \x0, \dot{\theta}_0)\notin \bar{D}$, that is, there are no other intersections of $\partial D$ and the arc $y([0,T(\x0,\dot{\theta}_0)], \x0,\dot{\theta}_0)$ other than  $\gamma(\x0)$ ad $\gamma(\x1)$. This is in fact a consequence the continous dependence on the initial conditions, for wich, if $(\x0,\dot \theta_0)$ are sufficiently close to $(\bx,0)$, then $y(\cdot;\x0,\dot \theta_0)$ is arbitrarily close to $y(\cdot;\bx,0)$ in the $C^0$ topology. 
\end{rem}
%\begin{oss}
The above Lemma states the existence of a local Poincar\'e section in the energy manifold of $\partial D\times \mathbb R^2$ for the outer dynamics  \eqref{eq:outer_dynamics app} in a neighbourhood of the initial condition of a radial brake orbit (namely, the direction of the velocity vector coincides with the radial one) and under some local conditions on $\partial D$. The condition $\vert\dot{\theta}_0\vert<\delta_0(\bar{\xi})\equiv\delta$ can be rephrased by considering the angle $\alpha\in[-\pi/2, \pi/2]$ between the initial velocity $y'(0)=\dot{r}_0\hat{x}_1+\dot{\theta}_0\hat{x}_2$ and the radial unit vector $\hat{x}_1$ of $\gamma(\bar{\xi})$ (notice that $\alpha$ and $\dot{\theta}_0$ have always the same sign). In particular, one has that 
\begin{equation*}
	\tan{\alpha}=\frac{\dot{\theta}_0}{\dot{r}_0}=\frac{\dot{\theta}_0}{\sqrt{2\mathcal{E}-\omega^2\vert\gamma(\x0)\vert^2-\dot{\theta}_0^2}} \Leftrightarrow \dot{\theta}_0=\tan{\alpha}\sqrt{\frac{2\mathcal{E}-\omega^2\vert\gamma(\x0)\vert^2}{1+\tan{\alpha}^2}}=f(\alpha, \x0).
\end{equation*}
As $f(\alpha,\x0)$ is continous and $f(0,\bx)=0$,  there exist $\epsilon_\alpha,\epsilon_{\x0}>0$ such that $\epsilon_{\x0}<\delta_0(\bx)$ and, if $\vert\alpha\vert<\epsilon$ and $|\bx-\x0|<\epsilon_{\x0}$, then $\vert\dot{\theta}_0\vert<\delta$. 
%\end{oss}
Taking together Lemma \ref{lemma esistenza outer app} and Remark \ref{rem appB}, one can eventually state Theorem \ref{thm esistenza ext}.

\subsection{Local existence and transversality of the inner arcs}\label{ssec: localInner app}

Now we turn to the inner Problem (\ref{inprob}).
We propose the proof of Proposition \ref{lem levi civita}, along with the preliminary results  which allows to prove Theorem \ref{thm esistenza int}.
\begin{proof}[Proof of Proposition \ref{lem levi civita}]
	Denoted with $r(s)=\vert z(s)\vert$, consider the reparametrisation $s=s(\tilde{\tau})$ such that
	\begin{equation*}
		\frac{d}{ds}=\frac{1}{r(s(\tilde{\tau}))}\frac{d}{d\tilde{\tau}}\Rightarrow\frac{d^2}{ds^2}=-\frac{1}{r(s(\tilde{\tau}))^3}\frac{d}{d\tilde{\tau}}+\frac{1}{r(s(\tilde{\tau}))^2}\frac{d^2}{d\tilde{\tau}^2}. 
	\end{equation*}
	Denoting, with an abuse of notation, $'=d/d\tilde{\tau}$, the first and second equations in (\ref{inprob}) become
	\begin{equation*}
		r(\tilde{\tau})z''(\tilde{\tau})-r'(\tilde{\tau})z'(\tilde{\tau})+\mu z(\tilde{\tau})=0, \quad \frac{1}{2r(\tilde{\tau})^2}\vert z'(\tilde{\tau})\vert^2-\mathcal{E}-h-\frac{\mu}{r(\tilde{\tau})}=0.
	\end{equation*} 
	Identifying now $\mathbb{R}^2$ and $\mathbb{C}$, let us consider a new spatial coordinate $w\in\mathbb{C}$ such that $z(\tilde{\tau})=w^2(\tilde{\tau})$: we have then 
	\begin{equation*}
		\begin{split}
			&2r(\tilde{\tau})w(\tilde{\tau})w''(\tilde{\tau})-w(\tilde{\tau})^2(\vert w'(\tilde{\tau})\vert^2-\mu)=0, \quad 2\vert w'(\tilde{\tau})\vert^2-\mu=r(\tilde{\tau})(\mathcal{E}+h)\\
			&\Rightarrow 2w''(\tilde{\tau})=w(\tilde{\tau})(\mathcal{E}+h).
		\end{split}
	\end{equation*}
	Finally, considering the new time variable $\tau=\tilde{\tau}/2$ (again, with an abuse of notation, $'=d/d\tau$), one obtains the final Cauchy problem
	\begin{equation*}
		\begin{cases}
			w''(\tau)=2(\mathcal{E}+h)w(\tau), &\tau\in[-T,T]\\
			\frac{1}{2}\vert w'(\tau)\vert^2-(\mathcal{E}+h)\vert w(\tau)\vert^2=\mu &\tau\in[-T,T]\\
			w(-T)=w_0, w'(-T)=\dot{w}_0
			
		\end{cases}
		=
		\begin{cases}
			w''(\tau)=\Omega^2w(\tau), &\tau\in[-T,T]\\
			\frac{1}{2}\vert w'(\tau)\vert^2-\frac{\Omega^2}{2}\vert w(\tau)\vert^2=E &\tau\in[-T,T],\\
			w(-T)=w_0, w'(-T)=\dot{w}_0
		\end{cases}
	\end{equation*}
	for some $T>0$ and suitable initial conditions $w_0,\dot{w}_0$ (we will return to the determination of $w_{0}$ and $\dot{w}_0$ in Proposition \ref{InizCondTras app}). The solutions of (\ref{inprob}) can be then seen, in a suitable parametrisation, as complex squares of solutions of an harmonic repulsor with fixed ends boundary conditions, energy equal to $E=\mu$ and frequency $\Omega=\sqrt{2(\mathcal{E}+h)}$.	
\end{proof}

As for the outer problem, suppose that $\partial D$ is a closed curve of class $C^2$ parametrised by $\gamma(\xi): I\rightarrow \mathbb{R}^2$: passing to the Levi-Civita plane, $\gamma$ is transformed according to the same rule $w^2=z$. The existence in the physical plane of a point $z_1$ for wich the inner arc $z(s)$ encounters  the boundary $\partial D$ again translates, in the Levi-Civita plane, in the existence of a point $w_1$ which encounters the transformed of $\gamma$.   As the complex square determines a double covering of $\mathbb{C}$, it is clear that every arc $z(\tau)$ in the physical plane corresponds to two arcs $w(\tau)$ in the Levi-Civita plane, depending on the choice of $w_0$, which is such that $w_0^2=z_0^I$, and a suitable transformed velocity $\dot{w}_0$. In the following, we will work with the Levi-Civita variables, taking respectively for $w_0$ the negative determination of the square root of $z_0$ and for $w_1$ the positive determination of the square root of $z_1$, namely, in polar coordinates, 
\begin{equation}\label{determinazioni app}
	z_0=\vert z_0\vert e^{i\theta_0}\Rightarrow w_0=-\sqrt{\vert z_0\vert}e^{i\frac{\theta_0}{2}},\quad z_1=\vert z_1\vert e^{i\theta_1}\Rightarrow w_1=\sqrt{\vert z_1\vert}e^{i\frac{\theta_1}{2}}. 
\end{equation}
The transformed boundary follows the same rules, and is defined in two neighborhoods of $w_0$ and $w_1$. More precisely, let us suppose that $\bx$ satisfies condition (\ref{cond Hom}) and, additionally, $\gamma(\bar{\xi})$ points in the direction of $e_1=(1,0)$: for the sake of simplicity, we will focus on this particular value of $\bx$, as for every $\bx'\in I$ satisfying (\ref{cond Hom}) we can consider the rotated basis $(e_1',e_2')$ such that $\bx'$ has the properties of $\bar{\xi}$.  
\begin{defn}
	Defined $\bar{\xi}$ as above, there exists $\bar{\epsilon}>0$ such that, if $\gamma(\xi)$ is expressed in polar coordinates, namely, $\gamma(\xi)=\rho(\xi)e^{i\theta(\xi)}$,  the curves 
	\begin{equation*}
		\begin{aligned}
			&\phi_+(\xi):(\bar{\xi}-\bar{\epsilon},\bar{\xi}+\bar{\epsilon})\rightarrow\mathbb{C}, \quad \phi_+(\xi)=\sqrt{\rho(\xi)}e^{i\theta(\xi)/2},\\
			&\phi_-(\xi):(\bar{\xi}-\bar{\epsilon},\bar{\xi}+\bar{\epsilon})\rightarrow\mathbb{C}, \quad \phi_-(\xi)=-\sqrt{\rho(\xi)}e^{i\theta(\xi)/2}=\sqrt{\rho(\xi)}e^{i(\theta(	\xi)/2+\pi)}
		\end{aligned}
	\end{equation*}
	are well defined in the Levi-Civita plane. 
\end{defn}
As an immediate consequence of the conformality of the map $w\mapsto w^2$ we have the following
\begin{lemma}\label{lemma app}
	The transformed curves $\phi_{\pm}(\xi)$ preserve the angle between the radial and the tangent direction of $\gamma(\xi)$. In particular, if condition \ref{cond Hom} holds for $\gamma{(\xi)}$, then it holds for $\phi_{\pm}(\xi)$ with $\xi\in(\bar{\xi}-\bar{\epsilon},\bar{\xi}+\bar{\epsilon})$, possibly reducing $\bar{\epsilon}$.
\end{lemma}

%\begin{proof}
%	As already observed, it is sufficient to prove the Lemma for $\bar{\xi}$ such that $\gamma(\bar{\xi})\parallel e_1$ and points in the direction of $e_1$. Defining $\rho, \dot{\rho}>0$ and $\theta\in\mathbb{T}$ such that, in complex coordinates,  $\gamma(\bar{\xi})=\rho e^{i\cdot 0}$ and $\dot{\gamma}(\bar{\xi})=\dot{\rho}e^{i\theta}$, we have 
%	\begin{equation}
%		\phi_+(\bar{\xi})=\sqrt{\rho}e^{i\cdot 0}, \quad \phi_-(\bar{\xi})=\sqrt{\rho}e^{i\pi}. 
%	\end{equation}
%	We can express, for some $\theta_-,\theta_+\in\mathbb{T}$ and $\lambda_+,\lambda_-\in\mathbb{R}$, 
%	\begin{equation}
%		\dot{\phi}_+(\bar{\xi})=\lambda_+e^{i\theta_+}, \quad \dot{\phi}_-(\bar{\xi})=\lambda_-e^{i\theta_-},
%	\end{equation}
%	and, using the relation $\phi_{\pm}^2(\xi)=\gamma(\xi)\Rightarrow 2\phi_{\pm}(\xi)\dot{\phi}_{\pm}(\xi)=\dot{\gamma}(\xi)$, we have 
%	\begin{equation}
%		\theta_+=\theta, \theta_-=\theta+\pi, \lambda_{\pm}=\frac{\dot{\rho}}{2\sqrt{\rho}}, 
%	\end{equation}
%	
%	then the angles 	$\widehat{\gamma(\bar{\xi}), \dot{\gamma}(\bar{\xi})}$, 	$\widehat{\phi_+(\bar{\xi}), \dot{\phi}_+(\bar{\xi})}$ and 	$\widehat{\phi_-(\bar{\xi}), \dot{\phi}_-(\bar{\xi})}$ are equal.
%\end{proof}
Let us focus on the transformed arc in the Levi-Civita plane: the next Proposition states the existence, under suitable hypotheses on the initial conditions, of a solution of Problem \ref{LCprob} which has the desired transversality properties. 
\begin{prop}\label{prop1 app}
	If condition \ref{cond Hom} holds for $\gamma(\bar{\xi})$, then there are $\tilde{\lambda}>0,0<\tilde{\epsilon}<\bar{\epsilon}$ such that, for every $\xi_0\in [\bar{\xi}-\tilde{\epsilon},\bar{\xi}+\tilde{\epsilon}]$, $\dot{\theta}_0\in[-\tilde{\lambda},\tilde{\lambda}]$ there are $T>0, \xi_1\in I$ such that the Cauchy problem
	\begin{equation*}
		\begin{cases}
		(\Prob_{LC})[w(\tau)] &\tau\in[0,T]\\
			w(0)=\phi_-(\xi_0), w'(0)=\dot{r}_0 e_1+\dot{\theta}_0 e_2
		\end{cases}
	\end{equation*} 
	with $\dot{r}_0=\sqrt{2E+\Omega^2|\phi_-(\xi_0)|^2-\dot{\theta}_0^2}$ admits the unique solution $w(\tau; \xi_0, \dot{\theta}_0)$. Moreover, $w(T;\xi_0,\dot{\theta}_0)=\phi_+(\xi_1)$. In addition, $w'(T(\xi_0, \dot{\theta}_0); \xi_0,\dot{\theta}_0)\nparallel\dot{\phi}_+(\xi_1(\xi_0,\dot{\theta}_0))$, namely, the arc is not tangent to $\partial D$.  
\end{prop}

The proof is again rather standard and relies on a transversality argument for the regularised flow. Moreover, continuity of the regularised flow with respect to the initial conditions and angle preserving of the complex square map entail the desired transversality property. 

Let us notice that the smallness condition on the velocity's orthogonal component $\dot{\theta}_0$ can be given also in terms of the angle between the radial direction and the initial velocity vector. 
As in the outer case, we can conseder the angle $\alpha\in[-\pi/2, \pi/2]$ between  $w'(0)=\dot{r}_0e_1+\dot{\theta}_0e_2$ and $\phi_-(\bar{\xi})$ and have 
\begin{equation*}
	\tan{\alpha}=\frac{\dot{\theta}_0}{\dot{r}_0}=\frac{\dot{\theta}_0}{\sqrt{2E+\Omega^2\vert\phi_-(\x0)\vert^2-\dot{\theta}_0^2}} \Leftrightarrow \dot{\theta}_0=\tan{\alpha}\sqrt{\frac{2E+\Omega^2\vert\phi_-(\x0)\vert^2}{1+\tan{\alpha}^2}}=f(\alpha,\x0).
\end{equation*}
As $f(\alpha,\x0)$ is continous and $f(0,\bx)=0$,  there exist $\lambda_{\alpha}>0$ and $\lambda_{\x0}>0$ such that $\lambda_{\x0}<\epsilon$ and, if $\vert\alpha\vert<\lambda_{\alpha}$ and $|\bx-\x0|<\lambda_{\x0}$, then $\vert\dot{\theta}_0\vert<\lambda$.
%\end{oss}
\begin{prop}\label{InizCondTras app}
	Let us consider $\bar{\xi}\in I$ such that $\gamma(\bar{\xi})=\rho e_1$, $\rho>0$, and suppose that condition \ref{cond Hom} holds. Let $\bar{\epsilon}>0$ such that the curves $\phi_{\pm}: [-\bar{\epsilon}+\bar{\xi},\bar{\epsilon}+\bar{\xi}]$ are well defined, and choose $\xi_0\in [-\bar{\epsilon}+\bar{\xi},\bar{\epsilon}+\bar{\xi}]$ and $\beta\in[-\frac{\pi}{2}, \frac{\pi}{2}]$. Then the system (in polar coordinates)
	\begin{equation*}
		\begin{cases}
			(\Prob_I)[z(s)], \quad s\in[0,S]\\
			z(0)=\gamma(\xi_0)=\rho(\xi_0) e^{i\theta(\xi_0)}, z'(0)=\sqrt{2}\sqrt{\mathcal{E}+h+\frac{\mu}{\rho(\xi_0)}}e^{i(\theta(\xi_0)+\beta)}
		\end{cases}
	\end{equation*}
	is conjugated, in the Levi-Civita plane, and considering $\tau=\tau(s)$ the Levi-Civita time, to the problem
	\begin{equation*}
		\begin{cases}
		(\Prob_{LC})[w(\tau)], \quad \tau\in[0,T]\\
			w(0)=-\sqrt{\rho(\xi_0)}e^{i\theta(\xi_0)/2}, w'(0)=-\sqrt{2E+\Omega^2\rho(\xi_0)}e^{i(\theta(\xi_0)/2+\beta)}
		\end{cases}
	\end{equation*}
	for a suitable $T>0$. In other words, the angles between the original initial conditions and the trasformed ones, namely, $\widehat{z(0),\dot{z}(0)}$ and $\widehat{w(0), w'(0)}$, are equal. 
\end{prop}
\begin{proof}
	From the definition of $\phi_-$, we have that $w(0)=\phi_-(\xi_0)=-\sqrt{\rho(\xi_0)}e^{i(\theta(\xi_0))}=\sqrt{\rho(\xi_0)}e^{i(\theta(\xi_0)+\pi)}$. To compute $w'(0)=dw(0)/d\tau$, we go through the following Levi-Civita transformations: 
	\begin{itemize}
		\item $\frac{d}{ds}=\frac{1}{|z(\tilde{\tau}(s))|}\frac{d}{d\tilde{\tau}}$, then, for $s=0$,  
		\begin{equation*}
			\sqrt{2}\sqrt{\mathcal{E}+h+\frac{\mu}{\rho(\xi_0)}}e^{i\theta(\xi_0)}=\frac{d}{ds}z(0)=\frac{1}{\rho(\xi_0)}\frac{d}{d\tilde{\tau}}z(0); 
		\end{equation*}
		\item $z=w^2$, then 
		\begin{equation*}
			\sqrt{2}\rho(\xi_0)\sqrt{\mathcal{E}+h+\frac{\mu}{\rho(\xi_0)}}e^{i\theta(\xi_0)}=\frac{d}{d\tilde{\tau}}z(0)=2w(0)\frac{d}{d\tilde{\tau}}w(0); 
		\end{equation*}
		\item $\tau=\tilde{\tau}/2$, then 
		\begin{equation*}
			\sqrt{2}\sqrt{\mathcal{E}+h+\frac{\mu}{\rho(\xi_0)}}e^{i(\theta(\x0)+\beta)}=w(0)\frac{d}{d\tau}w(0)\Rightarrow w'(0)=-\sqrt{2E+\Omega^2\rho(\xi_0)}e^{i(\theta(\xi_0)/2+\beta)}
		\end{equation*}		
	\end{itemize} 
\end{proof}
The angle between the initial conditions in the physical plane is then preserved after the passage in the Levi-Civita reference frame: this assures that, if we fix some "smallness" condition on the angle $\alpha$ between the initial velocity and the direction of $\gamma(\bx)$ in the original reference frame, they will hold also in the Levi-Civita plane. This allows, along with the previous results, to state Theorem \ref{thm esistenza int}.

\bibliography{articolo_ellissi}
\bibliographystyle{acm}
\end{document}